\documentclass[12pt,a4paper,reqno]{amsart}
\usepackage[utf8]{inputenc}
\usepackage[T1]{fontenc}
\usepackage{lmodern}
\usepackage{microtype}
\usepackage[english]{babel}
\usepackage{csquotes}
\usepackage[DIV=11,twoside=off]{typearea}
\usepackage{amsmath,amssymb,amsthm,mathtools}
\usepackage{enumitem}
\usepackage{mathrsfs}
\usepackage[colorlinks=true,linkcolor=blue,citecolor=blue,urlcolor=blue]{hyperref}
\usepackage[backend=biber, style=ams-alphabetic, giveninits=true,texencoding=ascii,isbn=false]{biblatex}

\newtheorem{theorem}{Theorem}[section]
\newtheorem{proposition}[theorem]{Proposition}
\newtheorem{lemma}[theorem]{Lemma}
\newtheorem{corollary}[theorem]{Corollary}
\newtheorem{definition}[theorem]{Definition}
\newtheorem{introtheorem}{Theorem}

\theoremstyle{remark}
\newtheorem{remark}{Remark}[section]
\newtheorem{example}[theorem]{Example}

\newcommand{\R}{\mathbb{R}}
\newcommand{\F}{\mathcal{F}}
\newcommand{\C}{\mathbb{C}}
\newcommand{\N}{\mathbb{N}}
\newcommand{\Q}{\mathbb{Q}}
\newcommand{\Z}{\mathbb{Z}}

\newcommand{\D}{\mathcal{D}}
\newcommand{\E}{\mathcal{E}}
\newcommand{\I}{\mathcal{I}}

\newcommand{\WF}{\operatorname{WF}}
\newcommand{\supp}{\operatorname{supp}}
\newcommand{\dotT}{\dot{T}^{*}}
\newcommand{\eps}{\varepsilon}

\newcommand{\cA}{\mathcal{A}}
\renewcommand{\S}{\mathcal{S}}
\newcommand{\Cinft}{C^\infty}
\newcommand{\CT}{C^\infty_c}
\newcommand{\bq}{\begin{equation}}
\newcommand{\eq}{\end{equation}}
\newcommand{\bqn}{\begin{equation*}}
\newcommand{\eqn}{\end{equation*}}
\newcommand{\eklm}[1]{\left\langle #1 \right\rangle}
\newcommand{\norm}[1]{\left\Vert #1 \right\Vert}
\renewcommand{\Re}{\operatorname{Re}}
\renewcommand{\Im}{\operatorname{Im}}

\title{Distributional zero divisors with full support}
\author{Benjamin Delarue}
\email{bdelarue@math.upb.de}
\address{Universität Paderborn, Warburger Str.~100, 33098 Paderborn, Germany}

\begin{document}

\begin{abstract}We show that every smooth manifold $M$ of positive dimension admits fully supported distributions $u,v\in \D'(M)$ with $\WF(u)\cap({-}\WF(v))=\emptyset$ and $uv=0$. Neither factor can be continuous on any non-empty open set, but $u$ may be chosen of optimal Sobolev order $\dim M/2$,  with wavefront set contained in any prescribed closed conical set $\Gamma\subset T^\ast M\setminus 0$ with non-empty and non-maximal fibers. Simultaneously, the wavefront set of $v$ may be confined to any compatible open conical set or, in geometric situations, to  closed conical sets such as conormal bundles of foliations or ray bundles generated by locally conformally closed one-forms. We give positivity criteria and show that full-support zero divisor pairs accumulate at $(1,0)$ in Sobolev-microlocal topologies. The one-dimensional case uses a Fourier null series of Kozma-Olevski\u{\i}. 
\end{abstract}

\maketitle

\section{Introduction}

Let $M$ be a non-empty smooth manifold equipped with a smooth measure, such as an open set in Euclidean space with Lebesgue measure. Distributions  generalize locally integrable complex-valued functions on $M$, but two distributions $u,v\in \D'(M)$ cannot in general be multiplied \cite{Schwartz1954}. If they satisfy the wavefront set condition
\bq
\WF(u)\cap (-\WF(v))=\emptyset,\label{eq:Hormandercond}
\eq
one defines $uv:=\iota^\ast (u\otimes v)$, where $\iota:M\to M\times M$ is the diagonal map. This distributional (Hörmander) product generalizes the usual product of smooth functions and satisfies $\supp(uv)\subset \supp u\cap \supp v$  \cite[Thm.~8.2.10]{HormanderI}. We study the structural question whether two distributions with full support can multiply to zero:
\bq
\WF(u)\cap (-\WF(v))=\emptyset,\qquad\supp u=\supp v=M\implies uv\neq 0\quad ?\label{eq:mainquestion}
\eq
This is subtle because full supports rule out an elementary zero divisor construction through disjoint supports, and it is not clear a priori whether the  condition \eqref{eq:Hormandercond} makes the Hörmander product behave more like the usual product of continuous functions or of discontinuous $L^1_\mathrm{loc}$-functions: If $f,g\in C(M)$ have full support, then also $fg$ has full support. By contrast, if $\dim M>0$, one can decompose $M=A\sqcup B$ into measurable sets, each intersecting every non-empty open subset of $M$ in a set of positive measure (cf.\ \cite{Erdos}). The characteristic functions $1_A,1_B$ are discontinuous on every non-empty open set, have  full support, and $1_A1_B=0$ while  \eqref{eq:Hormandercond} is violated:  $1_A+1_B=1$ and $1_A,1_B$ are real, so $\WF(1_A)=\WF(1_B)=-\WF(1_A)\neq\emptyset$. Thus  \eqref{eq:Hormandercond} is important in \eqref{eq:mainquestion}. 

\begin{introtheorem}\label{thm:intro1}
If $\dim M>0$, there exist distributions  $u,v\in \D'(M)$ with 
\bq
\mathrm{WF}(u)\cap (-\mathrm{WF}(v))=\emptyset,\qquad \supp u=\supp v=M,\qquad uv=0.\label{eq:threecond}
\eq
If $\dim M\geq 2$, $u$ and $v$ may be chosen positive. On the other hand, if $u,v\in \D'(M)$  satisfy \eqref{eq:threecond}, then neither $u$ nor $v$ is a continuous function on any non-empty open set.
\end{introtheorem}
Recall that a distribution $u\in\D'(M)$ is  \emph{positive} if $u(\varphi)\geq 0$ for all real-valued $\varphi\in \CT(M)$ with $\varphi\geq 0$. Then $u$ is a measure \cite[Thm.~2.1.7]{HormanderI}. 

To the best of the author’s knowledge, no pair of distributions satisfying \eqref{eq:threecond} has previously been recorded in the literature. 
Theorem \ref{thm:intro1} and its refinements below show, however, that such pairs are not exceptional: They exist on every positive-dimensional
smooth manifold (Theorem \ref{thm:intro1}), at optimal Sobolev regularity of one factor (Theorem \ref{thm:intro2a}), under simultaneous wavefront-set restrictions  (Theorems \ref{thm:intro2} and \ref{thm:intro4}), and arbitrarily close to $(1,0)$ in the relevant Sobolev-microlocal topologies (Theorem \ref{thm:intro3}). 

 The elimination of continuity in Theorem \ref{thm:intro1} is non-trivial: Even if $u\in C(M)$ is nowhere-vanishing, given $v\in \D'(M)$ with $\WF(u)\cap (-\WF(v))= \emptyset$ we cannot in general write $v=\frac{1}{u} uv$  because we do not know if covectors in $\WF(\frac{1}{u})$, $\WF(u)$ and $\WF(v)$ might add to zero in some fiber. However, it is trivial that \eqref{eq:threecond} cannot occur with \emph{both} factors $u,v$ continuous; then the distributional product coincides with the usual product, giving $\supp uv=M$.

Since Theorem \ref{thm:intro1} holds for every smooth manifold of dimension $>0$, the existence assertion in Theorem \ref{thm:intro1} can be equivalently formulated with \eqref{eq:threecond} replaced by \eqref{eq:threecond1} or \eqref{eq:threecond2}:
\begin{align}
\mathrm{WF}(u)\cap (-\mathrm{WF}(v))&=\emptyset,\qquad \supp u=\supp v=M,\qquad \supp uv\neq M.\label{eq:threecond1}\\
\mathrm{WF}(u)\cap (-\mathrm{WF}(v))&=\emptyset,\qquad \mathrm{Int}(\supp u)\cap \mathrm{Int}(\supp v)\not\subset \mathrm{Int}(\supp uv).\label{eq:threecond2}
\end{align}
Here $\mathrm{Int}$ denotes the interior in $M$. The equivalence is established by replacing $M$ with $M\setminus \supp uv$  or $\mathrm{Int}(\supp u)\cap \mathrm{Int}(\supp v)\setminus \supp uv$  and using that the full support property and the wavefront set condition are compatible with restriction to open sets. 

In dimension one, say for $M=S^1$, Theorem \ref{thm:intro1} admits an additional algebraic interpretation because the condition \eqref{eq:Hormandercond} is exceptionally rigid. Namely, we obtain zero divisor pairs in convolution algebras of Fourier coefficients that are not due to disjoint supports of the corresponding distributions. For more details, see Section \ref{sec:relatedconsequ}.

\medskip To proceed to more detailed results, we call a pair of distributions $u,v\in \D'(M)$ as in \eqref{eq:threecond} a \emph{full-support zero divisor pair} and use the term \emph{full-support zero divisor} for a single distribution $u\in \D'(M)$ that is a member of some full-support zero divisor pair.  Knowing from Theorem \ref{thm:intro1} that such pairs exist and that no full-support zero divisor can be continuous on any non-empty open set, a natural subsequent question is how regular a full-support zero divisor may be: In which Sobolev regularity classes can it lie, and how large does its wavefront set have to be? There are clear a priori limits: \begin{itemize}[leftmargin=*]
\item A full-support zero divisor can be only of Sobolev orders $s\leq\dim M/2$ because distributions of Sobolev order $s>\dim M/2$ are continuous by  Sobolev embedding.\smallskip

\item The wavefront set of a full-support zero divisor $u\in \D'(M)$ must have a non-empty and non-maximal fiber over every point: $\WF_p(u)\neq\emptyset$ and $\WF_p(u)\neq T_p^\ast M\setminus \{0\}$ for all $p\in M$. This is because $\WF_p(u)=\emptyset$ at some $p$ means that $u$ is smooth on some open set $U\subset M$ around $p$, in particular $u|_U$ is continuous, contradicting Theorem \ref{thm:intro1}. Dually, $\WF_p(u)= T_p^\ast M\setminus \{0\}$ at some $p$ implies that every $v \in \D'(M)$ satisfying $\mathrm{WF}(u)\cap (-\mathrm{WF}(v))=\emptyset$ is smooth near $p$, which again contradicts Theorem \ref{thm:intro1}. 

If in addition we want $u$ to be a real distribution, then one necessarily has $\WF_p(u)\cap (-\WF_p(u))\neq \emptyset$ for all $p\in M$, since $\WF(\bar u)=-\WF(u)$.
\end{itemize}
This motivates calling a closed conical set $\Gamma\subset  T^\ast M\setminus 0$ \emph{admissible} if
\[
\emptyset\neq\Gamma_p\neq T_p^\ast M\setminus \{0\}\quad \forall\; p\in M.
\]
Asking if a full-support zero divisor $u$ can be of Sobolev orders close to the maximal order $\dim M/2$ can be interpreted as measuring quantitatively ``how much  discontinuity'' the full-support zero divisor property requires. Also the weaker question if $u$ can be of Sobolev order $s\geq 0$ suggests itself: This is the threshold for $u$ being a regular distribution, since $H^s_\mathrm{loc}(M)\subset L^1_\mathrm{loc}(M)$ if, and only if, $s\geq 0$. Concerning the wavefront set, one may wonder in dimensions $>1$ whether the full-support zero divisor property of $u$  requires a certain ``room'' in $\WF(u)$ beyond the mere non-emptiness of the fibers. 

Our next result answers these questions: The maximal Sobolev order $s=\dim M/2$ can always be achieved. Simultaneously, the wavefront set can be confined to an arbitrary admissible closed conical set. Moreover, the distribution can be chosen positive whenever the conical set satisfies the necessary condition to contain the wavefront set of a nowhere-smooth real distribution.

\begin{introtheorem}\label{thm:intro2a}
For every smooth manifold $M$ of dimension $>0$ and every admissible closed conical set $\Gamma\subset  T^\ast M\setminus 0$, there exists a full-support zero divisor $u\in \D'_\Gamma(M)\cap H^{\dim M/2}_\mathrm{loc}(M)$. If $\Gamma_p\cap (-\Gamma_p)\neq \emptyset$ for all $p\in M$, we can achieve that $u$ is positive.
\end{introtheorem}
Here we write $\D'_\Gamma(M):=\{u\in \D'(M)\,|\, \WF(u)\subset \Gamma\}$ for $\Gamma\subset T^\ast M\setminus 0$ conical. 
Theorem \ref{thm:intro2a} is an asymmetric result: Although a full-support zero divisor $u$  by definition has a companion $v$ with which it forms a full-support zero divisor pair, Theorem \ref{thm:intro2a} makes no statement about regularity or wavefront set properties of any such $v$. We proceed with a strengthening of Theorem \ref{thm:intro2a} addressing both members $u,v$ of a full-support zero divisor pair.  It says that we can achieve a ``thick/thin'' situation: The thick factor $u$ is of maximal Sobolev order and may have an arbitrarily confined wavefront set as in Theorem \ref{thm:intro2a}, while the thin factor $v$ is more singular.

In the statement below, we use the definition $H^{s-0}_\mathrm{loc}(M):=\bigcap_{\eps>0}H^{s-\eps}_\mathrm{loc}(M)$,  $s\in \R$. 

\begin{introtheorem}\label{thm:intro2}
Let $M$ be a smooth manifold and let $\Gamma,\Upsilon\subset T^\ast M\setminus 0$ be two conical sets such that $\Gamma$ is closed, $\Upsilon$ is open, $\Gamma_p\neq \emptyset$ and $\Upsilon_p\neq \emptyset$ for all $p\in M$, and $\Gamma\cap(-\Upsilon)=\emptyset$. Then the following holds:

\begin{enumerate}
\item There exist a closed conical set $\Lambda\subset \Upsilon$ and distributions   
 \[
 u\in \D'_\Gamma(M)\cap H^{\dim M/2}_\mathrm{loc}(M),\qquad v\in \D'_\Lambda(M)\cap H^{-1/2-0}_\mathrm{loc}(M)
 \]
   with 
\[
\supp u=\supp v=M,\qquad uv=0.\vspace*{0.5em}
\]
\item If $\Gamma_p\cap (-\Gamma_p)\neq \emptyset$ for all $p\in M$, then in (1) we can achieve that $u$ is positive. 
\item If $\Upsilon_p\cap (-\Upsilon_p)\neq \emptyset$ for all $p\in M$, then in (1) we can achieve that $v$ is positive and $\Lambda=-\Lambda$.
\item If $\Gamma_p\cap (-\Gamma_p)\neq \emptyset$ and $\Upsilon_p\cap (-\Upsilon_p)\neq \emptyset$ for all $p\in M$, then in (1) we can achieve that $u$ and $v$ are positive and $\Lambda=-\Lambda$.
\end{enumerate}
\end{introtheorem}
For closed conical sets $\Gamma,\Lambda\subset  T^\ast M\setminus 0$ with $\Gamma\cap(-\Lambda)=\emptyset$ and $k\geq 0$, consider the set of full-support zero divisor pairs
\begin{align*}
\mathscr{F}_{\Gamma,\Lambda;k}(M):=\big\{&(u,v)\in (\D'_\Gamma(M)\cap H^{\dim M/2}_\mathrm{loc}(M)) \times (\D'_\Lambda(M)\cap H^{-k/2-0}_\mathrm{loc}(M))\,|\,\\ &\supp u=\supp v=M,\; uv=0\big\}.
\end{align*}
To describe the shape of such a set, for every $\eps_0>0$, we can equip the ambient space $(\D'_\Gamma(M)\cap H^{\dim M/2}_\mathrm{loc}(M)) \times (\D'_\Lambda(M)\cap H^{-k/2-0}_\mathrm{loc}(M))$ with the topology generated by the family of seminorms $p_\Gamma(u,v):=q_\Gamma(u)$ and $p_\Lambda(u,v):=q_\Lambda(v)$, where $q_\Gamma \in \mathsf{SEM}(\Gamma,\dim M/2)$ and $q_\Lambda\in\mathsf{SEM}(\Lambda,-k/2-\eps_0)$ are the Sobolev and microlocal seminorms defined in Section \ref{sec:summary}. Then the topology depends on $\eps_0$, while the space  does not. Let us denote by $\overline{A}^{\eps_0}$ the closure of a set $A$ with respect to this topology.
\begin{introtheorem}\label{thm:intro3}
Fix $N\in \{1,2,3,4\}$ and assume the hypotheses of Theorem \ref{thm:intro2}(N). The cone $\Lambda$ in that statement can be chosen such that for every $\eps_0>0$ the following holds:  \begin{enumerate}
\item If $N=1$, then $\overline{\mathscr{F}_{\Gamma,\Lambda;1}(M)}^{\eps_0}$ contains all pairs $(f,0)$, where $f\in \Cinft(M)$ is such that $f^{-1}(\{0\})\subset M$ has empty interior. 

\item  If $N=2$, then   $\overline{\{(u,v)\in\mathscr{F}_{\Gamma,\Lambda;1}(M)\,|\, u\text{\emph{ positive}}\}}^{\eps_0}$ contains all pairs $(f,0)$, where $f\in \Cinft(M)$ is such that $f^{-1}(\{0\})\subset M$ has empty interior and $f\geq 0$. 

\item If $N=3$, then  $\overline{\{(u,v)\in\mathscr{F}_{\Gamma,\Lambda;1}(M)\,|\, v\text{\emph{ positive}}\}}^{\eps_0}$ contains all pairs $(f,0)$, where $f\in \Cinft(M)$ is such that $f^{-1}(\{0\})\subset M$ has empty interior. 

\item  If $N=4$, then  $\overline{\{(u,v)\in\mathscr{F}_{\Gamma,\Lambda;1}(M)\,|\, u,v\text{\emph{ positive}}\}}^{\eps_0}$ contains all pairs $(f,0)$, where $f\in \Cinft(M)$ is such that $f^{-1}(\{0\})\subset M$ has empty interior and $f\geq 0$. 
\end{enumerate}
\end{introtheorem}
Thus, the thick factors of full-support zero divisors accumulate at every smooth function with zero set of empty interior (such as the constant function $f=1$), whereas the thin factors accumulate at $0$. In particular,  in the situation of Theorem \ref{thm:intro3}, injectivity is not a stable property of the multiplication maps 
\begin{align*}
m_{u}:\D'_\Lambda(M)\cap H^{-1/2-0}_\mathrm{loc}(M)&\to \D'(M),\qquad u\in \D'_\Gamma(M)\cap H^{\dim M/2}_\mathrm{loc}(M),\\
v&\mapsto uv.
\end{align*}
Indeed: While $m_1$ is injective, every neighborhood of $1$ in $\D'_\Gamma(M)\cap H^{\dim M/2}_\mathrm{loc}(M)$  with respect to the topology generated by the seminorm family $\mathsf{SEM}(\Gamma,\dim M/2)$ contains a distribution $u$ such that $m_{u}$ is not injective because it annihilates some $v$ with full support that can be chosen arbitrarily close to $0$ in the topologies generated by the seminorm families $\mathsf{SEM}(\Lambda,-1/2-\eps_0)$ for $\eps_0>0$.

\medskip

Finally, let us state a result that allows to confine the wavefront sets of \emph{both} factors in prescribed \emph{closed} conical sets. The following theorem summarizes results from Section \ref{sec:examples}, to which we refer the reader for more detailed statements, e.g.\ on positivity.

\begin{introtheorem}\label{thm:intro4}
Let $M$ be a smooth manifold and let $\Gamma,\Lambda\subset  T^\ast M\setminus 0$ be closed conical sets projecting onto all of $M$ and satisfying $\Gamma\cap(-\Lambda)=\emptyset$.  If  $\Lambda$ satisfies one of the following conditions, then the statements of Theorem \ref{thm:intro2}(1) and Theorem \ref{thm:intro3}(1)  hold with the given $\Gamma$ and $\Lambda$, after replacing $\mathscr{F}_{\Gamma,\Lambda;1}$ by  $\mathscr{F}_{\Gamma,\Lambda;k}$  with $k\geq 1$ as indicated below:
\begin{enumerate}
\item $\Lambda=N^\ast \F\setminus 0$ is the punctured conormal bundle of a continuous foliation of $M$ with smooth leaves of positive dimension and codimension $k$ (Definition \ref{def:foliation}). 
\item $\Lambda=\Lambda^\ast \mathcal Z$ is the conormal cone of a more general family $\mathcal Z$ (called \emph{fragmentation}, see Definition \ref{def:fragmentation}) of submanifolds of $M$ of positive codimension which may overlap and have different dimensions, with maximal codimension $k=k_\mathrm{max}$. 
\item $\Lambda$ contains the ray bundle $\{(p,r\alpha(p))\,|\,p\in M,r>0\}$ of a smooth nowhere-vanishing locally conformally closed $1$-form $\alpha$ on $M$. Here we can keep $k=1$.
\end{enumerate}
\end{introtheorem}
Theorem \ref{thm:intro4}(1) applies to pairs of conormal bundles of stable/unstable foliations of Anosov diffeomorphisms and  weak stable/unstable foliations $\mathcal{W}^{\mathrm{cs}},\mathcal{W}^{\mathrm{cu}}$ of Anosov flows. For instance, as explained in Example \ref{ex:Anosov} and Corollary \ref{cor:foliations1}, if $M$ is compact and carries an Anosov flow, then there exist positive distributions 
\bq
u\in\D'_{N^\ast \mathcal{W}^{\mathrm{cs}}\setminus 0}(M)\cap H^{\dim M/2}_\mathrm{loc}(M),\qquad v\in\D'_{N^\ast \mathcal{W}^{\mathrm{cu}}\setminus 0}(M)\cap H^{-k/2-0}_\mathrm{loc}(M),\label{eq:Anosovintro}
\eq
 where $k$ is the codimension of the weak unstable leaves, such that
\[
\supp  u=\supp  v=M, \qquad uv=0.
\]
An analogous result holds with weak stable/unstable foliations swapped. Here $u$ and $v$ are not claimed to be (co-)resonant states of the Anosov flow. Such additional properties can rule out the full-support zero divisor property. For example, on the unit sphere bundle of a compact Riemannian locally symmetric space of rank one, products of first-band resonant and co-resonant states which are dual to each other in the sense of \cite[Cor.~6.2]{GHW21} are non-zero by the pairing formula \cite[Thm.~6.1]{GHW21}.

 A notable consequence of Theorem \ref{thm:intro4}(3) is that, in any dimension $n>0$, there is a pair $u,v$ of full-support zero divisors such that all fibers of both $\WF(u)$ and $\WF(v)$ consist of one-dimensional rays: Take in Theorem \ref{thm:intro4} a non-empty open subset of $\R^n$ as the manifold $M$ and $\Gamma=\Lambda=\R_{>0} dx_1$. On the other hand, the local Frobenius integrability condition in Theorem \ref{thm:intro4}(3) encoded in the assumption that $\alpha$ is locally conformally closed is linked to a global restriction: A given smooth manifold $M$  may not admit a global nowhere-vanishing  locally conformally closed one-form. However, this assumption can be related to a purely local assumption on $\Lambda$, see  Corollary \ref{cor:integrable}.

\medskip

The technical main result of this paper is Theorem \ref{thm:main}. Theorems \ref{thm:intro1}--\ref{thm:intro3} are proved in Section \ref{sec:mainproofs} and Theorem \ref{thm:intro4} is proved in Section \ref{sec:introthmproof}.  

\subsection{Related results and consequences}\label{sec:relatedconsequ}

The Hörmander product from  \cite[Thm.~8.2.10]{HormanderI}, which we use in this paper, is not the only way to multiply distributions. There is a long list of various established criteria for existence and desirable properties of distributional products  \cite{Mikusinski1962,Ambrose1980,Oberguggenberger1986,Oberguggenberger1992}.  More recent works study the normal topology and continuity of operations on spaces with prescribed
wavefront set \cite{DabrowskiBrouder,BrouderDangHelein} or utilize  refined  microlocal concepts such as Sobolev wavefront sets  \cite{DappiaggiRinaldiSclavi2024}. These results concern existence, continuity, and regularity properties of products of distributions. Our question, whether the existing Hörmander product of two fully supported distributions can vanish, is different.

In dimension one, Theorem \ref{thm:intro2} is linked to trigonometric null series as featured in Kozma-Olevski\u\i's work \cite{KozmaOlevskii}  (see Proposition \ref{prop:KOatom}). For a connected one-dimensional manifold (i.e., $M=\R$ or $M=S^1$) the wavefront set condition \eqref{eq:Hormandercond} is very rigid:  $T^\ast M\setminus 0$ has two connected components $T^\ast_\pm M$ and each of the two spaces of distributions
\[
\D'_\pm(M):=\{u\in \D'(M)\,|\, \WF(u)\subset T^\ast_\pm M\}
\] 
forms a ring with respect to distributional multiplication. Theorem \ref{thm:intro2} implies that these rings contain zero divisor pairs  $u,v\in \D'_\pm(M)$ with  $uv=0$, $\supp u=\supp v=M$. When $M=S^1$, the Fourier transform identifies $\D'_\pm(M)$ with the convolution rings
\begin{align*}
\mathcal R_\pm:=\big\{(x_j)_{j\in \Z}\,|\,x_j\in \C, x_j&=\mathcal{O}(|j|^{-\infty})\text{ as }j\to \mp\infty,\\
\exists\, N>0:x_j&=\mathcal{O}(|j|^N)\hspace*{0.675em}\text{ as }j\to \pm\infty\big\}.
\end{align*}
We conclude that there are non-zero elements $x,y\in \mathcal R_\pm$ with  $x\ast y=0$ which are the Fourier coefficient sequences of distributions $u,v\in \D'_\pm(S^1)$ with full support.

Multiplication of distributions with wavefront sets contained in transverse stable/unstable bundles happens naturally in hyperbolic dynamics. Products of Ruelle-Pollicott resonant and co-resonant states define invariant Ruelle densities \cite{GHW21}, describe equilibrium measures \cite{Humbert2025}, and play an important role in perturbative arguments \cite[Prop.~4.6]{CDDP}. In the latter reference, the particular multiplied resonant and co-resonant states $u,v\in \D'(SM)$ on the unit sphere bundle of a closed hyperbolic $3$-manifold $M$ have full support (for simplicity we ignore here that the actual statement of  \cite[Prop.~4.6]{CDDP} involves a product of distributional differential forms rather than scalar distributions as considered in this paper). That the product of these fully supported distributions has again full support is crucial to the proof of \cite[Thm.~4]{CDDP}. The problem is reduced to the analogous problem for the tensor product of the boundary values of the lifts $\tilde u,\tilde v\in \D'(S\mathbb H^3)$, which is easy to solve because the support of a tensor product is the product of the supports of the factors. In general, for topologically transitive Anosov flows, generalized (co-)resonant states have full support  \cite{Weichfullsupport}. Moreover, the restriction of a Ruelle-Pollicott resonant state to every strong unstable leaf is well-defined and has full support \cite[Cor.~1.5]{lefeuvrepotrie}, a property that is much stronger than having full support on the original manifold.  The distributions constructed here in \eqref{eq:Anosovintro} when $M$ carries an Anosov flow are not (generalized) (co-)resonant states. Instead, our results show that full support, positivity, and the transverse  wavefront set condition alone do not imply a non-zero product. Any non-vanishing result for products of (generalized) resonant and co-resonant states must therefore use dynamical, spectral or other special properties of the distributions.

Another area in which products of distributions occur prominently is tensor tomography on surfaces \cite{BohrLefeuvrePaternain}: Here $M=S\Sigma$ is the unit sphere bundle of a closed connected oriented Riemannian surface $\Sigma$, and one considers the generator $X$ of the geodesic flow as a differential operator of order one. It is proved in  \cite[Thm.~1.1]{BohrLefeuvrePaternain} that the fiber-wise holomorphic distributions $u\in \D'(M)$ solving the transport equation $Xu=0$ form a unital algebra. This is done by showing that their wavefront set is contained in a closed conical set 
\[
\mathcal{C}\subset X^\perp:=\{(p,\xi) \in T^\ast M\,|\,\xi(X|_p)=0\}
\]
with $\mathcal{C}\cap (-\mathcal{C})=\emptyset$ \cite[Cor.~4.3]{BohrLefeuvrePaternain}.  When the geodesic flow is Anosov, $\Sigma$ has no conjugate points \cite{Klingenberg,Klingcorr} and for every $p\in M$ the fiber $\mathcal{C}_p$ has non-empty interior in the fiber $X^\perp_p$ \cite[comments below Cor.~4.3]{BohrLefeuvrePaternain}. Fix $p\in M$ and choose a covector $\xi\in \mathcal{C}_p$ in the interior of $X^\perp_p$. On a flow box $U\subset M$ around $p$ there is a smooth function $f:U\to \R$ with $df|_p=\xi$ and $Xf=0$. After shrinking $U$, the map $f$ is a submersion and one has
\[
\{(q,\tau \,df|_q)\,|\,q \in U,\tau>0\}\subset \mathcal C.
\]
Thus $\mathcal C$ is $H^{-1/2-0}$-atomizable by Corollary \ref{cor:integrable}, and by Theorem \ref{thm:main} there is a full-support zero divisor pair
\[
u\in\D'_{\mathcal C}(S\Sigma)\cap H^{3/2}(S\Sigma),\qquad
v\in\D'_{\mathcal C}(S\Sigma)\cap H^{-1/2-0}(S\Sigma)
\]
arbitrarily close to $(1,0)$ in the considered Sobolev and microlocal seminorms. Here neither $u$ nor $v$ is claimed to be fiber-wise holomorphic or to satisfy the transport equation. In contrast, the application of our results shows that absence of full-support zero divisors in the unital algebra considered in \cite[Thm.~1.1]{BohrLefeuvrePaternain} does not follow from the containment of this algebra in $\D'_{\mathcal C}(S\Sigma)$ alone, but only from the combination of this with the transport equation, the fiberwise holomorphy or other special structures.

\subsection{Discussion, open questions and further directions}\label{sec:discussion}

While Theorem \ref{thm:intro2a} is a sharp Sobolev regularity result for a single full-support zero divisor, it remains open how large the sum $r+s$ can get such that a full-support zero divisor pair exists whose members are of Sobolev orders $r$ and $s$, respectively. Theorem \ref{thm:intro2} achieves $r+s=(\dim M-1)/2-\eps$ for all $\eps>0$ with a thick/thin construction. It would be interesting to know if there are other, perhaps  more symmetric, constructions achieving  a greater sum.

Similarly, while  Theorem \ref{thm:intro2a} allows to confine the wavefront set of a single full-support zero divisor $u$ to an arbitrary closed conical subset of $T^\ast M\setminus 0$ with non-empty and non-maximal fibers and Theorems \ref{thm:intro2}--\ref{thm:intro4} allow to simultaneously confine the wavefront set of the companion distribution $v$ to certain conical subsets of $T^\ast M\setminus 0$, it remains open whether one can simultaneously confine both $\WF(u)$ and $\WF(v)$ to arbitrary pairs of closed conical subsets $\Gamma,\Lambda\subset T^\ast M\setminus 0$ with non-empty fibers satisfying $\Gamma\cap(-\Lambda)=\emptyset$. In particular, Corollary \ref{cor:raybundles} leaves open whether there exist full-support zero divisor pairs whose wavefront sets are \emph{both} contained in rays spanned by one-forms that are not locally conformally closed, such as contact one-forms. 

In a converse direction, it would be interesting to identify additional conditions on a pair of distributions $u,v\in \D'(M)$ ruling out that they form a full-support zero divisor pair. We already mentioned the example of the pairing formula from \cite{GHW21} in the context of Ruelle-Pollicott (co-)resonant states. Other possible hypotheses include regularity restrictions beyond classical Sobolev regularity such as anisotropic Sobolev regularity, negative Hölder regularity, or Besov regularity.

\subsection{Proof strategy and structure of the paper}

The proof of the main Theorem \ref{thm:main}, from which the existence results mentioned above are deduced, consists of an asymmetric thick/thin construction  which might be of independent interest. 

Sections \ref{sec:prelim} and \ref{sec:approx} provide the general technical framework for distributions with controlled wavefront set and Sobolev spaces.  Given a closed conical set $\Gamma\subset T^\ast M\setminus 0$ and a Sobolev order $s\in \R$, convergence in $\D'_\Gamma(M)\cap H^s_\mathrm{loc}(M)$ is described by microlocal and Sobolev seminorms.  For the construction of the thick factor, Proposition~\ref{prop:control} dominates any finite family of such seminorms by a \emph{Hilbert control seminorm}. This allows to formulate the key arguments in the proof of the central Proposition \ref{prop:finiteflat} in a Hilbert space. The latter  proposition says that, for a chosen closed conical set $\Gamma\subset T^\ast M\setminus 0$, near every point $p\in M$ with $\Gamma_p\neq\emptyset$, a smooth function $m\in \Cinft(M)$ vanishing to infinite order on an arbitrary sufficiently small compact set $K\subset M$ can be made as close to $1$ as desired in the relevant microlocal and Sobolev seminorms for every Sobolev order $s\leq\dim M/2$.  By the elementary but important Lemma~\ref{lem:flatkills}, multiplication by such a function $m$ annihilates every distribution supported on $K$. 

Section \ref{sec:atomization} develops the construction of the thin factor through a concept we call \emph{atomization} of closed conical subsets $\Lambda\subset T^\ast M\setminus 0$ (Definition \ref{def:tame-atomizable}): For a chosen Sobolev order $s'\in \R$, the atoms of an $H^{s'}$-atomizable $\Lambda$ consist of compactly supported distributions in $\D'_\Lambda(M)\cap H^{s'}_\mathrm{loc}(M)$ with supports of empty interior, occurring inside every non-empty open subset of $M$.  Proposition \ref{prop:existence1} and Corollary \ref{cor:existence1} establish the existence of atomizable closed conical sets inside prescribed open conical sets.  The atoms for fiber-wise one-sided sets $\Lambda$ ultimately come from   Proposition \ref{prop:KOatom}, which is a distributional formulation of a deep result by Kozma-Olevski\u{\i} \cite{KozmaOlevskii}. On the other hand, positive atoms in dimensions $\geq 2$ are obtained by an elementary construction of distributions supported on hypersurfaces contained in charts. 

Section~\ref{sec:mainproofs} combines the two constructions of the thick and thin factors. Leaving the Sobolev regularity and positivity aspects aside for simplicity, the argument can be summarized as follows. For a countable basis $\{\mathscr U_j\}_{j\in \N}$ of the topology of $M$, a given smooth function $f\in \Cinft(M)$ with zero set of empty interior, and two given closed conical sets $\Gamma,\Lambda\subset T^\ast M\setminus 0$ with non-empty fibers and $\Gamma\cap (-\Lambda)=\emptyset$ such that $\Lambda$ is atomizable, the proof of Theorem \ref{thm:main} inductively chooses atoms $A_j\in \D'_\Lambda(M)$ supported in $\mathscr U_j$, smooth functions $m_j\in \Cinft(M)$ vanishing to infinite order on $\supp A_j$ and  with zero sets of empty interior, and coefficients $\lambda_j\in \C$ in such a way that 
\[
 u=\lim_{N\to\infty} f\prod_{j=1}^N m_j, \qquad v=\sum_{j=1}^{\infty}\lambda_jA_j
\]
converge in $\D'_\Gamma(M)$ and $\D'_\Lambda(M)$, respectively. Separately for $u$ and $v$, chosen sequences of test functions ensure non-vanishing on every $\mathscr U_j$, so that $\supp u=\supp v=M$,  while Lemma \ref{lem:flatkills} and Proposition \ref{prop:seqmult} give $uv=0$.  Section~\ref{sec:remainingthms} derives Theorems \ref{thm:intro1}--\ref{thm:intro3}, with Lemma \ref{lem:supportdetection} ruling out the continuity of full-support zero divisors.  

Section~\ref{sec:examples} proves atomizability for conormal bundles of foliations (Lemma \ref{lem:foliations}), conormal cones of fragmentations (Corollary \ref{cor:fragments}), and applies the atomizability of integrable ray bundles (Corollary \ref{cor:integrable} from Section \ref{sec:atomization}) in Corollary \ref{cor:raybundles}. These results are summarized in Theorem \ref{thm:intro4}, which is proved in Section \ref{sec:introthmproof}. 

Finally, Section \ref{sec:alternative} presents an example of an alternative symmetric  thin-thin construction of a pair of full-support zero divisors on $\R^2$, illustrating that our thick-thin construction is by no means the only way to produce full-support zero divisors. This distinguishes the bare existence of full-support zero divisor pairs from the stronger Sobolev regularity, wavefront-set confinement and accumulation conclusions obtained by the thick-thin construction.

\subsection*{AI use statement}\label{sec:Aiuse}

The author used the GPT-5.5 Pro and GPT-5.6 Sol models of OpenAI’s ChatGPT LLM in the following ways: A literature search initiated in a conversation with GPT-5.5 Pro identified the work of Kozma-Olevski\u{\i} \cite{KozmaOlevskii} and its relevance for the construction of full-support zero divisors in dimension one. The author developed the argument in Proposition \ref{prop:KOatom}, which is used through Corollary \ref{cor:existence1} in the one-dimensional case of the proof of Theorem \ref{thm:intro1}. Later, GPT-5.6 Sol was used to audit the proofs and explore strengthenings. The AI suggested sharpening the statements on the achievable Sobolev regularity and the absence of continuous full-support zero divisors, contributing the ideas for the proofs of the technical Lemmas \ref{lem:wflemma}, \ref{lem:nearidentity} and \ref{lem:supportdetection}. The author independently checked every AI-suggested correction and improvement and decided whether to incorporate it into the manuscript, which was then done  manually.  Finally, conversations with the AI led to stylistic optimization of the exposition, again only through independently verified manual edits by the author. The author made every final decision concerning the manuscript and wrote it entirely.

\subsection*{Acknowledgments}

The author would like to thank Mihajlo Cekić, Semyon Dyatlov, Colin Guillarmou and Tobias Weich for stimulating discussions. He is also grateful to Tobias Weich and Lasse Wolf  for feedback on  earlier versions of this paper and to Zhongkai Tao for suggesting the example presented in Section \ref{sec:alternative}. This work is funded by the Deutsche Forschungsgemeinschaft (DFG, German Research Foundation) – Project-ID 491392403 – TRR 358 and was carried out while the author was a member of the DFG's Heisenberg Programme (Project no.~DE 3864/3-1).

\section{Preliminaries}\label{sec:prelim}

\subsection{Conventions and notation}

In this paper, $M$ always denotes a smooth manifold. For convenience, we  assume that a smooth measure on $M$ has been fixed, providing a linear embedding of the space $L^1_\mathrm{loc}(M)$ of locally integrable functions $M\to \C$ into the space $\D'(M)$ of distributions. This avoids a half-density formalism. Readers who prefer to focus on the case that $M$ is an open subset of $\R^n$ may use Lebesgue measure.

 For the evaluation of a distribution 
 $u\in \D'(M)$ at a test function $\varphi\in \CT(M)$, we will write either $u(\varphi)$ or $\langle u,\varphi \rangle$, depending on the context. For example, when $u\in L^1_\mathrm{loc}(M)$, then $\langle u,\varphi \rangle=\langle u,\bar \varphi \rangle_{L^2(M)}$, where $\overline{\phantom{\varphi}}$ denotes complex conjugation. 
 
Given a vector bundle $E$ over a base space $B$, $\dot E$ denotes the complement of the zero section in $E$. When we say that a subset of $\dot E$ is conical, this is understood fiberwise. For  $p\in B$, $E_p$ denotes the fiber over $p$ and for any set $S\subset E$ we write $S_p:=E_p\cap S$.

We use $\N=\{1,2,\ldots\}$, $\N_0:=\N\cup\{0\}$, $\R_{>0}:=\{x\in \R\,|\,x>0\}$, $\R_\times:=\R\setminus\{0\}$.

For $m\in \N_0$ and $x\in \R^m$, $B^m_r(x)\subset \R^m$ denotes the Euclidean open ball of radius $r>0$ centered at $x$. If the dimension $m$ is clear from the context, we just write $B_r(x)$.

\subsection{Distributions with controlled wavefront set}\label{sec:distribcontrolledWF}

For a closed conical set $\Gamma\subset\dot T^*M$, we define 
$\D'_{\Gamma}(M):=\{u\in\D'(M)\,|\,\WF(u)\subset\Gamma\}$. Following \cite[p.~262]{HormanderI}, we say that a sequence $(u_n)_{n\in \N}\subset \D'_{\Gamma}(M)$ converges to $u\in \D'_{\Gamma}(M)$  if
\begin{enumerate}[label=(\roman*)]
\item $u_n\to u$ in $\D'(M)$ as $n\to \infty$, i.e. $\langle u_n-u,\varphi\rangle\to0$ for every $\varphi\in\CT(M)$;
\item for every function $\chi\in \CT(U)$ supported in a chart $U\subset M$ (identified with an open subset of $\R^{\dim M}$), every $N\in \N$, and every closed cone $V\subset \R^{\dim M}\setminus \{0\}$ satisfying $(\supp \chi\times V)\cap\Gamma=\emptyset$ (after identifying $T^\ast U\equiv U\times \R^{\dim M}$), one has
\bq
\sup_{\substack{\xi\in V}}|\xi|^N \left|\widehat{\chi u_n}(\xi)-\widehat{\chi u}(\xi)\right|\stackrel{n\to \infty}{\longrightarrow} 0,\label{eq:seqconvHorm}
\eq
where  $\widehat{\quad}$ denotes the standard Euclidean Fourier transform in $\R^{\dim M}$.
\end{enumerate}
Dabrowski--Brouder \cite{DabrowskiBrouder} showed that $\D'_{\Gamma}(M)$ can be given a natural locally convex topology, called the \emph{normal topology}, whose notion of convergence of sequences coincides with the above. We will, however, not use this topology and require only the above sequential notion of convergence in Hörmander's sense.
\begin{proposition}\label{prop:seqmult}
Let $\Gamma,\Lambda\subset\dot T^*M$ be closed conical sets with $\Gamma\cap(-\Lambda)=\emptyset$ and let $v\in\D'_\Lambda(M)$.  If $u_n\to u$ in $\D'_\Gamma(M)$, then $u_nv\to uv$ in $\D'(M)$.
\end{proposition}
\begin{proof}
See \cite[Theorems 8.2.4, 8.2.10]{HormanderI}. 
\end{proof} 

\subsection{Convergence criteria}

We tweak the above formulation of the sequential convergence in $\D'_\Gamma(M)$ in a couple of ways for later convenience. In \eqref{eq:seqconvHorm} we can assume without loss of generality that the cutoffs $\chi$ are real-valued, by considering real and imaginary parts separately.  Furthermore:
\begin{lemma}The convergence condition defined by \emph{(i)} and \emph{(ii)} above  remains equivalent  if  in \eqref{eq:seqconvHorm} we replace the supremum $\sup_{\substack{\xi\in V}}$ by the supremum $\sup_{\substack{\xi\in V}:|\xi|\geq 1}$.
\end{lemma}
\begin{proof}
With the identification $T^\ast U\equiv U\times \R^{\dim M}$ as in \eqref{eq:seqconvHorm}, we have for all $v\in \D'(U)$, $\chi\in \CT(U)$  and all $\xi\in \R^n$ the formula $\widehat{\chi v}(\xi)=v(e^{-i\xi \cdot}\chi)$, 
and for any compact subset $\mathscr C\subset \R^n$ the resulting set of test functions $\{e^{-i\xi \cdot}\chi\,|\,\xi \in \mathscr C\}\subset C^\infty_c(U)$ 
is bounded. As is well-known, the pointwise (weak-$\ast$) formulation of the condition (i) that $u_n\to u$ in $\D'(M)$ is equivalent to the convergence $u_n\to u$ in the actual (i.e., strong dual) topology of $\D'(M)$ (the reason is that $C^\infty_c(M)$ is a Montel space). Therefore, assuming (i), we already have
\[
\sup_{\xi\in \R^n:|\xi|\leq 1}\left|\widehat{\chi u_n}(\xi)-\widehat{\chi u}(\xi)\right|\stackrel{n\to \infty}{\longrightarrow} 0,
\]
since $\mathscr C:=\{\xi\in \R^n:|\xi|\leq 1\}\subset \R^n$ is compact. We do not even need $\xi\in V$ here.
\end{proof} 
Finally, we replace the Euclidean norm $|\xi|$ by the \emph{japanese bracket} $
\langle\xi\rangle:=\sqrt{1+|\xi|^2}\geq |\xi|$. 
The upshot is that we can formulate the convergence conditions (i) and (ii) above equivalently by using in (ii) the  \emph{microlocal seminorms}
\bq
 \|u\|^\Gamma_{U,N,V,\chi}:=\sup_{\xi\in V:|\xi|\geq 1}\,\langle\xi\rangle^N|\widehat{\chi u}(\xi)|,       \qquad u\in \D'_\Gamma(M),\label{eq:microlocalseminorms} 
\eq
where $N\in \N$, $\chi\in \CT(U,\R)$ is supported in the chart $U\subset M$,  and $V\subset \R^{\dim M}\setminus \{0\}$ a closed cone such that $(\supp \chi\times V)\cap\Gamma=\emptyset$ holds when identifying $T^\ast U\equiv U\times \R^{\dim M}$. 

It actually suffices to consider only a countable family of these seminorms. More precisely, we have the following sufficient condition for the existence of limits in $\D'_\Gamma(M)$:
\begin{lemma}\label{lem:countableseminorms1}
There is a countable family of microlocal seminorms $\|\cdot\|_1^{\Gamma},\|\cdot\|_2^{\Gamma},\ldots$ with the following property:  If $u_n\in\D'_\Gamma(M)$, $u\in \D'(M)$, 
$u_n\to u$ in $\D'(M)$, and $\{u_n\}_{n\in \N}$ is a Cauchy sequence with respect to $\|\cdot\|_k^{\Gamma}$ for each $k\in \N$, then $u\in\D'_\Gamma(M)$ and 
$u_n\to u\text{ in }\D'_\Gamma(M)$.
\end{lemma}
\begin{proof}
When $M$ is an open subset of $\R^n$, it is shown in \cite[Cor.~13]{DabrowskiBrouder} that it suffices to consider a countable family of microlocal seminorms, and it is pointed out already in \cite[p.~262]{HormanderI} that, in general, the convergence condition \eqref{eq:seqconvHorm} for a given seminorm $\|\cdot\|^\Gamma_{U,N,V,\chi}$ can be replaced by the boundedness condition $
\sup_{n\in \N} \|u_n\|^\Gamma_{U,N,V,\chi}<\infty$ which is satisfied by a Cauchy sequence for the seminorm $\|\cdot\|^\Gamma_{U,N,V,\chi}$. For a general smooth manifold $M$, one combines these observations with a countable atlas of $M$.
\end{proof}
The microlocal seminorms are continuous with respect to multiplication by smooth functions. More precisely: If $f\in \Cinft(M)$, $u\in \D'_\Gamma(M)$, and $U,N,V,\chi$  are as in \eqref{eq:microlocalseminorms}, then
\bq
\|fu\|^\Gamma_{U,N,V,\chi}\leq \|u\|^\Gamma_{U,N,V,(\Re f)\chi}+\|u\|^\Gamma_{U,N,V,(\Im f)\chi}.\label{eq:multiplicationmicrolocalnorms}
\eq
Note, however, that if we fix a countable family of seminorms as in Lemma \ref{lem:countableseminorms1} and take the seminorm on the left-hand side in \eqref{eq:multiplicationmicrolocalnorms} from that family, then the seminorms on the right-hand side in \eqref{eq:multiplicationmicrolocalnorms} are not necessarily part of the same family.

\subsection{Sobolev spaces}\label{sec:Sobolev}

In this paper we will not really be interested in the whole spaces $\D'_\Gamma(M)$  of distributions with controlled wavefront set as introduced in Section \ref{sec:distribcontrolledWF}, but rather in their intersections with Sobolev spaces. Therefore, we spend some time to introduce the latter and prepare useful technicalities for later use. 

In this whole Section \ref{sec:Sobolev}, we put $n:=\dim M$. We fix a family of Sobolev spaces  $
H^s_\mathrm{loc}(M)\subset \D'(M)$, $s\in \R$. 
For their definition and more details on the basics summarized in the following, see \cite[Sec.~10]{hintz}, \cite[App.~E]{dyatlovzworskibook}, \cite[Ch.~5]{Lefeuvrebook}. Each Sobolev space $H^s_\mathrm{loc}(M)$ is a Fréchet space. If $u\in H^s_\mathrm{loc}(M)$ and $\chi\in \CT(M)$ is a cutoff function supported in a chart $U\subset M$ identified with an open subset of $\R^n$, then one has $\chi u\in H^s(\R^n)$, where 
\bqn
H^s(\R^n):=\big\{u\in \mathcal S'(\R^n)\,|\, \big(\xi\mapsto \eklm{\xi}^{s}\hat u(\xi) \big)\in L^2(\R^n) \big\}
\eqn
is the classical Sobolev space (here $\mathcal S'(\R^n)$ is the space of all tempered distributions, the dual space of the space $\mathcal S(\R^n)$ of Schwartz functions) equipped with the inner product
\bq
\eklm{u,v}_{H^s(\R^n)}:=\eklm{\eklm{\cdot}^{s}\hat u,\eklm{\cdot}^{s}\hat v}_{L^2(\R^n)}=\langle(I+\Delta)^{s/2} u,(I+\Delta)^{s/2} v\rangle_{L^2(\R^n)},\label{eq:innerproductHs}
\eq
where $(I+\Delta)^{s/2}:\S'(\R^n)\to \S'(\R^n)$ is the formally $L^2$-symmetric invertible pseudodifferential operator built from the non-negative Euclidean Laplacian $\Delta$. The Sobolev spaces with different exponents are related by the ``order-shifting'' isomorphisms
\bq
(I+\Delta)^{r/2}: H^{s+r}(\R^n)\to H^{s}(\R^n),\qquad r,s\in \R.\label{eq:IplusLaplacian}
\eq
 Between $H^s(\R^n)$ and $H^{-s}(\R^n)$ there is the canonical perfect pairing
\begin{align}\begin{split}
H^s(\R^n)\times H^{-s}(\R^n)&\to \C\\
(u,v)&\mapsto \eklm{u,v}_{H^s,H^{-s}}:=\eklm{\eklm{\cdot}^{s}\hat u,\eklm{\cdot}^{-s}\hat v}_{L^2(\R^n)},\label{eq:dualityRn}\end{split}
\end{align}
extending the $L^2$-pairing on $\CT(\R^n)\times \CT(\R^n)$ continuously.  Choosing $r=-2s$ in \eqref{eq:IplusLaplacian} allows us to write the pairing \eqref{eq:dualityRn}  as an inner product in $H^{s}(\R^n)$:
\begin{align}\begin{split}
\eklm{u,v}_{H^s,H^{-s}}&=\eklm{\eklm{\cdot}^{s}\hat u,\eklm{\cdot}^{s}\eklm{\cdot}^{-2s}\hat v}_{L^2(\R^n)}\\
&=\eklm{u,(I+\Delta)^{-s}v}_{H^{s}(\R^n)}.\end{split}\label{eq:pairingalternative}
\end{align}

\subsubsection{Sobolev Hilbert spaces}\label{sec:sobolevhilbert}
Fix a partition of unity $\{\tau_j\}_{j\in \N}\subset \CT(M;\R)$ of $M$ subordinate to the open cover formed by the chart domains $U_j\subset M$ of a countable atlas $\mathscr A=\{\Phi_j:U_j\to V_j\}_{j\in \N}$ with charts $\Phi_j:U_j\to V_j\subset \R^n$.

Given a compact set $K\subset M$, we define
\[
H^s(K):=\{u\in H^s_\mathrm{loc}(M)\,|\, \supp u \subset K\}\subset \E'(M)\cap H^s_\mathrm{loc}(M).
\]
Let $I_K:=\{j\in \N\,|\,\supp \tau_j\cap K\neq \emptyset\}$. This set is finite because the partition of unity is locally finite and $K$ is compact. We can put on $H^s(K)$ an inner product $\eklm{\cdot,\cdot}_{H^s(K)}$ obtained by polarization from the  norm defined by
\bq
\norm{u}_{H^s(K)}^2:=\sum_{j\in I_K} \norm{(\Phi_{j}^{-1})^\ast (\tau_{j} u)}^2_{H^s(\R^n)}.\label{eq:localnorm}
\eq
This makes $H^s(K)$ a Hilbert space such that the topology induced by the inner product agrees with the subspace topology inherited from $H^s_\mathrm{loc}(M)$ \cite[Prop.~6.58]{hintz}.

Next, fix an exhaustion
\bq
M=\bigcup_{l=1}^\infty M_l\label{eq:exhaustion}
\eq
by compact subsets $M_l\subset M$ such that $M_{l}\subset \mathrm{Int}(M_{l+1})$ holds for all $l\in \N$ (if $M$ is already compact, we put $M_l:=M$ for all $l$). Then 
\[
H^s_{\mathrm{comp}}(M):=\bigcup_{K\subset M\text{ compact}}H^s(K)=\bigcup_{l\in \N}H^s(M_l)
\]
is equipped with the inductive limit topology. We have continuous, dense  inclusions $\CT(M)\subset H^s_{\mathrm{comp}}(M)\subset  \E'(M)$, $s\in \R$. 
Also, if $s'\leq s$, then one has
\bq
H^s_{\mathrm{loc}}(M)\subset H^{s'}_{\mathrm{loc}}(M), \qquad H^s(K)\subset H^{s'}(K),\;K\subset M\text{ compact}.\label{eq:ssprimeincl}
\eq
In combination with the fact that every compactly supported distribution has finite order, this allows us to write $\E'(M)$ as a countable union of Hilbert spaces:
\bq
\E'(M)=\bigcup_{k,l\in \N} H^{-k}(M_l).\label{eq:EHilbert}
\eq
Note that, if $M$ is compact, we have for all $s\in \R$
\[
H^s_\mathrm{loc}(M)=H^s_\mathrm{comp}(M)=H^s(M_l)=H^s(M)\quad\forall\;l\in \N.
\]
Now, fix a family of real-valued cutoff functions $\{\chi_l\}_{l\in \N}\subset \CT(M,\R)$ such that
\bqn
\supp \chi_l\subset \mathrm{Int}(M_{l}),\qquad \chi_{l+1}=1\text{ on }M_{l}.
\eqn
If $M$ is compact, we take $\chi_l=1$ for all $l$, which is consistent with our convention that $M=M_l$ for all $l$ in this case. 

\begin{lemma}\label{lem:distribconv}A sequence of distributions $u_j\in \D'(M)$ converges in $\D'(M)$ if, and only if, for each $l\in \N$ the sequence of distributions $\chi_l u_j\in \E'(M)$ converges in $\E'(M)$.
\end{lemma}
\begin{proof}
The ``only if'' is clear since multiplication $\chi_l\cdot:\D'(M)\to \E'(M)$ is continuous.

Conversely, suppose that, for each $l\in \N$, the sequence $\chi_l u_j\in \E'(M)$ converges in $\E'(M)$. Then we define a distribution  $u\in \D'(M)$ as follows: Given $\varphi\in \CT(M)$, choose $l\in \N$ large enough that $\chi_l=1$ on $\supp \varphi$. Then put $u(\varphi):=\lim_{j\to \infty}(\chi_l u_j)(\varphi)$. This is well-defined: If for some other $l'\in \N$ one has $\chi_{l'}=1$ on $\supp \varphi$, then for each $j\in \N$ one has 
$(\chi_l u_j)(\varphi)=u_j(\varphi)=(\chi_{l'} u_j)(\varphi)$, 
so the two limits agree, as required. By construction,  $u_j\to u$ in $\D'(M)$.
\end{proof}

We will not need a manifold version of the duality pairing \eqref{eq:dualityRn}. It suffices to observe that, by the definition \eqref{eq:localnorm} of the Hilbert norm $\norm{\cdot}_{H^s(K)}$ for a compact set $K\subset M$ as a finite sum of local Sobolev norms in $\R^n$, there is a constant $C_{K,s}>0$ such that
\bq
|\langle u,v\rangle|\leq C_{K,s} \norm{u}_{H^s(K)}\norm{v}_{H^{-s}(K)}\qquad \forall\; u,v\in \CT(M),\supp u,\supp v\subset K. \label{eq:globalpairing}
\eq

\begin{definition}\label{def:uniformorder}A family of distributions $\{u_j\}_{j\in J}\subset \D'(M)$ has \emph{uniformly bounded Sobolev order} if there is an $s\in \R$ such that $\chi_l u_j\in H^{s}(M_l)$ for all $j\in J$, $l\in \N$. In this situation, we say that the family is \emph{of Sobolev order $s$}.
\end{definition}
Note that if a family of distributions of uniformly bounded Sobolev order is of some Sobolev order $s$, then it is also of Sobolev order $s'$ for all $s'<s$ in view of  \eqref{eq:ssprimeincl}. 

\begin{definition}\label{def:uniformsminus0}Given $s\in \R$ and a Sobolev space $H^s_\bullet(\circ)$ of any of the forms introduced above, we define $
H^{s-0}_\bullet(\circ):=\bigcap_{\eps>0}H^{s-\eps}_\bullet(\circ)$. 
Furthermore, \emph{``Sobolev order $s-0$''} means \emph{``Sobolev order $s-\eps$ for all $\eps>0$''}. We extend Definition \ref{def:uniformorder} accordingly.
\end{definition}

\begin{example}\label{ex:codimSobolev}Let $n:=\dim M$, $k\in \N_0$, $k\leq n$, and let $S_j\subset M$, $j\in J$, be a family of embedded submanifolds of dimension $k$ with smooth volume densities $ds_j$. Take cutoff functions $\chi_j\in \CT(S_j)$ and define a family of distributions $u_j\in \E'(M)$ by 
\[
u_j(\varphi):=\int_{S_j}\chi_j(s_j) \varphi(s_j)\, d s_j,\quad \varphi\in \CT(M).
\]
Then the family $\{u_j\}_{j\in J}$ has uniformly bounded Sobolev order $-\frac{n-k}{2}-0$. Indeed, for each $j$, in any chart of $M$ intersecting $S_j$ non-trivially, the Fourier transform of $u_j$ decays rapidly outside the conormal bundle of $S_j$, where it is only bounded.  Thus, to get an $L^2$-function (after identifying the fibers of the conormal bundle with $\R^{n-k}$), we need to multiply the Fourier transform by a factor $\eklm{\xi}^s$ such that $\eklm{\xi}^{2s}\in L^2(\R^{n-k})$, which holds iff $s<-\frac{n-k}{2}$.
\end{example}

\begin{corollary}\label{cor:convergencekl}Let $s\in \R$ and let $u_j\in \D'(M)$ be a sequence of uniformly bounded Sobolev order $\bullet$, where $\bullet=s$ or $\bullet=s-0$. 

If $\bullet=s$, suppose that for every $l\in \N$ the sequence $\chi_l u_j\in H^{s}(M_l)$ converges in  $H^{s}(M_l)$ as $j\to \infty$. 

If $\bullet=s-0$, suppose that for every $l\in \N$ the sequence $\chi_l u_j\in H^{s-0}(M_l)$ converges in  $H^{s-1/l}(M_l)$ as $j\to \infty$. 

 Then $u_j$ converges in $\D'(M)$.
\end{corollary}
\begin{proof}
By Lemma \ref{lem:distribconv}, it suffices that for each $l\in \N$ the sequence $\chi_l u_j$  converges in $\E'(M)$ as $j\to \infty$. This holds since the inclusions $H^{s}(M_l)\hookrightarrow \E'(M)$ as well as $H^{s-1/l}(M_l)\hookrightarrow \E'(M)$ are continuous.
\end{proof}

\subsubsection{Controlling Sobolev norms of smooth functions}\label{sec:Sobolevcontrol}

For later use, we now estimate the Sobolev norms of smooth functions. To this end, let $f\in \Cinft(M)$, $l\in \N$, and $I_{M_l}=\{j_{1,l},\ldots,j_{N_l,l}\}$, $N_l:=|I_{M_l}|$. 
Then, by \eqref{eq:localnorm}, we have for all $s\in \R$
\bq
\norm{\chi_l f}_{H^{-s}(M_l)}^2=\sum_{k=1}^{N_l} \norm{f_{k,l}}^2_{H^{-s}(\R^n)},\qquad f_{k,l}:=(\Phi_{j_{k,l}}^{-1})^\ast (\tau_{j_{k,l}} \chi_l f)\in \CT(\R^n).\label{eq:localfkl}
\eq
Now we write for each $k\in \{1,\ldots,N_l\}$
\begin{align*}
\norm{f_{k,l}}^2_{H^{-s}(\R^n)}&=\norm{(I+\Delta)^{s}(I+\Delta)^{-s}f_{k,l}}^2_{H^{-s}(\R^n)}\\
&\leq \norm{(I+\Delta)^{s}}_{H^{s}(\R^n)\to H^{-s}(\R^n)}^2\norm{(I+\Delta)^{-s}f_{k,l}}^2_{H^{s}(\R^n)}.
\end{align*}
Consider the Hilbert space $\mathcal{H}^s_l:=(H^s(\R^n))^{\oplus N_l}$. 
Then the above estimate shows that the linear operator
\begin{align}\begin{split}
J^s_l:\Cinft(M)&\to \mathcal{H}^s_l\\
f &\mapsto \big((I+\Delta)^{-s}f_{1,l},\ldots,(I+\Delta)^{-s}f_{N_l,l}\big)\label{eq:Jdef}\end{split}
\end{align}
satisfies
\bq
\norm{\chi_l f}_{H^{-s}(M_l)}\leq C_{s} \norm{J^s_l(f)}_{\mathcal{H}^s_l}, \label{eq:chi_lfestimate}
\eq
where $C_{s}:=\Vert(I+\Delta)^{s}\Vert_{H^s(\R^n)\to H^{-s}(\R^n)}$. 

Now let $y=(y_1,\ldots,y_{N_l})\in \mathcal{H}^{s}_l=(H^{s}(\R^n))^{\oplus N_l}$. Then we have
\begin{align}\begin{split}
\eklm{J^s_l(f),y}_{\mathcal{H}^s_l}&=\sum_{k=1}^{N_l}\langle(I+\Delta)^{-s}f_{k,l},y_k\rangle_{H^{s}(\R^n)}\\
&\stackrel{\text{\eqref{eq:pairingalternative}}}{=}\sum_{k=1}^{N_l}\langle y_k,f_{k,l}\rangle_{H^s,H^{-s}}\\
&=\sum_{k=1}^{N_l}y_k(f_{k,l}).\end{split}\label{eq:Lfypairing}
\end{align}
This shows that the linear form
\bq
L^{s,l}_y:f\mapsto \eklm{J^s_l(f),y}_{\mathcal{H}^s_l},\quad f\in \Cinft(M),\label{eq:Ly1}
\eq
is a distribution $L^{s,l}_y\in \E'(M)$.

\subsubsection{Complex conjugation}\label{sec:complexconj}

All Sobolev spaces introduced above  naturally carry a complex conjugation $u\mapsto \bar u$ extending the usual pointwise complex conjugation of test functions continuously. In particular, the spaces of real-valued smooth functions
\[
\Cinft(M,\R)\subset H^s_\mathrm{loc}(M),\qquad \{f\in\CT(M,\R)\,|\,\supp f\subset K\}\subset H^s(K)
\]
are real, respectively. Our custom Hilbert spaces $\mathcal{H}^s_l$ are just finite direct sums of Sobolev spaces and we equip them with the component-wise complex conjugation.  Since we chose all cutoff functions real-valued and the operators $(I+\Delta)^{-s}$ are formally $L^2$-symmetric for each $s\in \R$, it follows for all $s\in \R$, $l\in \N$ that
\bq
J^s_l(\Cinft(M,\R))\subset \mathcal{H}^s_l\quad\text{ is real}. \label{eq:JslRreal}
\eq
The Hilbert space $H^s(\R^n)$ can also be equipped with the alternative complex conjugation
\bq
u\mapsto \check {\bar u},\label{eq:conjugalternative}
\eq
given by the composition of the usual complex conjugation and the pullback $\check{\;}$ along the flip $x\mapsto -x$. The alternative complex conjugation  corresponds to the usual complex conjugation via the Fourier transform, since one has
\bq
\hat{\check{\bar u}}=\bar{\hat{u}}\quad \forall\;u\in\mathcal S'(\R^n).\label{eq:Fourierconjug}
\eq
Using this in combination with the Plancherel formula and the fact that $(I+\Delta)^{-s}$ is formally $L^2$-symmetric, one checks that the alternative complex conjugation is indeed compatible with the inner product:
\begin{align*}
\langle\check{\bar u},\check{\bar v}\rangle_{H^s(\R^n)}=\overline{\eklm{u,v}}_{H^s(\R^n)}\quad \forall\;u,v\in H^s(\R^n).
\end{align*}

\subsubsection{Multiplication by smooth functions}
For later use, let us record the following elementary continuity property of the Sobolev norms with respect to multiplication by smooth functions.
\begin{lemma}\label{lem:multiplicationsobolev}Let $f\in \Cinft(M)$, $s\in \R$, $l\in \N$. Then there is a $C_{f,s,l}>0$ such that
\[
\norm{\chi_l fg}_{H^{s}(M_l)}\leq C_{f,s,l} \norm{\chi_l g}_{H^{s}(M_l)}\quad \forall\; g\in H^{s}_{\mathrm{loc}}(M).
\]
\end{lemma}
\begin{proof}As in \eqref{eq:localfkl}, we have
\bqn
\norm{\chi_l fg}_{H^{s}(M_l)}^2=\sum_{k=1}^{N_l} \norm{(fg)_{k,l}}^2_{H^{s}(\R^n)},\qquad (fg)_{k,l}:=(\Phi_{j_{k,l}}^{-1})^\ast (\tau_{j_{k,l}} \chi_l fg).
\eqn
Therefore, it suffices to estimate each $(fg)_{k,l}\in H^{s}(\R^n)$. To this end, let $\rho_{k,l}\in \CT(M)$ be supported in the domain of $\Phi_{j_{k,l}}$ and equal to $1$ on $\supp(\tau_{j_{k,l}} \chi_l)$. Then we can write
\begin{align*}
(fg)_{k,l}&=(\Phi_{j_{k,l}}^{-1})^\ast (\tau_{j_{k,l}} \rho_{k,l}f\chi_l g)\\
&=(\Phi_{j_{k,l}}^{-1})^\ast (\rho_{k,l}f)\cdot (\Phi_{j_{k,l}}^{-1})^\ast (\tau_{j_{k,l}} \chi_l g)\\
&=\tilde f_{k,l}\, g_{k,l},\qquad\qquad\qquad\qquad \qquad \qquad \tilde f_{k,l}:=(\Phi_{j_{k,l}}^{-1})^\ast (\rho_{k,l}f)\in \CT(\R^n).
\end{align*}
Now we estimate
\begin{align*}
\norm{(fg)_{k,l}}_{H^{s}(\R^n)}&=\Vert\tilde f_{k,l}\, g_{k,l}\Vert_{H^{s}(\R^n)}\\
&=\Vert(I+\Delta)^{s/2} (\tilde f_{k,l}\, g_{k,l})\Vert_{L^2(\R^n)}\\
&=\Vert(I+\Delta)^{s/2} \tilde f_{k,l}(I+\Delta)^{-s/2}(I+\Delta)^{s/2} g_{k,l}\Vert_{L^2(\R^n)}\\
&\leq \underbrace{\Vert(I+\Delta)^{s/2} \tilde f_{k,l}(I+\Delta)^{-s/2}\Vert_{L^2(\R^n)\to L^2(\R^n)}}_{C_{f,s,k,l}}\Vert(I+\Delta)^{s/2} g_{k,l}\Vert_{L^2(\R^n)}\\
&=C_{f,s,k,l}\Vert g_{k,l}\Vert_{H^{s}(\R^n)}.
\end{align*}
This proves the claim with $C_{f,s,l}:=\max_{1\leq k \leq N_l}C_{f,s,k,l}$.
\end{proof}

\subsubsection{Microlocal regularity}

Let us introduce a convenient terminology:
\begin{definition}\label{def:microlocalSobolevregularity}Given $s\in \R$, a distribution $u\in \D'(M)$, a point $p\in M$ and a covector $\xi\in \dot T^\ast_p M$, we say that $u$ is \emph{microlocally in $H^s$ at $\xi$} if there is a chart $U\subset M$ containing $p$, a cutoff function $\chi \in \CT(U,\R)$ with $0\leq \chi\leq 1$, equal to $1$ near $p$, and an open cone  $V\subset \R^{\dim M}\setminus \{0\}$ with $\xi \in V$ when identifying $T^\ast U\equiv U\times \R^{\dim M}$, such that
\[
\int_V \eklm{\xi}^{2s}|\widehat{\chi u}(\xi)|^2\, d \xi<\infty.
\]
If $u$ is microlocally in $H^s$ at $\xi$ for all $s\in \R$, we say that $u$ is \emph{microlocally smooth at $\xi$.}
\end{definition}
Note that $u$ being microlocally smooth at $\xi$ is equivalent to $\xi \not\in \WF(u)$. 

\subsection{Distributions with fixed Sobolev order and controlled wavefront set}\label{sec:summary}Fix a closed conical subset $\Gamma\subset \dot T^\ast M$ and some $s\in \R$. The space of distributions of main relevance in this paper is the intersection
\[
\D'_\Gamma(M)\cap H^s_\mathrm{loc}(M),
\]
on which we have the family of seminorms 
\[
\mathsf{SEM}(\Gamma,s):=\big\{\norm{\chi_l\cdot}_{H^{s}(M_l)},\|\cdot\|^\Gamma_{U,N,V,\chi}\,|\,l\in \N, (U,N,V,\chi)\text{ as in \eqref{eq:microlocalseminorms}}\big\}
\]
consisting of all the  microlocal seminorms introduced in \eqref{eq:microlocalseminorms} and all the Sobolev seminorms of regularity $s$, as featured in  Corollary \ref{cor:convergencekl} and \eqref{eq:EHilbert}. 

On the subspace of smooth functions $\Cinft(M)\subset\D'_\Gamma(M)\cap H^s_\mathrm{loc}(M)$, it turns out that we can bound all seminorms in $\mathsf{SEM}(\Gamma,s)$ by \emph{Hilbert control seminorms} (Definition \ref{def:Hilbertcontrol} below) in a microlocally controlled way. This is developed in the following lemmas, which are later unified in  Proposition \ref{prop:control}.

We first treat the case of Sobolev seminorms, for which the result is nothing but a summary of what has already been observed in Sections \ref{sec:Sobolevcontrol} and \ref{sec:complexconj}.
\begin{lemma}\label{lem:SobolevdomHilbert}
Let  $q\in \mathsf{SEM}(\Gamma,s)$ be a Sobolev seminorm. Then there is a Hilbert space ${\mathcal H}_q$ equipped with a complex conjugation and a linear map  $J_q:\Cinft(M)\to {\mathcal H}_q$ such that  the following holds:
\begin{enumerate}
\item There is a constant $C>0$ such that
\[
q(f)\leq  C\|J_{q}(f)\|_{{\mathcal H}_q}\quad\forall\, f\in C^\infty(M).
\]
\item For each $y\in {\mathcal H}_q$, the pairing
\[
L_y(f):=\langle J_{q}(f),y\rangle_{{\mathcal H}_q},  \quad f\in C^\infty(M),
\]
defines a distribution $L_y\in H^{-s}_{\mathrm{comp}}(M)$.

\item The subspace $J_{q}(C^\infty(M,\R))\subset  {\mathcal H}_q$ is real.
\end{enumerate}
\end{lemma}
\begin{proof}
By assumption, $q$ is of the form $q=\norm{\chi_l\cdot}_{H^{s}(M_l)}$ for some $l$. We define $J_q:=J^{-s}_{l}:\Cinft(M)\to \mathcal{H}_l^{-s}=:\mathcal{H}_q$ 
with $J^{-s}_{l}$ as in \eqref{eq:Jdef}. We equip $\mathcal{H}_q$ with the usual complex conjugation (see Section \ref{sec:complexconj}). Then (1) was shown in \eqref{eq:chi_lfestimate}, (2) follows from \eqref{eq:Lfypairing} and  \eqref{eq:Ly1}, and (3) is \eqref{eq:JslRreal}. 
\end{proof} 
 
Now we turn to the microlocal seminorms, for which the statement is less obvious than for the Sobolev seminorms.

\begin{lemma}\label{lem:microdomHilbert}Let $\mathcal H_\mathrm{micro}$ denote the Sobolev space $H^{\frac{\dim M}{2}+1}(\R^{\dim M})$ equipped with the alternative complex conjugation \eqref{eq:conjugalternative}, let $q=\|\cdot\|^\Gamma_{U,N,V,\chi}\in \mathsf{SEM}(\Gamma,s)$ be a microlocal seminorm as in \eqref{eq:microlocalseminorms}, and let $W\subset(\R^{\dim M}\setminus \{0\})$ be a closed cone with $W\cap V=\emptyset$ (where $W=\emptyset$ is allowed). 

Then there is a linear map $J_{q,W}:\Cinft(M)\to \mathcal H_\mathrm{micro}$
 such that  the following holds:
\begin{enumerate}
\item There is a constant $C>0$ such that
\[
q(f)\leq  C\|J_{q,W}(f)\|_{\mathcal H_\mathrm{micro}}\quad\forall\, f\in C^\infty(M).
\]
\item For each $y\in \mathcal H_\mathrm{micro}$, the pairing
\[
L_y(f):=\langle J_{q,W}(f),y\rangle_{\mathcal H_\mathrm{micro}},  \quad f\in C^\infty(M),
\]
defines a distribution $L_y\in \E'(M)$ satisfying 
\bqn
\supp L_y\subset \supp \chi,\qquad \WF(L_y)\cap (U\times (-W))=\emptyset,
\eqn 
where we identify $T^\ast U\equiv U\times \R^{\dim M}$.
\end{enumerate}
Moreover, if $W\cap (V\cup -V)=\emptyset$ holds, then a map $J_{q,W}$ as above exists which in addition satisfies that  $J_{q,W}(C^\infty(M,\R))\subset  \mathcal H_\mathrm{micro}$ is real.
\end{lemma}
\begin{proof}Let $n:=\dim M$. We chose the Sobolev order $\frac{n}{2}+1$ to define $\mathcal H_\mathrm{micro}=H^{\frac{n}{2}+1}(\R^n)$ because for any $s>\frac{n}{2}$ we have the Sobolev embedding $H^s(\R^n)\hookrightarrow L^\infty(\R^n)$ which we use below; instead of $s=\frac{n}{2}+1$ we could also have chosen any other $s>\frac{n}{2}$.

Given a smooth function $f\in \Cinft(M)\subset \D'_\Gamma(M)$, we have
\[
q(f)= \sup_{\substack{\xi\in V}:|\xi|\geq 1}    \langle\xi\rangle^N|\widehat{\chi f}(\xi)|.
\]
Define a function $\vartheta\in C^\infty(\R^n)$, homogeneous of degree $0$ for $|\xi|\geq 1/2$,  as follows: If $V=\R^n$, then $W=\emptyset$ and we put $\vartheta:=1$. Otherwise, let $\tilde W\subset \R^n\setminus \{0\}$ be a closed cone disjoint from $V$ containing $W$ in its interior. This exists by compactness of the disjoint intersections $V\cap S^{n-1}$, $W\cap S^{n-1}$ with the unit sphere.   Choose $\vartheta$ such that $\vartheta=1$ on $V\cap\{|\xi|\geq 1\}$ and $\vartheta=0$ in a conical neighborhood of $\tilde W\cap\{|\xi|\geq 1\}$. Again, to see that such a function $\vartheta$ exists one uses the compactness of $S^{n-1}$.
 
Define $J_{q,W}:\Cinft(M)\to  H^{\frac{n}{2}+1}(\R^n)=\mathcal H_\mathrm{micro}$ by
\[
J_{q,W}(f)(\xi):=\vartheta(\xi)\langle\xi\rangle^N\widehat{\chi f}(\xi).
\]
This is well-defined because
$\widehat{\chi f}\in\mathcal S(\R^n)$. Using that $\vartheta=1$ on $V\cap \{|\xi|\geq 1\}$, the Sobolev embedding $\mathcal H_\mathrm{micro}\hookrightarrow L^\infty(\R^n)$ gives us a constant $C>0$ such that
\bqn 
q(f)\leq  C\|J_{q,W}f\|_{\mathcal H_\mathrm{micro}}\quad \forall\; f\in \Cinft(M).
\eqn
This finishes the proof of (1). 

If $W\cap (V\cup -V)=\emptyset$, we can achieve this for $\tilde W$ as well and  choose $\vartheta$ such that $\vartheta(-\xi)=\vartheta(\xi)$ holds for all $\xi\in \R^n$. Then, since $\chi$ is real-valued, $\eklm{-\xi}=\eklm{\xi}$, and the Fourier transform satisfies \eqref{eq:Fourierconjug}, 
\[
(J_{q,W}\bar f)(\xi)=\overline{(J_{q,W}f)(-\xi)}=\check{\overline{J_{q,W}f}}(\xi)\quad\forall\; f\in C^\infty(M),
\]
showing that $J_{q,W}$ maps real functions to real vectors in $H^{n/2 +1}(\R^n)=\mathcal H_\mathrm{micro}$ with respect to the chosen complex conjugation $u\mapsto \check{\bar u}$. This proves the reality claim.

It remains to prove (2). To this end, let $y\in H^{n/2 +1}(\R^n)$. Then the functional $\langle \cdot,y\rangle_{H^{n/2 +1}(\R^n)}$ on  $H^{n/2 +1}(\R^n)$ restricts to a continuous functional on $\mathcal S(\R^n)\subset H^{n/2 +1}(\R^n)$; denote the corresponding tempered distribution by $b_y\in \mathcal S'(\R^n)$.  Since $\vartheta(\xi)\langle\xi\rangle^N$ grows only polynomially, $
a_y:=\vartheta(\xi)\langle\xi\rangle^N b_y \in\mathcal S'(\R^n)$ 
is a well-defined tempered distribution, and by construction of $\vartheta$ we have $\mathrm{supp}\, a_y\cap \tilde W\cap \{|\xi|\geq 1\}=\emptyset$. 

For $f\in C^\infty(M)$, we now find
\[
L_y(f)\equiv \langle J_{q,W} f,y\rangle_{\mathcal H_\mathrm{micro}}=\langle \widehat{\chi f},a_y\rangle=\langle \chi f,\hat a_y\rangle=\langle f,\chi \hat a_y\rangle,
\]
where $\hat a_y\in \mathcal S'(\R^n)$ is the Fourier transform of  $a_y$. This shows that $L_y=\chi \hat a_y\in \E'(\R^n)$ is a compactly supported distribution with $\supp L_y\subset \supp \chi$.

Let us finally check the microlocal regularity of $L_y$. Given $\psi\in C_c^\infty(\R^n)$, the Fourier transform of $\psi\chi \hat a_y$ is the convolution $C_0\widehat{\psi\chi}\ast \check{a}_y$, where $C_0>0$ is a global Fourier normalization constant determined by $C_0\check{a}_y=\widehat{\widehat{a_y}}$. Note that 
\[
\supp \check{a}_y \cap \{|\xi|\geq 1\}\cap (-\tilde W)=\emptyset.
\]
This implies that there is a function $\tau\in \Cinft(\R^n)$, homogeneous of degree $0$ for $|\xi|\geq 1/2$, with $\tau=1$ on $\supp \check{a}_y \cap \{|\xi|\geq 1\}$ and $\supp \tau$ disjoint from $-\tilde W$.  Then there is some $c>0$ such that
\bq
|\xi-\zeta|\geq  c(|\xi|+|\zeta|)\quad \forall\;\xi\in -\tilde W,\; \zeta\in \supp \tau,\; |\xi|\geq 1.\label{eq:89502890295}
\eq
Now choose a $\varrho\in \CT(\R^n)$ equal to $1$ on $B_1(0)\subset \R^n$ and write
\[
\check{a}_y=\varrho\check{a}_y + (1-\varrho)\check{a}_y = \varrho\check{a}_y + (1-\varrho)\tau\check{a}_y.
\]
The estimate \eqref{eq:89502890295} implies that for each $N\in \N$ the set of Schwartz functions
\[
\{t^{N+1}(1-\varrho)(\cdot)\tau(\cdot)\widehat{\psi \chi}(t\xi-\cdot)\,|\, t\in \R_{>0},\; \xi\in -\tilde W,\norm{\xi}=1\}\subset \S(\R^n)
\]
is bounded. Similarly, for each $N\in \N$ the set of Schwartz functions
\[
\{t^{N+1}\varrho(\cdot)\widehat{\psi \chi}(t\xi-\cdot)\,|\, t\in \R_{>0}, \xi\in \R^n,\norm{\xi}=1\}\subset \S(\R^n)
\]
is bounded. Consequently, for $t\in \R_{>0}$, we get
\begin{align*}
\sup_{\substack{\xi \in -\tilde W\\ \norm{\xi}=1}}t^N|(\widehat{\psi \chi}\ast \check{a}_y)(t\xi)|&= \sup_{\substack{\xi \in -\tilde W\\ \norm{\xi}=1}}t^N|\check{a}_y(\widehat{\psi \chi}(t\xi-\cdot))|\\
&\leq \sup_{\substack{\xi \in -\tilde W\\ \norm{\xi}=1}}\frac{1}{t}|\check{a}_y(t^{N+1}(1-\varrho)(\cdot)\tau(\cdot)\widehat{\psi \chi}(t\xi-\cdot))| \\
&\qquad + \sup_{\substack{\xi \in -\tilde W\\ \norm{\xi}=1}}  \frac{1}{t}|\check{a}_y(t^{N+1}\varrho(\cdot)\widehat{\psi \chi}(t\xi-\cdot))|\\
&\leq C\frac{1}{t}\stackrel{t\to \infty}{\longrightarrow}0.
\end{align*}
Since $-\tilde W$ is a cone that contains $-W$ in its interior, the estimate above shows that  $\WF(L_y)  \cap (U\times(-W))=\emptyset$, finishing the proof.
\end{proof}

We now combine Lemmas \ref{lem:SobolevdomHilbert} and \ref{lem:microdomHilbert} into the technical main result of this section. 

\begin{proposition}\label{prop:control}Let $\Gamma\subset \dot T^\ast M$ be closed and conical, and let $\mathsf{Q}\subset \mathsf{SEM}(\Gamma,s)$ be a finite set. There exists a Hilbert space ${\mathcal H}_{\mathsf{Q}}$, equipped with a complex conjugation, such that the following holds:
\begin{enumerate}[leftmargin=*]
\item For each point $p\in M$ with $\Gamma_p\neq \emptyset$ and each $\xi\in \Gamma_p$ there exists a linear map $J_{\mathsf{Q},p,\xi}:\Cinft(M)\to {\mathcal H}_{\mathsf{Q}}$ with the following properties:
\begin{enumerate}
\item There is a constant $C>0$ such that
\[
q(f)\leq  C\|J_{\mathsf{Q},p,\xi}(f)\|_{{\mathcal H}_{\mathsf{Q}}}\quad\forall\, f\in C^\infty(M),\;q\in \mathsf{Q}.
\]
\item For each $y\in {\mathcal H}_{\mathsf{Q}}$, the pairing
\[
L_y(f):=\langle J_{\mathsf{Q},p,\xi}(f),y\rangle_{{\mathcal H}_{\mathsf {Q}}},  \quad f\in C^\infty(M),
\]
defines a distribution $L_y\in \E'(M)$.
\item For all $y\in  {\mathcal H}_\mathsf{Q}$, the distribution $L_y$ is microlocally in $H^{-s_0}$ at $-\xi$, where
\[
s_0:=\max\{s\in \R\,|\, \mathsf Q\text{\emph{ contains a Sobolev seminorm of regularity }}s\},
\]
with the convention that $L_y$ is microlocally smooth at $-\xi$ if the set on the right-hand side is empty.
\end{enumerate}

\item If $\Gamma_p\cap (-\Gamma_{p})\neq \emptyset$ and $\xi\in \Gamma_p\cap (-\Gamma_{p})$, then there exists a  linear map  $
J_{\mathsf{Q},p,\xi}$ as in (1) which in addition satisfies that $J_{\mathsf Q,p,\xi}(C^\infty(M,\R))\subset  {\mathcal H}_\mathsf{Q}$ is real.
\end{enumerate}
\end{proposition}
\begin{proof}
Let 
\[
q^\mathrm{micro}_1=\|\cdot\|^\Gamma_{U_1,N_1,V_1,\chi_1},\;\ldots,\;q^\mathrm{micro}_m=\|\cdot\|^\Gamma_{U_m,N_m,V_m,\chi_m}
\]
be the microlocal seminorms in $\mathsf{Q}$ (where $m=0$ is allowed, since $\mathsf{Q}$ may contain no microlocal seminorms). Fix some $j\in \{1,\ldots,m\}$.  For the sake of clarity, we now deviate from our usual practice to identify $U_j$ tacitly with an open subset of $\R^{\dim M}$ and instead explicitly introduce a chart $\Phi_j:U_j \to \tilde U_j\subset \R^{\dim M}$. We shall, however, still make the canonical identification $T^\ast \tilde U_j\equiv \tilde U_j\times \R^{\dim M}$ for the cotangent bundle of the open subset $\tilde U_j\subset \R^{\dim M}$. 

Put $\tilde{\chi}_j:=\chi_j\circ \Phi_j^{-1}\in \CT(\tilde U_j)\subset \CT(\R^{\dim M})$. By definition of the seminorm $q^\mathrm{micro}_j$, 
\bq
(\supp \tilde \chi_j\times V_j)\cap (\Phi^{-1}_j)^\ast(\Gamma\cap T^\ast U_j)=\emptyset.\label{eq:disjointj}
\eq
Now let some $p\in M$ with $\Gamma_p\neq \emptyset$ and a covector $\xi\in\Gamma_p$ be given. If $p\not\in \supp \chi_j$, define $W_j:=\emptyset$. Otherwise, define $
W_j:=\R_{>0}\mathrm{pr}_2((\Phi_j^{-1})^\ast \xi)$, 
where $\mathrm{pr}_2:\tilde U_j\times \R^{\dim M}\to \R^{\dim M}$ is the projection onto the second factor. In both cases $W_j$ is a closed cone in $\R^{\dim M}\setminus \{0\}$ and we infer from \eqref{eq:disjointj} that $
V_j\cap W_j=\emptyset$. 
Therefore, Lemma \ref{lem:microdomHilbert}  provides us with a linear map
$J_{q^\mathrm{micro}_j,W_j}:\Cinft(M)\to \mathcal H_\mathrm{micro}$ and for each $y\in \mathcal H_\mathrm{micro}$ with a distribution $L^{\mathrm{micro},j}_y:=\langle J_{q^\mathrm{micro}_j,W_j}(\cdot),y\rangle_{\mathcal H_\mathrm{micro}}\in \E'(M)$ satisfying 
\bq
\supp L^{\mathrm{micro},j}_y\subset \supp \chi_j,\qquad \WF(L^{\mathrm{micro},j}_y)\cap \Phi_j^\ast(\tilde U_j\times (-W_j))=\emptyset.\label{eq:WFWprecise}
\eq
For every $j\in \{1,\ldots,m\}$, we either have $p\not\in \supp \chi_j$, which by the first statement in \eqref{eq:WFWprecise} implies that  $p\not\in \supp L^{\mathrm{micro},j}_y$ and thus in particular $-\xi \not\in \WF(L^{\mathrm{micro},j}_y)$, or we have $p\in \supp \chi_j$, in which case $W_j$ is defined precisely such that $-\xi \not\in \WF(L^{\mathrm{micro},j}_y)$. 

If $\xi \in\Gamma_p\cap (-\Gamma_p)\neq \emptyset$ holds and $\xi \in\Gamma_p\cap (-\Gamma_p)$, then \eqref{eq:disjointj} implies $(V_j\cup -V_j)\cap W_j=\emptyset$, so the strengthened version of Lemma \ref{lem:microdomHilbert} applies and allows us to achieve that $J_{q^\mathrm{micro}_j,W_j}(C^\infty(M,\R))\subset  \mathcal H_\mathrm{micro}$ is real for all $j\in\{1,\ldots,m\}$.

Furthermore,  Lemma \ref{lem:microdomHilbert}(1) gives us for each $j\in \{1,\ldots,m\}$ a constant $C_j>0$ such that 
\bq
q^\mathrm{micro}_j(f)\leq  C_j\|J_{q^\mathrm{micro}_j,W_j}(f)\|_{{\mathcal H}_\mathrm{micro}}\quad\forall\, f\in C^\infty(M).\label{eq:microuniformC}
\eq
Now let $q^\mathrm{Sob}_1,\ldots,q^\mathrm{Sob}_n$, $s_1,\ldots,s_n$ 
be the Sobolev seminorms in $\mathsf{Q}$  and their regularities, respectively, (where $n=0$ is allowed, since $\mathsf{Q}$ may contain no Sobolev seminorms). 

For each $k\in \{1,\ldots,n\}$ Lemma \ref{lem:SobolevdomHilbert} provides us with a linear map $J_{q^\mathrm{Sob}_k}:\Cinft(M)\to {\mathcal H}_{q^\mathrm{Sob}_k}$, a constant $C'_k>0$ such that 
\bq
q^\mathrm{Sob}_k(f)\leq  C'_k\|J_{q^\mathrm{Sob}_k}(f)\|_{{\mathcal H}_{q^\mathrm{Sob}_k}}\quad\forall\, f\in C^\infty(M),\label{eq:SobunifofmC}
\eq
and for each $y\in {\mathcal H}_{q^\mathrm{Sob}_k}$ the distribution $L^{\mathrm{Sob},k}_y:=\langle J_{q^\mathrm{Sob}_k}(\cdot),y\rangle_{{\mathcal H}_{q^\mathrm{Sob}_k}}\in \E'(M)$ 
satisfies $L^{\mathrm{Sob},k}_y\in H^{-s_k}_{\mathrm{comp}}(M)$. Furthermore, $J_{q^\mathrm{Sob}_k}(C^\infty(M,\R))\subset  {\mathcal H}_{q^\mathrm{Sob}_k}$ is real.

Finally, let us put things together: We define
\[
J_{\mathsf{Q},p,\xi}:=\bigg(\bigoplus_{j=1}^m J_{q^\mathrm{micro}_j,W_j}\bigg)\oplus \bigg(\bigoplus_{k=1}^n J_{q^\mathrm{Sob}_k}\bigg):\Cinft(M)\longrightarrow \mathcal H_\mathrm{micro}^m\oplus \bigoplus_{k=1}^n {\mathcal H}_{q^\mathrm{Sob}_k}=:\mathcal{H}_\mathsf{Q}
\]
and equip $\mathcal{H}_\mathsf{Q}$ with the product Hilbert space structure and the complex conjugation given by summand-wise complex-conjugation. Then the claim (1a) follows from \eqref{eq:microuniformC} and \eqref{eq:SobunifofmC}  with
\[
C:=\begin{cases}\max_{j,k}(C_j+C_k'),&\text{if }n,m>0,\\
\max_{j}C_j, &\text{if } n=0,m>0,\\
\max_{k}C_k', &\text{if } n>0,m=0.\end{cases}
\]
Note that if $n=m=0$, then $\mathsf{Q}=\emptyset$ and all claims are trivially true.

Given $y\in {\mathcal H}_{\mathsf{Q}}$, write 
\[
y=y^\mathrm{micro}_1+\cdots + y^\mathrm{micro}_m +y^\mathrm{Sob}_1+\cdots+y^\mathrm{Sob}_n
\]
according to the definition of ${\mathcal H}_{\mathsf{Q}}$. Then the linear form $L_y=\langle J_{\mathsf{Q},p,\xi}(\cdot),y\rangle_{{\mathcal H}_{\mathsf {Q}}}$  on $\Cinft(M)$ decomposes as
\bq
L_y=\sum_{j=1}^m L^{\mathrm{micro},j}_{y^\mathrm{micro}_j}+\sum_{k=1}^n L^{\mathrm{Sob},k}_{y^\mathrm{Sob}_k}\in \E'(M),\label{eq:sumLy}
\eq
proving (1b). Here  all distributions in the second sum lie in $H^{-s_\mathrm{max}}_\mathrm{comp}(M)$ with $s_\mathrm{max}:=\max_{1\leq k\leq n}s_k$, in particular they are microlocally in $H^{-s_\mathrm{max}}$ at all $\xi \in \dot T^\ast M$. 

On the other hand, all  distributions in the first sum in \eqref{eq:sumLy} are simultaneously microlocally smooth at $-\xi$. Thus the whole sum, and therefore  $L_y$, is microlocally in $H^{-s_\mathrm{max}}$ at $-\xi$ if $n\geq 1$ and microlocally smooth at $-\xi$ otherwise. This proves (1c).

Finally,  if $\Gamma_p\cap (-\Gamma_p)\neq \emptyset$ holds, then we arranged that all components of $J_{\mathsf{Q},p,\xi}$ map $\Cinft(M)$ onto a real subspace of their respective target Hilbert spaces, and since we equipped $\mathcal{H}_\mathsf{Q}$ with the summand-wise complex conjugation, we conclude that $J_{\mathsf{Q},p,\xi}(C^\infty(M,\R))\subset \mathcal{H}_\mathsf{Q}$ is real. This proves (2).
\end{proof}

\begin{definition}\label{def:Hilbertcontrol}For a finite set $\mathsf{Q}\subset \mathsf{SEM}(\Gamma,s)$, a point $p\in M$ with $\Gamma_p\neq \emptyset$, a covector $\xi\in \Gamma_p$ and a linear map $J_{\mathsf Q,p,\xi}:\Cinft(M)\to {\mathcal H}_{\mathsf Q}$ as in Proposition \ref{prop:control},
\bq
\rho_{\mathsf Q,p,\xi}:=\norm{\cdot}_{{\mathcal H}_{\mathsf Q}}\circ J_{\mathsf Q,p,\xi}:\Cinft(M)\to [0,\infty)\label{eq:rhosqGamma}
\eq
is called a \emph{Hilbert control seminorm} for ${\mathsf Q}$ at $(p,\xi)$.
\end{definition}

\section{Approximation lemmas}\label{sec:approx}

In our main proofs, sequences of smooth functions vanishing to infinite order at prescribed compact sets play a key role. In this section, we set up  corresponding tools.

\subsection{Ideals and cones of smooth functions}

For a compact set $K\subset M$, consider the ideal in $C^{\infty}(M)$ given by all smooth functions vanishing to infinite order at all points in $K$:
\[
\I^{\infty}(K):=\{f\in C^{\infty}(M)\,|\, f^{(k)}(p)=0\;\forall\;p\in K,k\geq 0\}.
\]
Inside this ideal, we shall be interested in the convex real cone of all functions that take on positive real values outside  $K$:
\[
\I^{\infty,>0}(K):=\{f\in \I^{\infty}(K)\,|\, f(p)\in \R_{>0}\;\forall\;p\in M\setminus K\}.
\]
It is well-known that this set is non-empty: In fact, it would even be non-empty if $K$ were merely a closed subset of $M$. Namely, in this generality, there exists a smooth non-negative function $g:M\to \R$ with $g^{-1}(\{0\})=K$ \cite[Thm.~2.29]{LeeSmooth}. Composing $g$ with the function $
h:\R\to \R$ defined by $h(t):=0$ if $t\leq 0$ and $h(t):=e^{-1/t}$ if $t>0$ 
gives a function $f:=h\circ g\in \I^{\infty,>0}(K)$. 

Moreover, if $M=\R^n$ (or more generally, working in local coordinates in a chart of $M$), one can construct functions in $\I^{\infty,>0}(K)$ with a uniform quantitative control over the size of their derivatives:
\begin{lemma}\label{lem:quantitativeIinftyK}Let $n\in \N$ and $M=\R^n$. Then there are constants $C_{\alpha}>0$, $\alpha\in \N_0^n$,  such that the following holds: For all $r>0$ and all compact $K\subset B_{r/4}(0)\subset \R^n$ there exists a function 
$\theta\in \I^{\infty,>0}(K)$, equal to $1$ outside $B_{r/2}(0)$, satisfying
\[
\norm{\partial^\alpha \theta}_\infty\leq C_{\alpha} r^{-|\alpha|}  \quad\forall\;\alpha\in\N_0^n.
\]
\end{lemma}
\begin{proof}
This follows from the classical fact that regularized distances exist, see e.g.\ \cite[VI, Thm.~2]{Stein}: There are constants $c,C>0$, depending only on the dimension $n$, and constants $C_\alpha>0$ depending only on $\alpha\in \N_0^n$, such that for any closed subset $\mathscr C\subset \R^n$ there exists a function $\rho_{\mathscr C}\in \Cinft(\R^n\setminus \mathscr C,\R)$ which for all $x\in \R^n\setminus \mathscr C$, $\alpha\in \N_0^n$ satisfies
\[
c\,\mathrm{dist}(x,\mathscr C)\leq \rho_{\mathscr C}(x)\leq C \,\mathrm{dist}(x,\mathscr C),\qquad 
|\partial^\alpha \rho_{\mathscr C}(x)|\leq C_\alpha \,\mathrm{dist}(x,\mathscr C)^{1-|\alpha|}.
\]
To apply this to our situation, fix some smooth function $h\in \Cinft(\R,\R)$ vanishing to infinite order at $0$, satisfying $h(t)>0$ for all $t> 0$, and $h(t)=1$ for all $t\geq \frac{c}{4}$. Given a compact set $K\subset B_{r/4}(0)$, put $\tilde K:=\frac{1}{r}K\subset  B_{1/4}(0)$ and define $\theta_{\tilde K}\in \Cinft(\R^n,\R)$ by
\[
\theta_{\tilde K}(x):=\begin{cases}h(\rho_{\tilde K}(x)),&\text{ if }x\in \R^n\setminus\tilde K,\\
0, & \text{else}.
\end{cases}
\]
Then $\theta_{\tilde K}\in \I^{\infty,>0}(\tilde K)$, and $\theta(x):=\theta_{\tilde K}(x/r)\in \I^{\infty,>0}(K)$ is a function as desired.
\end{proof}

One reason for us to introduce the above sets of functions is that when a smooth function $f\in \Cinft(M)$ vanishes to infinite order on the support of a compactly supported distribution $u\in \E'(M)$, then $f$ and $u$ behave as if their supports were disjoint:
\begin{lemma}\label{lem:flatkills}
If $K\subset M$ is compact, $u\in\E'(M)$ satisfies $\supp  u\subset K$,  and
$f\in\I^{\infty}(K)$, then $fu=0$.
\end{lemma}
\begin{proof}This is a classical fact, see \cite[p.~102]{Malgrange1966}, for example. The argument given in that reference is that $\I^\infty(K)$ is the closure in $\Cinft(M)$ of the set of all functions vanishing in a neighborhood of $K$, so $u=0$ on $\I^{\infty}(K)$ by continuity. 
\end{proof}
Also the following microlocal lemma will be useful later:
\begin{lemma}\label{lem:wflemma}Let $p\in M$, and let $L\in \E'(M)$ be a real distribution (meaning that $L(\Cinft(M,\R))\subset \R$) with $L(1)>0$ and $L(f)\leq 0$ for all $f\in \I^{\infty,>0}(\{p\})$.  Then, if $s\geq-\frac{1}{2}\dim M$, the distribution $L$ is not microlocally in $H^s$ at any $\xi\in \dot T_p^\ast M$.  
\end{lemma}
\begin{proof}\footnote{As mentioned in the AI use disclosure on p.~\pageref{sec:Aiuse}, this proof is the author's revised version of arguments suggested by the  LLM  GPT-5.6 Sol.}
For the purpose of this proof, define $\I^{\infty,\geq 0}(\{p\}):=\{f\in \I^{\infty}(\{p\})\,|\, f\geq 0\}$. 
Since $f+\eps g\in \I^{\infty,>0}(\{p\})$ for all $f\in \I^{\infty,\geq 0}(\{p\})$, $g\in\I^{\infty,>0}(\{p\})$, $\eps>0$, continuity of $L$ gives $L(f)\leq 0$ for all $f\in \I^{\infty,\geq 0}(\{p\})$. 

Fix $\xi_0\in \dot T_p^\ast M$ and $s\geq  -n/2$, $n:=\dim M$. To provoke a contradiction, suppose that $L$ is microlocally in $H^s$ at $\xi_0$. By Definition \ref{def:microlocalSobolevregularity}, this means that there exists a cutoff $\chi\in C_c^\infty(M,\R)$ with $0\leq \chi\leq 1$, supported in a chart centered at $p$ and equal to $1$ near $p$, and an open cone $V\subset \R^n\setminus\{0\}$ containing $\xi_0$ such that, after identifying the chart with an open subset of $\R^n$ such that $p$ corresponds to $0\in \R^n$, the distribution $u:=\chi L\in \E'(\R^n) $
satisfies
\bq
\int_V\eklm{\xi}^{2s}|\widehat u(\xi)|^2\,d\xi<\infty.\label{eq:mainassumptionu}
\eq
Moreover, $
 u(1)=L(\chi)=L(1)-L(1-\chi) \geq  L(1)>0$ 
because $1-\chi\in \I^{\infty,\geq 0}(\{p\})$. Similarly, if $f\in\I^{\infty,\geq 0}(\{0\})\subset \Cinft(\R^n)$, then $\chi f\in \I^{\infty,\geq 0}(\{p\})\subset \Cinft(M)$, hence $u(f)=L(\chi f)\leq 0$. 
Finally, $u$ is real because both $L$ and $\chi$ are real. Since $u\in\E'(\R^n)$, it has finite order. Thus there is an $N\geq 0$ and a constant $C>0$ such that
\bq
|u(f)|\leq  C\max_{|\alpha|\leq  N}       \sup_{x\in \supp u}|\partial^\alpha f(x)|    \qquad \forall\; f\in C^\infty(\R^n).\label{eq:ufintieorder}
\eq
For $k\in \N_0$, define
\[
\I^{k,\geq 0}(\{0\}):=\{f\in \Cinft(\R^n,\R)\,|\, f\geq 0, f\text{ vanishes to order $k$ at }0\}.
\]
Next, we choose $\tau\in C_c^\infty(\R^n,\R)$ with $0\leq \tau\leq 1$ and $\tau=1$ near $0$ and define $\tau_\eps(x):=\tau(x/\eps)$,  $\eps>0$. 
Fix a function $f\in \I^{k,\geq 0}(\{0\})$, where $k> N$. Then one has
\bq
u((1-\tau_\eps) f)\leq 0\quad \forall\,\eps>0\label{eq:uepsleq}
\eq
since $(1-\tau_\eps) f\in \I^{\infty,\geq 0}(\{p\})$. By Taylor expansion, for every multi-index $\gamma$ with
$|\gamma|\leq  N$ there are a neighborhood $U$ of $0$ in $\R^n$ and a constant
$C_{f,\gamma}>0$ such that
\[
|\partial^\gamma f(x)|   \leq  C_{f,\gamma}|x|^{k-|\gamma|}    \quad \forall\; x\in U.
\]
For sufficiently small $\varepsilon$, one has $\supp\tau_\varepsilon\subset U$. For any multi-index $\alpha$ with $|\alpha|\leq  N$, the product rule gives
\[
\partial^\alpha(\tau_\varepsilon f)(x)=\sum_{\beta\leq \alpha}\binom{\alpha}{\beta} \varepsilon^{-|\beta|}(\partial^\beta\tau)(x/\varepsilon) \partial^{\alpha-\beta}f(x),\quad x\in \R^n.
\]
On $\supp\tau_\varepsilon$ one has $|x|\leq  C_\tau\varepsilon$ for some $C_\tau>0$, so the above implies
\[
\|\partial^\alpha(\tau_\varepsilon f)\|_{\infty}   \leq  C_{f,\alpha}\varepsilon^{k-|\alpha|}  \leq  C'_{f,\alpha}\varepsilon^{k-N}
\]
for some $C_{f,\alpha},C_{f,\alpha}'>0$. Combining this with \eqref{eq:ufintieorder} and recalling that $k>N$, we obtain $
u(\tau_\varepsilon f)\longrightarrow 0$ as $\eps\to 0$. 
In view of \eqref{eq:uepsleq}, we conclude that
\bq
u(f)\leq 0\quad \forall\, f\in \I^{k,\geq 0}(\{0\}),\; k>N.\label{eq:ufkleq0}
\eq
This intermediate result is very useful since $\I^{k,\geq 0}(\{0\})$ contains functions built from finitely many Fourier modes. Indeed, motivated by the fact that
\bq
\hat u(\xi)=u(e^{-ix\cdot\xi}),\quad\xi\in \R^n,\label{eq:uFT}
\eq
we now fix some $l\in \N$ with $2l>N$ and define for $\xi\in\R^n$ a function $P_\xi\in \I^{2l,\geq 0}(\{0\})$ by $
P_\xi(x):=(1-\cos(x\cdot\xi))^l$, $x\in \R^n$.
Then  \eqref{eq:ufkleq0} gives us
\bq
u(P_\xi)\leq 0\quad \forall\; \xi\in \R^n.\label{eq:uPthetaleq0}
\eq
We have for all $t\in \R$
\[
(1-\cos t)^l =\sum_{k=-l}^{l}a_k e^{ikt},\qquad a_k=\frac{(-1)^k}{2^l}\binom{2l}{k+l}.
\]
In particular, $a_{-k}=a_k$ for all $k$, and $a_0>0$. With \eqref{eq:ufkleq0},  \eqref{eq:uFT} and \eqref{eq:uPthetaleq0}, we get
\begin{align*}
 0\geq  u(P_\xi) &=a_0u(1)+\sum_{k=1}^{l}a_k\big(\hat u(k\xi)+\hat u(-k\xi)\big)\\
    &=a_0u(1)+2\sum_{k=1}^{l}a_k\,\Re\hat u(k\xi),
\end{align*}
where we used that $\hat u(-\zeta)=\overline{\hat u(\zeta)}$ for all $\zeta\in \R^n$ since $u$ is real. It follows that
\bq
a_0u(1)\leq  2\sum_{k=1}^{l}|a_k|\,|\hat u(k\xi)|.\label{eq:ua0positive}
\eq
Put $ A:=\left(\sum_{k=1}^{l}|a_k|^2\right)^{1/2}>0$. 
By \eqref{eq:ua0positive} and Cauchy-Schwarz, we find
\bq
\sum_{k=1}^{l}|\hat u(k\xi)|^2    \geq   \left(\frac{a_0u(1)}{2A}\right)^2    =:c_0>0    \qquad\forall\;\xi\in\R^n.\label{eq:lowerestim111}
\eq
Let us finally show that \eqref{eq:lowerestim111} is incompatible with \eqref{eq:mainassumptionu}. To this end, 
define for $R>0$ the annulus segment $\Omega_R:=\{\xi\in V\,|\,R<|\xi|<2R\}$. 
Because $V$ is a nonempty open cone, $\Omega_R$ satisfies
\bq
\mathrm{vol}(\Omega_R)= \mathrm{vol}(\Omega_1)R^n,\qquad  \mathrm{vol}(\Omega_1)>0.\label{eq:annvolume}
\eq
Integrating \eqref{eq:lowerestim111} over $\Omega_R$ gives
\begin{align}\begin{split}
c_0\mathrm{vol}(\Omega_1)R^n    &\leq  \sum_{k=1}^{l}\int_{\Omega_R} |\hat u(k\xi)|^2\,d\xi \\
    &=\sum_{k=1}^{l}k^{-n}
      \int_{k\Omega_R}|\hat u(\xi)|^2\,d\xi.\label{eq:integratedlower}\end{split}
\end{align}
Since $V$ is a cone, we have
\bq
k\Omega_R\subset V\quad\forall\; k\in \{1,\ldots,l\}\label{eq:coneincl}.
\eq
Moreover, if $\xi\in k\Omega_R$, $1\leq k\leq l$, then $R<|\xi|<2lR$. 
Consequently, there is a constant $C_{s,l}\geq 1$ such that
\bq
\eklm{\xi}^{-2s} \leq  C_{s,l}R^{-2s}   \qquad\forall\; \xi\in k\Omega_R,\ 1\leq  k\leq  l,\ R\geq 1.\label{eq:wcompare}
\eq
Define $ E_R:=\sum_{k=1}^{l}k^{-n} \int_{k\Omega_R}\eklm{\xi}^{2s}|\hat u(\xi)|^2\,d\xi$. 
Equations \eqref{eq:annvolume}, \eqref{eq:integratedlower}, and
\eqref{eq:wcompare} yield constants $c,C>0$, independent of $R$, such
that
\bq
 cR^{n+2s}\leq  CE_R.\label{eq:scaling}
\eq
On the other hand, by \eqref{eq:coneincl} and using that $|\eta|>R$ on $k\Omega_R$ for $1\leq k\leq l$, we find
\bq
0\leq  E_R \leq  \left(\sum_{k=1}^{l}k^{-n}\right) \int_{V\cap\{|\xi|>R\}}
          \eklm{\xi}^{2s}|\hat u(\xi)|^2\,d\xi.\label{eq:ER-zero}
\eq
By the integrability assumption \eqref{eq:mainassumptionu}, the right-hand side converges to $0$ as $R\to \infty$. However, if  $s>-n/2$, then $R^{n+2s}\to\infty$ and hence $E_R\to \infty$ by \eqref{eq:scaling}, a contradiction. If $s=-n/2$, the left-hand side of \eqref{eq:scaling} is the fixed positive constant $c$, while the right-hand side of \eqref{eq:ER-zero} still goes to zero, again a contradiction. The proof is finished.
\end{proof}

\subsection{Approximation near a point}

Let us formulate a first approximation lemma.

\begin{lemma}\label{lem:flat-cutoff-small}
Let $\Gamma\subset \dot T^\ast M$ be closed and conical, let $s\in \R$, let $\mathsf{Q}\subset \mathsf{SEM}(\Gamma,s)$ be a finite set, let $p\in M$ be a point with $\Gamma_p\neq \emptyset$, and for some $\xi \in \Gamma_p$, let $\rho_{\mathsf Q,p,\xi}$ be a Hilbert control seminorm for $\mathsf Q$ at $(p,\xi)$ as in \eqref{eq:rhosqGamma}. Furthermore, consider a function $f\in\I^\infty(\{p\})$ and a sequence $K_j\subset M$  of compact sets converging to $\{p\}$ in the sense that every neighborhood of $p$ in $M$ contains almost all $K_j$. Then there are functions $\theta_j\in \I^{\infty,>0}(K_j)$ 
such that
\[
\rho_{\mathsf Q,p,\xi}(f-\theta_jf)\longrightarrow 0 \quad \text{as } j\to\infty.
\]
\end{lemma}
\begin{proof}Choose a chart $U\subset M$ around $p$, identified with an open subset of $\R^{n}$, where $n=\dim M$. Then for all large enough $j\in \N$ we have   $K_j \subset B_{r_j/4}(p)\subset U$ for some radii $r_j\to 0$. Now fix one such large enough $j$. By  Lemma \ref{lem:quantitativeIinftyK} 
there is a function $\theta_j\in \I^{\infty,>0}(K_j)$, equal to $1$ outside $B_{r_j/2}(p)$, such that
\[
\norm{\partial^\alpha \theta_j}_\infty\leq C_\alpha r_j^{-|\alpha|}  \quad\forall\;\alpha\in\N_0^n,
\]
with constants $C_\alpha>0$ independent of $j$.  Define $g_j:=(1-\theta_j)f$.
 Then $\supp g_j\subset \overline B_{r_j/2}(p)$ and the sequence $g_j$ converges to zero with respect to every 
$C^m$-seminorm on $\overline B_{r_j/2}(p)$.  Indeed, for $|\beta|\leq  m$, the product rule gives
\[
\partial^\beta g_j  =\sum_{\mu\leq \beta}\left(\begin{matrix}
\beta\\
\mu
\end{matrix}\right) \partial^\mu(1-\theta_j)\partial^{\beta-\mu}f .
\]
On $B_{r_j/2}(p)$, the vanishing of $f$ to infinite order at $p$ implies for each $\nu\in \N_0^n$  that $
\sup_{B_{r_j/2}(p)} |\partial^\nu f|=\mathcal O(r_j^L)$ for all $L\geq 0$.  
For every $m\in \N_0$, taking $L>m$ and using
$|\partial^\mu \theta_j|\leq  C_\mu r_j^{-|\mu|}$, we obtain
$\max_{|\beta|\leq  m}\Vert\partial^\beta g_j\Vert_\infty\to 0$. 
Consequently, for every $s\in \R$,
\bq
\|g_j\|_{H^s(\R^n)}\longrightarrow 0\label{eq:gjtozero}
\eq
after extending $g_j$ by zero to all of $\R^n$.  This follows from the fact that the inclusion $\CT(B_{R}(p))\to H^s(\R^n)$ is continuous, where $R$ is chosen such that $R>r_j$ for all $j$. 

Recall from the proofs of Proposition \ref{prop:control} and Lemmas \ref{lem:SobolevdomHilbert}, \ref{lem:microdomHilbert} that the linear map  $J_{\mathsf{Q},p,\xi}:\Cinft(M)\to \mathcal{H}_\mathsf{Q}$ is a direct sum of individual linear maps of the form  $J^{-s}_{l}:\Cinft(M)\to \mathcal{H}_l^{-s}$ 
for Sobolev seminorms $\norm{\chi_l\cdot}_{H^{s}(M_l)}$, $l\in \N$, $s\in \R$, and linear maps of the form 
\[
J_{q,W}:\Cinft(M)\to \mathcal H_\mathrm{micro}=H^{\frac{\dim M}{2}+1}(\R^{\dim M})
\]
 for microlocal seminorms $q=\|\cdot\|^\Gamma_{U,N,V,\chi}\in \mathsf{SEM}(\Gamma,s)$ as in \eqref{eq:microlocalseminorms} and closed cones $W\subset(\R^{n}\setminus \{0\})$ with $W\cap V=\emptyset$. It therefore suffices to estimate the images of the functions $g_j$ under each of these linear maps in their respective Hilbert norms.

We begin with the Sobolev seminorms. The map $J^{-s}_l$ was defined in \eqref{eq:Jdef}:
\begin{align*}
J^{-s}_l:\Cinft(M)&\to \mathcal{H}^{-s}_l\\
f &\mapsto \big((I+\Delta)^{s}f_{1,l},\ldots,(I+\Delta)^{s}f_{N_l,l}\big),
\end{align*}
where $f_{1,l},\ldots,f_{N_l,l}\in \CT(\R^n)$ are the local pullbacks of $f$ multiplied by chart cutoffs, defined in \eqref{eq:localfkl}, and 
$\mathcal{H}^{-s}_l=(H^{-s}(\R^n))^{\oplus N_l}$. Since for each $m\in \{1,\ldots,N_l\}$ and all $s\in \R$ the map
\begin{align*}
H^s(\R^n)\to H^s(\R^n),\qquad f \mapsto f_{m,l}
\end{align*}
is continuous (according to \eqref{eq:localfkl} it is given by the pullback along a diffeomorphism followed by multiplication with a cutoff) and  $(I+\Delta)^{s}: H^s(\R^n)\to H^{-s}(\R^n)$ 
is continuous, it follows from \eqref{eq:gjtozero} that $
\|J^{-s}_l(g_j)\|_{\mathcal{H}^{-s}_l}\to 0$. 
It remains to consider the  linear maps for the microlocal seminorms  constructed in the proof of Lemma \ref{lem:microdomHilbert}:
\[
J_{q,W} (g_j) =\vartheta(\xi)\langle\xi\rangle^N\widehat{\chi g_j}(\xi)    \in H^{\frac{n}{2}+1}(\R^n).
\]
If $p\notin\supp \chi$, then $\chi g_j=0$ for all large $j$.  Otherwise we can assume w.l.o.g.\ that $\supp \chi \subset U$ because for large enough $j$ the function $g_j$ is supported in $U$. Let $m\in \N$ be such that $2m\geq \frac{n}{2}+1$. Then we estimate
\begin{align*}
\norm{J_{q,W} (g_j)}_{H^{n/2 +1}(\R^n)}&\leq \norm{J_{q,W} (g_j)}_{H^{2m}(\R^n)}\\&=\norm{(I+\Delta)^{m}J_{q,W} (g_j)}_{L^2(\R^n)}\\
&\leq C_{\vartheta} \sum_{\substack{\gamma\in \N_0^n:\\|\gamma|\leq  2m}}\|\langle\xi\rangle^N \partial_\xi^\gamma\widehat{\chi g_j}\|_{L^2(\R^n)}
\end{align*}
because all partial derivatives of the smooth function $\xi\mapsto \vartheta(\xi)\langle\xi\rangle^N$ are bounded by a constant times $\langle\xi\rangle^N$, depending on the angular cutoff function $\vartheta$. Now we use that $|\partial_\xi^\gamma\widehat{\chi g_j}(\xi)|=|\widehat{x^\gamma\chi g_j}(\xi)|$, 
where we wrote $x\in \R^n$ for the variable dual to $\xi\in \R^n=(\R^n)^\ast$. Recalling the definition \eqref{eq:innerproductHs} of the inner product on $H^{N}(\R^n)$, we arrive at 
\[
\norm{J_{q,W} (g_j)}_{H^{n/2 +1}(\R^n)}\leq C_{\vartheta} \sum_{\substack{\gamma\in \N_0^n:\\|\gamma|\leq  2m}}\|x^\gamma\chi g_j\|_{H^N(\R^n)}.
\]
Applying \eqref{eq:gjtozero} with $f$ replaced by $x^\gamma \chi f\in \I^\infty(\{p\})$ in the steps before, we find that
\[
\|x^\gamma \chi g_j\|_{H^N(\R^n)}\longrightarrow0   \quad\forall\;\gamma\in \N_0^n.
\]
This shows that $\norm{J_{q,W} (g_j)}_{H^{n/2 +1}(\R^n)}\to 0$.

Since all the finitely many components of $J_{\mathsf{Q},p,\xi}(g_j)$ tend to zero in their respective Hilbert spaces, it follows that $\rho_{\mathsf{Q},p,\xi}(g_j)=\norm{J_{\mathsf{Q},p,\xi}(g_j)}_{\mathcal{H}_\mathsf{Q}}\to 0$, finishing the proof.
\end{proof}

\subsection{Approximation near compact sets}

The following proposition is a technical key ingredient to the main proofs in Section \ref{sec:mainproofs}. It allows us to approximate the constant function $1$ in the relevant seminorms by functions that vanish to infinite order on a small prescribed  compact set. 
\begin{proposition}\label{prop:finiteflat} Let $\Gamma\subset \dot T^\ast M$ be closed and conical, let $p\in M$ be such that $\Gamma_p\neq \emptyset$, let $s\in (-\infty,\dim M/2]$, and let $\eps>0$.  Furthermore, let $\mathsf{Q}\subset \mathsf{SEM}(\Gamma,s)$ be a finite set. Then there is a neighborhood $U\subset M$ of $p$ in which every compact set $K\subset U$ admits  a function $m\in\I^{\infty}(K)$  such that
\bq
q(m-1)<\eps\quad \forall\,q\in \mathsf{Q},\label{eq:claim1}
\eq
and $m^{-1}(\{0\})\cap (M\setminus K)$ has empty interior.

Moreover, if $\Gamma_{p}\cap (-\Gamma_{p})\ne\emptyset$  holds, we can achieve the above  with $m\in\I^{\infty,>0}(K)$ (in particular, then $m^{-1}(\{0\})\cap (M\setminus K)=\emptyset$).
\end{proposition}
\begin{proof}Choose some $\xi\in \Gamma_p$, let $J_{\mathsf{Q},p,\xi}:\Cinft(M)\to {\mathcal H}_{\mathsf{Q}}$ be the linear map from Proposition \ref{prop:control} and write for simplicity 
$\mathcal H:=\mathcal H_{\mathsf Q}$, $J:=J_{\mathsf Q,p,\xi}$. 
By Proposition \ref{prop:control} (1a), up to making $\eps$ smaller, we can replace \eqref{eq:claim1}  by 
\bq
\|J(m-1)\|_{{\mathcal H}}<\eps.\label{eq:Jclaim}
\eq
Suppose that \eqref{eq:Jclaim} does not hold. Then there is a sequence of compact sets $K_j\subset M$ such that every neighborhood of $p$ in $M$ contains almost all $K_j$, and for some $\delta>0$
\bq
\|J(m)-J(1)\|_{\mathcal H}\geq \delta\quad  \forall\; m\in\I^{\infty}(K_j),\;j\in \N. \label{eq:Jdeltadistance1}
\eq
Given any  $f\in \I^{\infty}(\{p\})$, Lemma \ref{lem:flat-cutoff-small} provides functions $\theta_j\in \I^{\infty,>0}(K_j)$, $j\in \N$, with $
\|J(\theta_jf)-J(f)\|_{\mathcal H}\stackrel{j\to \infty}{\longrightarrow}0$.  
Observing that $\theta_jf\in \I^{\infty}(K_j)$, we see that \eqref{eq:Jdeltadistance1} implies 
\bqn 
\|J(f)-J(1)\|_{\mathcal H}\geq \delta/2\quad    \forall\;f\in\I^{\infty}(\{p\}).
\eqn 
Consider the closed subspace
\[
E:=\overline{J(\I^{\infty}(\{p\}))}\subset \mathcal H.
\]
Let $x_0\in E$ be the orthogonal projection of $J(1)$ onto $E$ and put $v:=J(1)-x_0$, so that $\langle x,v\rangle_{\mathcal H} =0$ for all $x\in E$.  Define
\[
L(f):=\langle J(f),v\rangle_{\mathcal H},\qquad f\in C^\infty(M).
\]
Proposition \ref{prop:control} (1b) says that $L$ is a
compactly supported distribution. It satisfies 
\[
L(1)=\langle J(1),v\rangle_{\mathcal H}=\langle J(1)-x_0,v\rangle_{\mathcal H}+\langle x_0,v\rangle_{\mathcal H}=\|v\|_{\mathcal H}^2\geq \delta^2/4,
\]
 while  $L(f)=0$ for all $f\in\I^{\infty}(\{p\})$, which implies  that $\supp L\subset \{p\}$. Thus $L$ is a sum of derivatives of the Dirac distribution at $p$ and non-zero since $L(1)\neq 0$, which implies that $L$ is not microlocally in $H^{s}$ at any $\xi_0\in \dot T_p^\ast M$ if $s\geq -\frac{1}{2}\dim M$.  However, since all Sobolev seminorms in $\mathsf Q$ (if any) are of regularity $\leq\frac{1}{2}\dim M$, Proposition \ref{prop:control} (1c) says that $L$ is microlocally in $H^{-s_0}$ at $-\xi$, where $s_0\leq \frac{1}{2}\dim M$ and hence $-s_0\geq -\frac{1}{2}\dim M$, a contradiction. This finishes the proof of \eqref{eq:Jclaim} and hence of \eqref{eq:claim1}. 

To prove the empty interior claim, first choose $m_0\in\I^{\infty}(K)$ with 
$q(m_0-1)<\eps/2$ for all $q\in \mathsf Q$. Then choose any $g\in\I^{\infty,>0}(K)$.  For all small enough $z\in \C$ one has $q(m_0+zg -1)<\eps$  for all $q\in \mathsf Q$.  Let $\{\mathscr U_i\}_{i\in \N}$ be a countable basis  of the topology of $M\setminus K$. Then for each $i$ there is at most one value of $z$ for which $m_0+zg$ vanishes
identically on $\mathscr U_i$.  Choosing $z$ different from these countably many values and putting $m:=m_0+zg$ gives $m^{-1}(\{0\})\cap (M\setminus K)$ empty interior.

Let us now prove the strengthened claim. When $\Gamma_{p}\cap (-\Gamma_{p})\ne\emptyset$, we can choose $\xi \in \Gamma_{p}\cap (-\Gamma_{p})$, and Proposition \ref{prop:control}(2) yields that $J(C^\infty(M,\R))\subset \mathcal H$ is real. 

We argue similarly as for the first claim, again replacing \eqref{eq:claim1} by \eqref{eq:Jclaim}, but this time with $m\in\I^{\infty,>0}(K)$. Suppose that the  statement does not hold. Then  there is a sequence of compact sets $K_j\subset M$ such that every neighborhood of $p$ in $M$ contains almost all $K_j$, and for some $\delta>0$ 
\bq 
\|J(m)-J(1)\|_{\mathcal H}\geq \delta\quad \forall\;m\in\I^{\infty,>0}(K_j),\;j\in \N.\label{eq:Jdeltadistance3}
\eq 
Given any $f\in \I^{\infty,>0}(\{p\})$, Lemma \ref{lem:flat-cutoff-small} provides functions $\theta_j\in \I^{\infty,>0}(K_j)$, $j\in \N$, with $
\|J(\theta_jf)-J(f)\|_{\mathcal H}\stackrel{j\to \infty}{\longrightarrow}0$. 
Observing that $\theta_jf\in \I^{\infty,>0}(K_j)$, we see that \eqref{eq:Jdeltadistance3} implies
\bqn 
 \|J(f)-J(1)\|_{\mathcal H}\geq \delta/2\quad  \forall\;f\in\I^{\infty,>0}(\{p\}).
\eqn 
Consider the closed convex real cone \[
\mathcal{C}:=\overline{J(\I^{\infty,>0}(\{p\}))}\subset \mathcal H.
\]
Since $\mathcal{C}$ is a closed convex subset of the Hilbert space $\mathcal H$,  there is an element $x_0\in \mathcal{C}$ with $\|J(1)-x_0\|_{\mathcal H}=\mathrm{dist}(J(1),\mathcal{C})\geq \delta/2>0$. Setting, $v:=J(1)-x_0$, there is an affine  hyperplane  $\{\Re \langle \cdot,v \rangle_{\mathcal H}=\alpha\}\subset \mathcal H$ strictly separating $J(1)$  and $\mathcal{C}$ (for example, the value $\alpha:=\Re \langle x_0,v \rangle_{\mathcal H} + \frac{1}{2}\|v\|_{\mathcal H}^2$ works). That is, we have
\bqn
  \Re \langle J(1),v\rangle_{\mathcal H}>\alpha,\qquad      \Re \langle x,v\rangle_{\mathcal H}< \alpha\quad\forall\;x\in\mathcal{C}.
\eqn
Since in this statement the vectors $J(1)$, $v$ and $x$ are real, we can  drop the real parts.  Define $L(f):=\langle J(f),v\rangle_{\mathcal H}$ for  $f\in C^\infty(M)$. By Proposition \ref{prop:control} (1b), $L\in \E'(M)$. Moreover, $L$ is a real distribution with $L(1)>\alpha$ and  $L(f)<\alpha$ for all $f\in\I^{\infty,>0}(\{p\})$. For such $f$ and all $\eps>0$ one has $L(\eps f)<\alpha$, so continuity and linearity of $L$ imply  $\alpha\geq 0$ and $L(f)\leq 0$. In particular, $L(1)>0$. Thus  Lemma \ref{lem:wflemma} applies to $L$. It states that, if $s\geq-\frac{1}{2}\dim M$, then $L$ is not microlocally in $H^s$ at any $\xi\in \dot T_p^\ast M$. As before, this contradicts Proposition \ref{prop:control} (1c).
\end{proof}

\section{Atomization}\label{sec:atomization}
We shall introduce a convenient language of \emph{atomizability} to prove our main results. 
\begin{definition}\label{def:tame-atomizable}
Let $\Lambda\subset\dot T^*M$ be a closed and conical set. We say that a family  $\cA\subset\E'_\Lambda(M)\setminus\{0\}$ of non-zero distributions \emph{atomizes} $\Lambda$ if:
\begin{enumerate}[label=(\roman*)]
\item For every $A\in \cA$, $\supp A\subset M$ has empty interior in $M$. 
\item For every non-empty open set $U\subset M$ there is an $A\in\cA$ with $\supp A\subset U$.
\item $\cA$ has uniformly bounded Sobolev order (Definitions \ref{def:uniformorder}, \ref{def:uniformsminus0}).
\end{enumerate}
If such a family $\cA$ exists, we say that $\Lambda$ is \emph{atomizable} and that $\cA$ is an \emph{atomizing} family for $\Lambda$. The distributions $A$ in an atomizing family $\cA$ are called \emph{atoms}. 

If all atoms $A\in\cA$ are positive distributions, we call $\mathcal A$ a \emph{positive} atomizing family. If $\Lambda$ admits a positive atomizing family, we say that $\Lambda$ is \emph{positively} atomizable.

If $s\in \R$ is such that all $A\in\cA$ satisfy $A\in H^s_\mathrm{comp}(M)$ (resp.\ $A\in H^{s-0}_\mathrm{comp}(M)$), then we call $\mathcal A$ an atomizing family \emph{of Sobolev order $s$} (resp.\ $s-0$). If $\Lambda$ admits such an $\cA$, we say that $\Lambda$ is \emph{$H^s$-atomizable} (resp.\ \emph{$H^{s-0}$-atomizable}).
\end{definition}
By the above and Definition \ref{def:uniformorder}, every atomizable conical set $\Lambda$ is $s$-atomizable for some $s\in \R$. Also, being $H^{s-0}$-atomizable is \emph{not} equivalent to being $H^{s-\eps}$-atomizable for every $\eps>0$ (which would allow a different atomizing family for each $\eps>0$). Furthermore, note that if a closed conical set $\Lambda\subset\dot T^*M$ has any of the atomizability properties of Definition \ref{def:uniformorder} and $\Lambda'\subset\dot T^*M$ is another closed conical set with $\Lambda\subset \Lambda'$, then $\Lambda'$ inherits all the atomizability properties of $\Lambda$ because $\E'_\Lambda(M)\subset\E'_{\Lambda'}(M)$. 

\begin{remark}\label{rem:atomiz-nonemptyfiber} \begin{enumerate}[leftmargin=*]
\item An atomizable set $\Lambda\subset\dot T^*M$ satisfies $\Lambda_p\neq \emptyset$ for all $p\in M$:  Otherwise, the atoms would be smooth on some non-empty open set $U\subset M$, which makes it impossible to have a non-empty support in $U$ with empty interior.

\item A positively atomizable set $\Lambda\subset\dot T^*M$ satisfies $\Lambda_p\cap (-\Lambda_p)\neq \emptyset$ for all $p\in M$: This follows from (1) and \eqref{eq:Fourierconjug}.
\end{enumerate} 

\end{remark}

\begin{example}\label{ex:maximal}
The maximal conical set $\Lambda=\dot T^*M$ is positively $H^{-\dim M/2-0}$-atomizable: For any dense set $Q\subset M$, the family $\cA_Q:=\{\delta_p\,|\, p\in Q\}\subset H^{-\dim M/2-0}_\mathrm{comp}$ positively atomizes  $\Lambda$, where  $\delta_p$ is the Dirac distribution centered at $p$. 
\end{example}

In contrast to the trivial maximal case from Example \ref{ex:maximal}, the following crucial intermediate result says that every ``nowhere-maximal'' closed conical set $\Gamma$  admits a ``nowhere-maximal'' atomizable conical set $\Lambda$ such that $\Gamma\cap (-\Lambda)=\emptyset$. 
\begin{proposition}\label{prop:existence1}

\begin{enumerate}[leftmargin=*]
\item Suppose that $\dim M=1$. Let $\Gamma\subset \dot T^*M$ be a closed conical set such that
\bq
\Gamma=\emptyset \qquad \text{or}\qquad \emptyset\neq\Gamma_p\neq \dot T^*_pM\quad \forall\; p\in M.\label{eq:nonmaximalconddim1}
\eq 
Then there is an $H^{-1/2-0}$-atomizable conical set $\Lambda\subset\dot T^*M$ with $
\Gamma\cap (-\Lambda)=\emptyset$ and
\bq
\Lambda_p\neq \dot T^*_pM\quad \forall\; p\in M.\label{eq:nonmaximalcondclaim}
\eq 

\item Suppose that $\dim M\geq 2$ and let $\Gamma\subset \dot T^*M$ be a closed conical set. 
\begin{enumerate}
\item If $\Gamma$ satisfies
\bq
\Gamma_p\neq \dot T^*_pM\quad \forall\; p\in M,\label{eq:nonmaximalcond}
\eq 
then there is an $H^{-1/2-0}$-atomizable conical set $\Lambda\subset\dot T^*M$ satisfying $\Gamma\cap (-\Lambda)=\emptyset$ and \eqref{eq:nonmaximalcondclaim}.
\item If $\Gamma$ satisfies
\bq
\Gamma_p\cup -\Gamma_p\neq \dot T^*_pM\quad \forall\; p\in M,\label{eq:nonmaximalcondminus}
\eq 
then there is a positively $H^{-1/2-0}$-atomizable symmetric conical set $\Lambda=-\Lambda\subset\dot T^*M$ satisfying $\Gamma\cap (-\Lambda)=\emptyset$ and \eqref{eq:nonmaximalcondclaim}.
\end{enumerate}
\end{enumerate}
\end{proposition} 
Note that, in all positive dimensions, Proposition \ref{prop:existence1}, applied with $\Gamma=\emptyset$, gives an $H^{-1/2-0}$-atomizable conical set $\Lambda\subset\dot T^*M$ which is ``nowhere-maximal'' in the sense of \eqref{eq:nonmaximalcondclaim}. Moreover, if $\dim M\geq 2$, the obtained set $\Lambda$ is positively $H^{-1/2-0}$-atomizable.

The proof of Proposition \ref{prop:existence1} will be given on p.~\pageref{proof:existence1} after some preparations that allow us to deal with the case $\dim M=1$, in which the Hörmander wavefront set condition and the condition \eqref{eq:nonmaximalcond} are much more rigid than in higher dimensions. 

\begin{remark}If $\dim M=1$, then the maximal set $\Lambda=\dot T^*M$ is the only closed conical subset of $\dot T^*M$ that is positively atomizable, by  Remark \ref{rem:atomiz-nonemptyfiber}. In Proposition \ref{prop:existence1} it is therefore a priori impossible to extend the positive atomizability  to dimension $1$.
\end{remark}

The following proposition reformulates a remarkable result by Kozma-Olevski\u{\i} on Fourier null-series \cite{KozmaOlevskii} in distribution language. As mentioned in the AI use disclaimer on p.~\pageref{sec:Aiuse}, the relevance of Kozma-Olevski\u{\i}'s work to the present paper was discovered by ChatGPT in a deep literature search.

\begin{proposition}[Kozma-Olevski\u{\i} \cite{KozmaOlevskii}]\label{prop:KOatom}
On the circle $S^1$, there exists a non-zero distribution $u\in \D'(S^1)$, supported on a Lebesgue null set, whose Fourier coefficients $\widehat u(n)$, $n\in \Z$, satisfy
\begin{align}\begin{split}
\widehat u(n)&=\mathcal O(|n|^{-\infty})\quad\text{ as }n\to-\infty,\\
\widehat u(n)&=o(1)\hspace*{3.35em}\text{ as }n\to+\infty.\label{eq:coeffestimun}\end{split}
\end{align}
 In particular, $\WF(u)$ is contained in one of the two connected components of $\dot T^\ast S^1$, and $u\in H^{-1/2-0}(S^1)$.
\end{proposition}
\begin{proof}The claim is explicitly stated in \cite[end of p.~1045]{KozmaOlevskii} as a consequence of \cite[Thm.~3', Eq.~(17), Lemmas 3--6]{KozmaOlevskii}. We shall briefly  summarize the argument: The aforementioned reference establishes the existence of a sequence $c(n)\in \C$, $n\in \Z$, with $c(n)=\mathcal O(|n|^{-\infty})$ as $n\to -\infty$ and $c(n)=o(1)$ as $n\to +\infty$ such that the associated Fourier series $\sum_{n\in \Z}c(n)z^n$ converges to zero pointwise for all $z\in S^1$ outside a compact Lebesgue null set $N\subset S^1$ (in the notation of \cite[Sec.~3.5]{KozmaOlevskii}, $N=(K')^\circ$). More precisely, in \cite[Sections 3.5, 3.6]{KozmaOlevskii} Kozma-Olevski\u{\i} obtain the distribution $u\in \D'(S^1)$ with $\widehat u(n)=c(n)$ for all $n\in \Z$ as the difference $u=f-T_F$, where $f\in \Cinft(S^1)$ and $T_F\in \D'(S^1)$ is the boundary value of a holomorphic function $F$ on the open unit disk $\mathbb{D}\subset \C$ with Taylor series $F=\sum_{n\geq 0}\widehat F(n)z^n$, where $\widehat F(n)=\widehat f(n)-c(n)$. That is, $\widehat T_F(n)=\widehat{F}(n)$ for $n\geq 0$ and $\widehat T_F(n)=0$ for $n<0$.  In \cite[Lem.~3~(i)]{KozmaOlevskii} it is shown that $F$ extends continuously from $\mathbb{D}$ to $\mathbb{D}\cup (S^1\setminus N)$, the resulting function coinciding with $f$ on $S^1\setminus N$. Thus, if a test function $\varphi\in \CT(S^1)$ is supported outside of $N$, we have
\[
\eklm{u,\varphi}=\eklm{f,\varphi}-\eklm{T_F,\varphi}=\eklm{f,\varphi}-\eklm{f,\varphi}=0,
\]
proving that $\supp u\subset N$ since $N$ is closed in $S^1$. 

Given an $\eps>0$, the estimates \eqref{eq:coeffestimun} imply that there exists $C_u>0$ such that
\[
\sum_{n\in \Z}\eklm{n}^{-1-2\eps}|\hat u(n)|^2\leq C_u\sum_{n\in \Z}\eklm{n}^{-1-2\eps}<\infty,
\]
which implies that $u\in H^{-1/2-0}(S^1)$.

Finally, if $\eta\in \Cinft(S^1)$, then the Fourier coefficients of $\eta u$ are the convolution of the rapidly decreasing sequence $\hat \eta$ with $\hat u$, so \eqref{eq:coeffestimun} implies that they decay rapidly in the negative direction. This implies that $\WF(u)$ is contained in the connected component of $\dot T^\ast S^1$ that is chosen as the positive one.
\end{proof}
This remarkable result has a very useful consequence:
\begin{corollary}\label{cor:integrable}
Let $\Lambda\subset\dot T^*M$ be a closed conical set that locally contains the image of an integrable smooth $1$-form, i.e., around each point of $M$ there is an open set $U$ and a submersion $f:U\to\R$ with $\tau df|_x\in \Lambda$ for all $x\in U$, $\tau>0$. 
Then $\Lambda$ is $H^{-1/2-0}$-atomizable.
\end{corollary}
\begin{proof}To define atoms supported in an arbitrary non-empty open set $O\subset M$, we take a set $U\subset O$ as in the hypothesis, which we further shrink to a chart with coordinates $(t,y)$ and $t=f$. The latter is possible because $f$ is a submersion. Choose a small interval $I$ in the $t$-coordinate range and a function $\eta\in \CT(S^1)$, supported in an arc that is mapped into $I$ by an orientation-preserving diffeomorphism, such that $\eta u\neq 0$, where $u\in \D'(S^1)$ is as in Proposition \ref{prop:KOatom}.  We can then regard $\eta u$ as a non-zero compactly supported distribution $u_I$ on $I$.  The multiplication by $\eta$ and the pullback to $I$ preserve the one-sided wavefront set property, so that
\[
\WF(u_I)\subset\{(t,\tau dt)\,|\,\tau>0\}.
\]
Choose a non-zero cutoff function $\rho$ in the variable $y$ such that $I\times\supp \rho\subset U$, and define $A\in \E'(U)$ by $A=u_I\otimes \rho$.  Then $A\ne0$, $\supp A\subset \supp u_I\times\supp \rho\subset O$ has empty interior (because $\supp u_I\subset I$ has empty interior), and
\[
\WF(A)\subset\{((t,y),\tau dt)\,|\, (t,y)\in U,\tau>0\}=\{(x,\tau df|_{x})\,|\, x\in U,\tau>0\}\subset\Lambda.
\]
Since localizing to $I$, applying a diffeomorphism, and tensoring with $\rho$ preserve $H^{-1/2-0}$-regularity, we have $A\in H^{-1/2-0}_\mathrm{comp}(M)$, similarly as in Proposition \ref{prop:KOatom}. 
\end{proof}

We are now in the position to prove Proposition \ref{prop:existence1}.
\begin{proof}[{Proof of Proposition \ref{prop:existence1}}]\label{proof:existence1}
We begin with (1). Let $M_0\subset M$ be a connected component. Then $\dot T^\ast M_0$ has two connected components $\Lambda_+$ and $\Lambda_-$, both of which are closed conical sets. Since $\Gamma\subset \dot T^\ast M$ is closed and satisfies \eqref{eq:nonmaximalconddim1}, we have $\Gamma\cap \dot T^\ast M_0\subset\Lambda_+$ or $\Gamma\cap \dot T^\ast M_0\subset\Lambda_-$. Indeed: If $\Gamma=\emptyset$, both inclusions hold trivially, and if $\emptyset\neq \Gamma_p\neq \dot T_p^\ast M$ for all $p\in M_0$, then at each $p\in M_0$ the fiber $\Gamma_p$ equals either $(\Lambda_+)_p$ or $(\Lambda_-)_p$ and the sign cannot change as $p$ varies in $M_0$ because $\Gamma$ is closed. To see this, let $s_\pm:M_0\to\Lambda_\pm$ be continuous sections and consider the closed sets $(M_0)_\pm:=s_\pm^{-1}(\Gamma\cap \Lambda_\pm)$. Then $(M_0)_+\cup (M_0)_-=M_0$ and $(M_0)_+\cap (M_0)_-=\emptyset$, so the connectedness of $M_0$ implies that $(M_0)_+=M_0$ or $(M_0)_-=M_0$. Without loss of generality, suppose that $\Gamma\cap \dot T^\ast M_0\subset\Lambda_+$.  

Since every one-dimensional manifold is orientable, we can cover $M_0$ by charts whose sole coordinate $t$ satisfies $\tau dt|_x\in \Lambda_+$ for all $x\in M_0$, $\tau>0$, and thus $\Lambda_+$ is atomizable by Corollary \ref{cor:integrable}. Taking as $\Lambda\subset \dot T^\ast M$ the union of such sets $\Lambda_+$ containing the fibers of $\Gamma$ on  each connected component of $M$, we get that $\Lambda$ satisfies $\Gamma\cap (-\Lambda)=\emptyset$. Furthermore, $\Lambda$ satisfies the nowhere-maximality condition \eqref{eq:nonmaximalcondclaim} by construction: For each point $p\in M$ the fiber $\Lambda_p$ contains only one of the two connected components of $T_p^\ast M\setminus \{0\}$.  
 Thus (1) is proved. 

Now we prove (2a). Let $\{U_j\}_{j\in \N}$ be an open cover of $M$ subordinate to a locally finite atlas $\{\tilde U_j\}_{j\in \N}$ of $M$ consisting of relatively compact open sets $\tilde U_j\subset M$ such that $\overline{U}_j\subset \tilde U_j$ holds for all $j\in \N$; in particular $\overline{U}_j$ is compact. For each $j\in \N$, we identify $\tilde U_j$ with an open subset of $\R^n$ and identify $T^\ast \tilde U_j\equiv \tilde U_j\times \R^n$.  
Given a point $p_0\in \tilde U_j$, choose some $\xi\in \R^n\setminus \{0\}\equiv \dot T_{p_0}^\ast M$ with $(p_0,\xi)\not\in \Gamma$, this exists by \eqref{eq:nonmaximalcond}. Then there is an open set $\mathscr U\subset \tilde U_j$ around $p_0$ such that $(p,\xi)\not\in \Gamma$ for all $p\in \mathscr U$.  
To see this, suppose that it does not hold: Then there is a sequence  $\{p_k\}_{k\in \N}$ of points $p_k\in \tilde U_j$ such that $p_k\to p_0$ in $\tilde U_j$ and $(p_k,\xi)\in \Gamma$ for all $k\in \N$. Since $(p_k,\xi)\to (p_0,\xi)$ in $T^\ast \tilde U_j\equiv \tilde U_j\times \R^n$ but $(p_0,\xi)\not\in \Gamma$, this contradicts that $\Gamma$ is closed.

Note that, if two open sets $\mathscr U,\mathscr V\subset \tilde U_j$ are such that $\overline{\mathscr V}\subset\mathscr U$ and $\overline{\mathscr U}\subset \tilde U_j$, then $\Gamma\cap \overline{T^\ast \mathscr V}\subset \Gamma\cap T^\ast \mathscr U$, 
where  $\overline{T^\ast \mathscr V}\equiv \overline{\mathscr V}\times \R^n$ is the closure in $T^\ast \tilde U_j\equiv \tilde U_j\times \R^n$. 
Therefore, since $\overline U_j\subset \tilde U_j$ is compact, we can cover $\overline U_j$ by a finite family of open sets $
\mathscr U^1_j,\ldots,\mathscr U^{N(j)}_j\subset \tilde U_j$ 
around points $p_j^1,\ldots,p_j^{N(j)}\in \overline U_j$ for which elements $\xi_j^1,\ldots,\xi_j^{N(j)}\in \R^n\setminus\{0\}$ exist with 
\bq
(p,\xi_j^l)\not\in \Gamma \quad \forall\;p\in \overline{\mathscr U_j^l},\qquad 1\leq l\leq N(j),\label{eq:83593869023968036}
\eq
and such that $\overline{\mathscr U^l_j}\subset \tilde U_j$, $1\leq l\leq N(j)$; in particular each closure $\overline{\mathscr U^l_j}$ is compact. 

For all $j\in \N$, $1\leq l\leq N(j)$, define
\[
\Lambda_j^l:=\{(x,-\tau \xi_j^l)\,|\,\tau > 0,x\in \mathscr U_j^l\}\subset \dot T^\ast \mathscr U_j^l.
\]
Then $\Lambda_j^l$ is a closed conical subset of $\dot T^\ast \mathscr U_j^l$ and the closure of $\Lambda_j^l$ in $\dot T^\ast M$ is given by
\bq
\overline{\Lambda_j^l}=\{(x,-\tau \xi_j^l)\,|\,\tau > 0,x\in \overline{\mathscr U^l_j}\}.\label{eq:closurelambdajl}
\eq
Since the product $\overline{\mathscr U^l_j}\times S^{n-1}$ with the Euclidean unit sphere is compact and  \eqref{eq:83593869023968036} implies that $\overline{\Lambda_j^l}\cap(-\Gamma)\cap(\overline{\mathscr U^l_j}\times S^{n-1})=\emptyset$, it follows that
\bq
\overline{\Lambda_j^l}\cap(-\Gamma)=\emptyset.\label{eq:closureLambdaj}
\eq
Define the submersion $f^l_j:\mathscr U^l_j\to \R$, $x\mapsto - x\cdot \xi_{j}^l$, where $\cdot$ is the Euclidean inner product. 
Let now $\{{\mathscr U}_i\}_{i\in \N}:=\{{\mathscr U}_{p_i}\}_{i\in \N}$ be a basis of the topology of $M$ subordinate to the open cover $\{\mathscr U_j^l\,|\,j\in \N,1\leq l\leq N(j)\}$, this exists by second-countability.  To each ${\mathscr U}_i$, assign a unique ``parent set'' $\mathscr U_{j(i)}^{l(i)}$, for example by the rule
\begin{align*}
j(i)&:=\min\{j\in \N\,|\,\exists\; l\in\{1,\ldots,N(j)\}: {\mathscr U}_i\subset \mathscr U_j^l\},\\
 l(i)&:=\min\{l\in\{1,\ldots,N(j(i))\}\,|\,{\mathscr U}_i\subset\mathscr U_{j(i)}^l\}.
\end{align*}
For each $i\in \N$ and all $x\in \mathscr U_i$, $\tau>0$, the closed conical set $\Lambda_{j(i)}^{l(i)}\cap \dot T^\ast \mathscr U_i\subset \dot T^\ast \mathscr U_i$ contains $\tau df|_x$, where $f:=f_{j(i)}^{l(i)}|_{\mathscr U_i}:\mathscr U_i\to \R$ is a submersion. Thus Corollary \ref{cor:integrable} says that $\Lambda_{j(i)}^{l(i)}\cap \dot T^\ast \mathscr U_i$ is $H^{-1/2-0}$-atomizable, in particular there is a non-zero distribution $A_i\in \E'(M)$, supported in $\mathscr U_i$, with $\WF(A_i)\subset \Lambda_{j(i)}^{l(i)}\cap \dot T^\ast \mathscr U_i$ and $A_i\in H^{-1/2-0}_\mathrm{comp}(M)$. Define
\[
\Lambda:=\overline{\bigcup_{j\in \N}\bigcup_{1\leq l\leq N(j)}\Lambda^l_j}=\bigcup_{j\in \N}\bigcup_{1\leq l\leq N(j)}\overline\Lambda^l_j\quad \subset \dot T^\ast M
\]
where the closures are taken in $\dot T^\ast M$ and the equality holds because the atlas $\{\tilde U_j\}_{j\in \N}$ is locally finite and for each $j$ one takes the union over only finitely many $l\in \{1,\ldots,N(j)\}$. The closed set $\Lambda$ is clearly conical and satisfies $\Lambda\supset \WF(A_i)$ for all $i\in \N$. Since every open subset of $M$ contains some $\mathscr U_i$, this proves that  $\Lambda$ is $H^{-1/2-0}$-atomizable. 
Furthermore, \eqref{eq:closureLambdaj} implies that $\Gamma\cap (-\Lambda)=\emptyset$. Finally, for each $p\in M$, \eqref{eq:closurelambdajl}, the inclusion $\overline{\mathscr U^l_j}\subset \tilde U_j$ and the local finiteness of  the atlas $\{\tilde U_j\}$ imply that the fiber $\Lambda_p$ is contained in the union of finitely many $1$-dimensional rays in $T_p^\ast M$. Thus $\Lambda_p\neq \dot T_p^\ast M$ since $\dim M\geq 2$, showing that  \eqref{eq:nonmaximalcondclaim} holds. So (2a) is proved.

Let us now turn to (2b). If $\Gamma_p\cup (-\Gamma_p)\neq \dot T^*_pM$ holds for all $p\in M$, we modify the above construction as follows:  Instead of \eqref{eq:83593869023968036}, we achieve that
\bqn
(p,\xi_j^l),(p,-\xi_j^l)\not\in \Gamma \quad \forall\;p\in \overline{\mathscr U_j^l},\qquad 1\leq l\leq N(j),
\eqn
and define
\[
\Lambda_j^l:=\{(x,\tau \xi_j^l)\,|\,\tau \in \R\setminus \{0\},x\in \mathscr U_j^l\}\subset \dot T^\ast \mathscr U_j^l.
\]
The resulting closed conical set $\Lambda\subset \dot T^\ast M$ then satisfies $\Lambda=-\Lambda$ and $\Lambda\cap (-\Gamma)=\emptyset$.

For each $i\in \N$, we can construct a non-zero positive distribution $A_i\in \E'(M)$ with $\supp A_i\subset \mathscr U_i$ and $\WF(A_i)\subset \Lambda^{l(i)}_{j(i)}$ as follows: Choose a point $x_i\in \mathscr U_i\subset \R^n$ and some $\delta>0$ such that $
B^n_\delta(x_i)\subset \mathscr U_i$. Denote by $
V_i:=(\xi^{l(i)}_{j(i)})^\perp\subset \R^n$ 
 the Euclidean orthogonal complement and equip it with the unique Lebesgue measure $dy$ such that $\mathrm{vol}_{dy}(B^n_\delta(0)\cap V_i)=\mathrm{vol}_{dx} B^{n-1}_\delta(0)$, where $dx$ is the standard Lebesgue measure on $\R^{n-1}$. Choose a non-negative real-valued cutoff function $\chi_i\in \CT(V_i,\R)$ with $\chi_i(0)=1$ and $\supp \chi_i\subset  B^n_\delta(0)\cap V_i$.   Then define $A_i\in \E'(M)$ by
\[
A_i(f):=\int_{V_i} f(x_i+y)\chi_i(y)\,d y,\qquad f\in \Cinft(M).
\] 
By construction, $\supp A_i\subset B^n_\delta(x_i)\cap (x_i+V_i)\subset\mathscr U_i\cap (x_i+V_i)$. The  hypersurface $\mathscr U_i\cap (x_i+V_i)\subset M$ has codimension $1$, so  $\supp  A_i$ has empty interior in $M$. By Example \ref{ex:codimSobolev}, $A_i$ has Sobolev order $-\frac{1}{2}-0$. By construction, each $A_i$ is non-zero, positive, supported in $\mathscr U_i$, and 
\[
\WF(A_i)\subset \dot N^\ast ((x_i+V_i)\cap\mathscr U_i)\subset \mathscr U_i\times (V_i^\perp\setminus\{0\})= \mathscr U_i\times(\R_\times\xi^{l(i)}_{j(i)})\subset \Lambda_{j(i)}^{l(i)}.
\]
Again, every open subset of $M$ contains some $\mathscr U_i$. Hence $\Lambda$ is positively  $H^{-1/2-0}$-atomizable. We already noted above that $\Lambda\cap (-\Gamma)=\emptyset$. Finally, for each $p\in M$, the fiber $\Lambda_p$ is contained in the union of finitely many $1$-dimensional lines in $T_p^\ast M$ by  the same argument as in the proof of (2a), which again implies $\Lambda_p\neq \dot T_p^\ast M$ since $\dim M\geq 2$ is assumed. Thus also  \eqref{eq:nonmaximalcondclaim} holds and the proof is finished.
\end{proof}

\begin{corollary}\label{cor:existence1}
Let $\Upsilon\subset \dot T^*M$ be an open conical set with $\Upsilon_p\neq \emptyset$ for all $p\in M$ (implying $\dim M>0$). Then there is an $H^{-1/2-0}$-atomizable conical set $\Lambda\subset\Upsilon$. 

Moreover, if  $\Upsilon_p\cap (-\Upsilon_p)\neq \emptyset$ for all $p\in M$, then there is a positively $H^{-1/2-0}$-atomizable symmetric closed conical set $\Lambda=-\Lambda\subset\Upsilon$. 
\end{corollary} 
\begin{proof}
First, suppose that $\dim M=1$. In this case, we cannot directly apply Proposition \ref{prop:existence1} (since we do not assume that $\Upsilon_p\neq \dot T^\ast_p M$ for all $p\in M$) and instead argue locally:  Decompose $\dot T^\ast M=\Gamma^+ \sqcup\Gamma^-$  such  that for every connected component $M_0\subset M$ the intersections $\dot T^\ast M_0\cap \Gamma^\pm$ are the two connected components of $\dot T^\ast M_0$. Define $M_\pm:=\{p\in M\,|\, \Gamma^\pm_p\subset \Upsilon_p \}$. 
Since $\Upsilon\subset \dot T^\ast M$ is an open conical set, the sets $M_\pm\subset M$ are open, and since $\Upsilon_p\neq \emptyset$ for all $p\in M$, one has $M=M_+\cup M_-$. Let $\{U_j\}_{j\in \N}$,  $\{V_j\}_{j\in \N}$ be locally finite families of open subsets of $M$ such that $\{U_j\}_{j\in \N}$ covers $M$, for each $j\in \N$ one has $\overline U_j\subset V_j$, and there is a sequence of signs $\sigma_j\in \{+,-\}$ such that $V_{j}\subset M_{\sigma_j}$ for all $j\in \N$. 

Fix some $j\in \N$ and consider the closed conical set $\Gamma_j:= \Gamma^{\sigma_j}|_{V_{j}}\subset \dot T^\ast V_{j}$. For every $p\in V_{j}$ it satisfies $\emptyset\neq(\Gamma_j)|_p\neq \dot T_p^\ast V_{j}$, so Proposition \ref{prop:existence1}(1) gives an $H^{-1/2-0}$-atomizable conical set $\Lambda_j\subset \dot T^\ast V_{j}$ such that $(\Lambda_j)_p\neq \dot T^\ast_p V_{j}$ for all $p\in V_{j}$ and $\Lambda_j\cap (-\Gamma_j)=\emptyset$. The atomizability implies that $(\Lambda_j)_p\neq \emptyset$ for all $p\in V_j$, which allows us to conclude that 
\bq
\Lambda_j = \Gamma^{\sigma_j}|_{V_{j}}.\label{eq:intersectionempty3636}
\eq
Now define
\[
\Lambda:=\bigcup_{j\in \N}\Gamma^{\sigma_j}|_{\overline U_j}\subset \dot T^\ast M.
\]
This is a closed conical set because the family $\{V_j\}_{j\in \N}$ is locally finite and $\overline U_j\subset V_j$ for all $j\in \N$. Moreover, since $V_{j}\subset M_{\sigma_j}$ for all $j\in \N$, one has $\Lambda\subset \Upsilon$. 

Choose a countable basis $\{\mathscr U_i\}_{i\in \N}$ of the topology of $M$ subordinate to the open cover $\{U_j\}_{j\in \N}$, and for each $i\in \N$, fix some $j(i)\in \N$ such that $\mathscr U_i\subset U_{j(i)}$. By \eqref{eq:intersectionempty3636}, the conical set $\Lambda_{j(i)}$ over the smooth manifold $V_{j(i)}\supset \mathscr U_i$ is $H^{-1/2-0}$-atomizable, so  there exists  $A_i\in \E'(V_{j(i)})\setminus \{0\}$, extended by zero to  $A_i\in \E'(M)$,  with $A_i\in H_\mathrm{comp}^{-1/2-0}(M)$, $\supp A_i\subset \mathscr U_i$, and $\WF(A_i)\subset \Lambda_{j(i)}\cap \dot T^\ast \mathscr U_i\subset \Lambda$. Every open subset of $M$ contains some $\mathscr U_i$, so $\Lambda$ is $H^{-1/2-0}$-atomizable, proving the first claim for $\dim M=1$.
 
If $\dim M\geq 2$, we apply Proposition \ref{prop:existence1} (2a) to $\Gamma:=\dot T^\ast M\setminus(-\Upsilon)$ and get an $H^{-1/2-0}$-atomizable conical set $\Lambda\subset \Upsilon$. This proves the first claim.

To prove the second claim, suppose that $\Upsilon_p\cap (-\Upsilon_p)\neq \emptyset$ for all $p\in M$. If $\dim M=1$, this means that $\Upsilon=\dot T^\ast M$, and then there is a positively $H^{-1/2-0}$-atomizable conical set $\Lambda\subset \Upsilon$ by Example \ref{ex:maximal}.  If $\dim M\geq 2$, the set  $\Gamma:=\dot T^\ast M\setminus(-\Upsilon)$ satisfies for $p\in M$
\[
\Gamma_p\cup (-\Gamma_p)=(\dot T^\ast_p M\setminus\Upsilon_p)\cup (\dot T^\ast_p M\setminus(-\Upsilon_p))=\dot T^\ast_p M\setminus (\Upsilon_p\cap (-\Upsilon_p))\neq \dot T_p^\ast M,
\]
so Proposition \ref{prop:existence1} (2b) applies to $\Gamma$. This provides a positively $H^{-1/2-0}$-atomizable conical set $\Lambda\subset\Upsilon$.
\end{proof}

\section{Proof of main results}\label{sec:mainproofs}

The following terminology is convenient to add extra technical strength to the main theorem stated below. The core statements, however, do not require this terminology.
\begin{definition}\label{def:controlfunction}Let $\Gamma\subset\dot T^*M$ be a closed conical set and let $s\in \R$. A function  $\mathsf{p}:\D'_{\Gamma}(M)\cap H^{s}_\mathrm{loc}(M)\to [0,\infty)$ 
is called a \emph{control function} if there exist seminorms  $q_1,\ldots,q_N\in\mathsf{SEM}(\Gamma,s)$, $N\geq 1$, and a constant $C>0$ such that
$\mathsf{p}\leq C(q_1+\cdots+q_N)$.
\end{definition} 

We now state and prove the main result of this paper:
\begin{theorem}\label{thm:main}Let $M$ be a smooth manifold, equipped with a smooth measure fixing an injection $L^1_\mathrm{loc}(M)\hookrightarrow \D'(M)$, and let $\Gamma,\Lambda\subset\dot T^*M$ be closed conical sets satisfying
\[
 \Gamma\cap(-\Lambda)=\emptyset,\qquad \pi(\Gamma)=\pi(\Lambda)=M,
\]
where $\pi:T^\ast M\to M$ is the canonical projection. Assume that, for some $s\in \R$, the conical set  $\Lambda$  is  $H^\bullet$-atomizable, where $\bullet=s$ or $\bullet=s-0$ (Definitions \ref{def:tame-atomizable}, \ref{def:uniformsminus0}). 

\begin{enumerate}
\item There exist distributions 
\[
u\in\D'_{\Gamma}(M)\cap H^{\dim M/2}_\mathrm{loc}(M),\qquad v\in\D'_{\Lambda}(M)\cap H^{\bullet}_\mathrm{loc}(M)
\]
such that
\[
\supp  u=\supp  v=M, \qquad uv=0.\vspace*{0.5em}
\]
\item If an $\eps>0$ (and in the case $\bullet=s-0$ another $\eps_0>0$), a smooth function $f\in \Cinft(M)$ with zero set of empty interior in $M$, and control functions
\begin{align*}
\mathsf p_\Gamma:&\;\D'_{\Gamma}(M)\cap H^{\dim M/2}_\mathrm{loc}(M)\to [0,\infty),\\
\mathsf p_\Lambda:&\begin{cases}\D'_{\Lambda}(M)\cap H^{s}_\mathrm{loc}(M)\hspace*{0.725em}\to [0,\infty),& \bullet =s,\\
\D'_{\Lambda}(M)\cap H^{s-\eps_0}_\mathrm{loc}(M)\to [0,\infty),& \bullet =s-0\end{cases}
\end{align*}
as in Definition \ref{def:controlfunction} are given, then we can achieve in (1) that
\bq
\mathsf p_\Gamma(u-f)<\eps,\qquad \mathsf p_\Lambda(v)<\eps.\label{eq:seminormseps}
\eq
\item If $\Gamma_p\cap (-\Gamma_p)\neq \emptyset$ for all $p\in M$, then we can achieve (1) with a positive distribution  $u$. If in  (2) the prescribed function $f$ satisfies $f\geq 0$, then we can achieve also \eqref{eq:seminormseps} with a positive distribution $u$.
\item If $\Lambda$ is positively $H^\bullet$-atomizable, then we can achieve (1) with a positive distribution $v$, and in the situation of (2) we can achieve \eqref{eq:seminormseps} with a positive $v$.
\item If $\Gamma_p\cap (-\Gamma_p)\neq \emptyset$ for all $p\in M$ and $\Lambda$ is positively $H^\bullet$-atomizable, then we can achieve (1)  with positive distributions $u,v$, and if in (2) the prescribed function $f$ satisfies $f\geq 0$, then we can also achieve \eqref{eq:seminormseps} with positive  $u,v$.
\end{enumerate} 
\end{theorem}
\begin{proof} We first prove (1) and (2). To prove both claims at the same time, we make the convention that if no function $f$, no control functions, and no $\eps>0$ (resp.\ $\eps>0$ and $\eps_0>0$ in the case $\bullet=s-0$) are specified, we put $\mathsf p_\Lambda=\mathsf p_\Gamma=0$, $\eps=1$ (resp.\ $\eps=1$ and $\eps_0=1$), as well as $f=1$, i.e., the constant function $1$. 

 Let  $\cA_\Lambda\subset \E'_\Lambda(M)\cap H^\bullet_{\mathrm{comp}}(M)$ be an atomizing family for the conical set $\Lambda$.  We will  inductively construct $v$ as the limit $v=\lim_{j\to \infty}v_j$ in $\D'_\Lambda(M)$ of a sequence of linear combinations  of  atoms 
\[
v_{j}=\sum_{i=1}^j\lambda_iA_i,\qquad A_i\in \cA_\Lambda,\;\lambda_i\in \C,
\]
and $u$ as the limit $u=\lim_{j\to \infty}u_j$ in $\D'_\Gamma(M)$ of a sequence of products of smooth functions 
\[
u_j=fm_1\cdots m_j,\qquad m_i\in C^\infty(M),
\]
where the zero set of each $m_i$ has empty interior in $M$ and contains $\supp A_i$.  
\medskip

\noindent\textbf{Step 0: Preparation.} Choose a countable family $\{\|\cdot\|_j^\Lambda\}_{j\in \N}$ of microlocal seminorms on $\D'_\Lambda(M)$ as in Lemma \ref{lem:countableseminorms1}. If $\bullet=s$, denote by
\[
\mathsf{sem}(\Lambda,s):=\big\{\norm{\chi_l\cdot}_{H^{s}(M_l)},\|\cdot\|_j^\Lambda\,|\,j,l\in \N\big\}\subset \mathsf{SEM}(\Lambda,s)
\]
the countable subset  formed by the microlocal seminorms from the chosen countable family and the Sobolev seminorms of order $s$ that we already fixed in Section \ref{sec:summary}. 

If $\bullet=s-0$, we let the Sobolev order approach $s$ from below as $l\to \infty$:
\[
\mathsf{sem}(\Lambda,s-0):=\big\{\norm{\chi_l\cdot}_{H^{s-1/l}(M_l)},\|\cdot\|_j^\Lambda\,|\,j,l\in \N\big\}. 
\]
Furthermore, choose seminorms $q^\Lambda_1,\ldots,q^\Lambda_{N}\in\mathsf{SEM}(\Lambda,s)$ (resp.\ $\mathsf{SEM}(\Lambda,s-\eps_0)$ if $\bullet =s-0$) and a constant $C>0$ with $
\mathsf{p}_\Lambda\leq C(q^\Lambda_1+\cdots+q^\Lambda_N)=:q_{\mathsf{p}_\Lambda}$. 

For $\Gamma$, we analogously choose a countable family $\{\|\cdot\|_j^\Gamma\}_{j\in \N}$ of microlocal seminorms on $\D'_\Gamma(M)$ as in Lemma \ref{lem:countableseminorms1} and  define
\[
\mathsf{sem}(\Gamma,\dim M/2):=\big\{\norm{\chi_l\cdot}_{H^{\dim M/2}(M_l)},\|\cdot\|_j^\Gamma\,|\,j,l\in \N\big\}\subset \mathsf{SEM}(\Gamma,\dim M/2),
\]
and we choose seminorms $q^\Gamma_1,\ldots,q^\Gamma_{N'}\in\mathsf{SEM}(\Gamma,\dim M/2)$ and a constant $C'>0$ with $
\mathsf{p}_\Gamma\leq C'(q^\Gamma_1+\cdots+q^\Gamma_{N'})=:q_{\mathsf{p}_\Gamma}$. 
Now we fix sequences of seminorms
\[
\{q_k^{\Lambda}\,|\,k\in \N\},\qquad \{q_k^{\Gamma}\,|\,k\in \N\}
\]
with $q_1^{\Lambda}=q_{\mathsf{p}_\Lambda}$, $q_1^{\Gamma}=q_{\mathsf{p}_\Gamma}$, and such that $\{q_k^{\Lambda}\}_{k\geq 2}$ (resp.\ $\{q_k^{\Gamma}\}_{k\geq 2}$) enumerates the countable set  $\mathsf{sem}(\Lambda,\bullet)$ (resp.\ $\mathsf{sem}(\Gamma,\dim M/2)$).

Let $\{{\mathscr U}_j\}_{j\in \N}$ be a countable basis of the topology of $M$ such that $\overline {\mathscr U}_j\subset \mathrm{Int}(M_j)$ for all $j$, where the $M_j$ are the compact sets exhausting $M$ chosen in \eqref{eq:exhaustion}. For example, we can take small Euclidean balls in charts as the sets ${\mathscr U}_j$.    

\medskip
\noindent\textbf{Step 1: Base case and induction hypothesis.} Put $u_0:=f\in\D'_\Gamma(M)$ and $v_0:=0\in \D'_\Lambda(M)$. Choose some  $\phi_1\in \CT(\mathscr U_1)$ with $\eklm{u_0,\phi_1}=1$. 

Now let some $j\in \N$ be given. Assume that, for each $i\in \{0,\ldots,j-1\}$,  smooth functions $u_{i}\in \Cinft(M)\subset \D'(M)$ and distributions $v_{i}=\sum_{k\leq i}\lambda_k A_k\in \E'(M)$ have been constructed, where $\lambda_k\in \C$, $A_k\in \cA_\Lambda$, in such a way that the zero set of the smooth function $u_{i}$ has  empty interior in $M$ and one has $u_iv_i=0$. 
For each $i\in \{1,\ldots,j\}$, suppose that some $\phi_i\in C_c^\infty({\mathscr U}_i)$ has been chosen with $
\langle u_{i-1},\phi_i\rangle=1$. 

\medskip
\noindent\textbf{Step 2: Beginning of induction step.}  By \eqref{eq:multiplicationmicrolocalnorms}, Lemma \ref{lem:multiplicationsobolev} and the fact that $\mathsf p_\Gamma$ is dominated by a finite number of seminorms in  $\mathsf{SEM}(\Gamma,\dim M/2)$ there is a $\delta_j>0$ and a finite subset $\mathsf{Q}_j\subset\mathsf{SEM}(\Gamma,\dim M/2)$ such that, for all $h\in C^\infty(M)$, we have 
\bq
q(h)<\delta_j \quad \forall\;q\in \mathsf{Q}_j\implies q_k^\Gamma(u_{j-1}h)\leq  2^{-j-1}\eps  \quad\forall\; k\leq  j.
\label{eq:qhconditionj}
\eq
Moreover, by making $\mathsf{Q}_j$ larger and $\delta_j$ smaller, we can achieve that, for all $h\in C^\infty(M)$,
\bq
q(h)<\delta_j \quad \forall\;q\in \mathsf{Q}_j\implies |\langle u_{j-1}h,\phi_i\rangle| \leq  2^{-j-2}  \qquad\forall\; i\leq  j.
\label{eq:qhconditionj2}
\eq
Indeed, it suffices to add the seminorm $\norm{\chi_{j+1}\cdot}_{H^{\dim M/2}(M_{j+1})}$ to $\mathsf{Q}_j$, since
\[
|\langle u_{j-1}h,\phi_i\rangle|=|\langle h,u_{j-1}\phi_i\rangle|\leq C_{j}\norm{\chi_{j+1}h}_{H^{\dim M/2}(M_{j+1})}\|u_{j-1}\phi_i\|_{H^{-\dim M/2}(M_{j+1})}
\]
with $C_j>0$ independent of $h$, using \eqref{eq:globalpairing} and the fact that $\chi_{j+1}=1$ on  $\supp \phi_i\subset M_j$. 

\medskip
\noindent\textbf{Step 3: Induction step for the sequence $\{u_j\}_{j\in \N_0}$.} Choose a point $p_j\in {\mathscr U}_j$.  Since $\pi(\Gamma)=M$, we have $\Gamma_{p_j}\ne\emptyset$, so we can apply the first part of Proposition \ref{prop:finiteflat} to the finite family of seminorms $\mathsf{Q}_j$. This gives us an open set $U_j\subset {\mathscr U}_j$ containing $p_j$. By atomizability, choose an  atom $A_j\in\cA_\Lambda$ with $K_j:=\supp A_j\subset U_j$. Then  Proposition \ref{prop:finiteflat} provides  a function $m_j\in\I^{\infty}(K_j)$ with $ \mathrm{Int}(m_j^{-1}(\{0\}))\cap (M\setminus K_j)=\emptyset$ and $q(m_j-1)<\delta_j$ for all $q\in \mathsf{Q}_j$. Since $K_j$ has empty interior in $M$ (by Definition \ref{def:tame-atomizable}), it follows that the full zero set of $m_j$ has empty interior in $M$.   We now define
\[
u_j:=u_{j-1}m_j,  \qquad  d_j:=u_j-u_{j-1}=u_{j-1}(m_j-1).
\]
Then, applying \eqref{eq:qhconditionj} and \eqref{eq:qhconditionj2} to $h=m_j-1$, we have
\begin{align}
q_k^\Gamma(d_j)&\leq 2^{-j-1}\eps\qquad\forall\; k\leq  j,\label{eq:estim_ckGamma}\\
|\langle d_j,\phi_i\rangle| &\leq 2^{-j-2}\qquad\forall\; i\leq  j.\label{eq:Aiestim}
\end{align} 
The zero set of $u_j$ is the union of the zero sets of $u_{j-1}$ and $m_j$, which has empty interior in $M$. So we can choose a function $\phi_{j+1}\in C_c^\infty({\mathscr U}_{j+1})$ with $
\langle u_{j},\phi_{j+1}\rangle=1$. 

This finishes the induction step for the sequence $\{u_j\}_{j\in \N_0}$.

\medskip
\noindent\textbf{Step 4: Induction step for the sequence $\{v_j\}_{j\in \N_0}$.}  We now construct $v_j$. Choose $\psi_j\in C_c^\infty({\mathscr U}_j)$ with $\langle A_j,\psi_j\rangle\ne0$, this exists because $A_j$ is supported in ${\mathscr U}_j$ and non-zero.   Then choose $\lambda_j\in \C\setminus\{0\}$ such that
\begin{align}
q_k^\Lambda(\lambda_jA_j)&\leq  2^{-j-1}\eps   \qquad\forall\; k\leq  j,\label{eq:estim_ckLAmbda}\\
|\lambda_j\langle A_j,\psi_i\rangle|        &\leq 2^{-j-2}|\langle v_i,\psi_i\rangle|\qquad \forall\; i< j,\label{eq:estim_lambdaSj}\\
\lambda_j&\neq- \frac{\langle v_{j-1},\psi_j\rangle}{\langle A_j,\psi_j\rangle},\label{eq:singlevalue}
\end{align}
and define 
\[
v_j:=v_{j-1}+\lambda_jA_j.
\]
By \eqref{eq:singlevalue} we have $\langle v_j,\psi_j\rangle\ne0$, and we find
\[
u_jv_j=u_{j-1}m_j(v_{j-1}+\lambda_jA_j)=m_j\hspace*{-0.75em}\underbrace{u_{j-1}v_{j-1}}_{=0\text{ by ind.\ hyp.}}\hspace*{-0.5em}+\;\lambda_j u_{j-1}\hspace*{-1.75em}\underbrace{m_jA_j}_{=0\text{ by Lemma \ref{lem:flatkills}}}\hspace*{-1.5em}=0.
\]
This finishes the induction step for the sequence $\{v_j\}_{j\in \N_0}$.

\medskip
\noindent \textbf{Step 5: Convergence and control function estimates.} For each fixed $k\in \N$, the estimates \eqref{eq:estim_ckGamma} and \eqref{eq:estim_ckLAmbda}  give
\bq
\sum_{j=k}^\infty q_k^\Gamma(u_j-u_{j-1})<\eps,     \qquad   \sum_{j=k}^\infty q_k^\Lambda(v_j-v_{j-1})<\eps.\label{eq:sumsfinite}
\eq
This implies that the sequences $u_j$ and $v_j$ converge in $\D'_\Gamma(M)$ to limits
\[
u\in\D'_\Gamma(M)\cap H^{\dim M/2}_\mathrm{loc}(M), \qquad   v\in\D'_\Lambda(M)\cap H^{\bullet}_\mathrm{loc}(M).
\]
Indeed: Considering an index $k>1$ in \eqref{eq:sumsfinite} such that $q_k^\Gamma$ (resp.\ $q_k^\Lambda$) is a Sobolev norm of the shape $\norm{\chi_l\cdot}_{H^{\dim M/2}(M_l)}$ (resp.\ $\norm{\chi_l\cdot}_{H^{s}(M_l)}$ if $\bullet =s$ and $\norm{\chi_l\cdot}_{H^{s-1/l}(M_l)}$ if $\bullet=s-0$) implies that the sequence $u_j$ (resp.\ $v_j$) is Cauchy with respect to this Sobolev norm, hence  $\chi_l u_j$ (resp.\ $\chi_l v_j$) is convergent in the Hilbert space $H^{\dim M/2}(M_l)$ (resp.\ $H^{s}(M_l)$ if $\bullet =s$ and $H^{s-1/l}(M_l)$ if $\bullet =s-0$ ) as $j\to \infty$. With Corollary \ref{cor:convergencekl} it follows that $u_j$ (resp.\ $v_j$) converge in $\D'(M)$ to some $u\in \D'(M)$ (resp.\ $v\in \D'(M)$). The  uniqueness of the limit in $\D'(M)$ implies that for all $l\in \N$ one has  $\chi_lu_j \to \chi_l u$ in $H^{\dim M/2}(M_l)$ (resp.\ $\chi_l v_j\to \chi_l v$ in $H^{s}(M_l)$ if $\bullet =s$ and $\chi_l v_j\to \chi_l v$ in $H^{s-\eps'}(M_l)\;\forall\;\eps'\geq 1/l$ if $\bullet=s-0$). 

By the analogous argument for the microlocal seminorms enumerated by the remaining $k>1$ we conclude with Lemma \ref{lem:countableseminorms1} that the limits lie in $\D'_\Gamma(M)$ (resp.\ $\D'_\Lambda(M)$) and the convergences happen in $\D'_\Gamma(M)$ (resp.\ $\D'_\Lambda(M)$). Since for each $l$ we have $\chi_{l+1}=1$ on $M_l$ and the sets $M_l$ exhaust $M$, and $s-1/l$ converges to $s$ as $l\to \infty$, it follows that $u\in H^{\dim M/2}_\mathrm{loc}(M)$ and $v\in H^{\bullet}_\mathrm{loc}(M)$. Indeed, let us consider, for example, the case $\bullet=s-0$ explicitly: Given a cutoff function $\tau\in \CT(M)$ and $\eps'>0$, we have for large enough $l$ that $\supp \tau \subset M_{l-1}$ and $1/l<\eps'$. Then, since $\chi_l=1$ on $M_{l-1}$ and $
\tau v=\chi_l \tau v\in H^{s-1/l}(M_l)\subset H^{s-\eps'}(M_l)$, 
we see that $\tau v\in H^{s-\eps'}_\mathrm{comp}$. Since $\tau$ and $\eps'>0$ were arbitrary, it follows that $v\in H^{s-0}_\mathrm{loc}(M)$.   

Finally, recall that $q_1^\Gamma=q_{\mathsf p_\Gamma}$ and $q_1^\Lambda=q_{\mathsf p_\Lambda}$ are the seminorms dominating the given control functions. Thus the $k=1$ case of \eqref{eq:sumsfinite} implies, by telescoping, that $\mathsf p_\Gamma(u-u_0)<\eps$ and $\mathsf p_\Lambda(v-v_0)<\eps$. Indeed: For all $N\in \N$ we have
\begin{align*}
\mathsf p_\Lambda(v-v_0)\leq q_{\mathsf p_\Lambda}(v-v_0)&\leq q_{\mathsf p_\Lambda}(v-v_N)+\sum_{j=1}^N q_{\mathsf p_\Lambda}(v_j-v_{j-1}) \\
&\leq q_{\mathsf p_\Lambda}(v-v_N)+\eps\sum_{j=1}^N 2^{-j-1}\;\leq q_{\mathsf p_\Lambda}(v-v_N)+\frac{\eps}{2}.
\end{align*}
Now, the finitely many Sobolev components occurring in $q_{\mathsf p_\Lambda}$ are of the form $\norm{\chi_l\cdot}_{H^{s}(M_l)}$ if $\bullet =s$ and $\norm{\chi_l\cdot}_{H^{s-\eps_0}(M_l)}$ if $\bullet=s-0$. Let $l_0\in \N$ be larger than all the finitely many occurring $l$; if $\bullet=s-0$, choose $l_0$ also large enough that $1/l_0\leq \eps_0$. Then the estimates $\norm{\chi_l\cdot}_{H^{s}(M_l)}\leq C_{s,l,l_0}\norm{\chi_{l_0}\cdot}_{H^{s}(M_{l_0})}$ and $\norm{\chi_l\cdot}_{H^{s-\eps_0}(M_l)}\leq C_{s,l,l_0}'\norm{\chi_{l_0}\cdot}_{H^{s-1/l_0}(M_{l_0})}$ imply that all Sobolev components in $q_{\mathsf p_\Lambda}$, evaluated on  $v-v_N$,  converge to   zero as $N\to \infty$. Furthermore, since $v_N\to v$ in $\D'_\Lambda(M)$, all microlocal components in $q_{\mathsf p_\Lambda}$ converge to zero when evaluated on  $v-v_N$. Thus $q_{\mathsf p_\Lambda}(v-v_N)\to 0$ as $N\to \infty$, and consequently $\mathsf p_\Lambda(v-v_0)<\eps$. We argue similarly for  $\mathsf p_\Gamma$.  Since $u_0=f$ and $v_0=0$, this proves \eqref{eq:seminormseps}.

\medskip
\noindent\textbf{Step 6: Full support.} We verify that $u$ and $v$ have full support.  For each  fixed $i\in \N$, we have $
\langle u,\phi_i\rangle =1+\sum_{j\geq  i}\langle d_j,\phi_i\rangle$. 
By \eqref{eq:Aiestim}, the sum is bounded by $\sum_{j\geq  i}2^{-j-2}< 1$, so $\langle u,\phi_i\rangle\ne0$. Similarly, with $C_i:=\langle v_i,\psi_i\rangle\ne0$, we have $
\langle v,\psi_i\rangle=C_i+\sum_{j>i}\lambda_j\langle A_j,\psi_i\rangle$, 
and by \eqref{eq:estim_lambdaSj} the sum is bounded by $|C_i|\sum_{j>i}2^{-j-2}<|C_i|$. Hence $\langle v,\psi_i\rangle\ne0$.  

Since $\{{\mathscr U}_i\}_{i\in \N}$ is a basis of the topology of $M$ and the test functions $\phi_i,\psi_i$ are supported in ${\mathscr U}_i$,  we infer that $u$ and $v$ are non-zero on every non-empty open subset of $M$. Therefore $\supp u=\supp v=M$.

\medskip
\noindent\textbf{Step 7: Vanishing of the product.} We finally consider the product $uv$. First, fix some $j\in \N_0$. Since $u_jv_j=0$, we find for all $N\geq  j$ that $u_Nv_j=m_N\cdots m_{j+1}u_jv_j=0$. 
By Proposition~\ref{prop:seqmult}, passing to the limit $N\to \infty$ gives $uv_j=0$.  

Since the above holds for every $j\in \N$, another application of Proposition~\ref{prop:seqmult} (with $\Lambda$ and $\Gamma$ swapped compared to the first application) gives $uv=\lim_{j\to\infty}uv_j=0$. 
This finishes the proof of the claims (1) and (2). 

\medskip
\noindent\textbf{Step 8: Positive cases.} Let us now prove (3): Assume that $\Gamma_p\cap (-\Gamma_p)\neq \emptyset$ for all $p\in M$. Then we make small adjustments to some of the above arguments:  In Step 1, we observe that $u_0=f$ is positive if, and only if, $f\geq 0$.  In Step 3, we apply the strengthened part of Proposition \ref{prop:finiteflat}, providing $m_j\in \I^{\infty,>0}(K_j)$. In particular, $m_j\geq 0$. Thus, if $u_{j-1}$ was positive, then also $u_j=m_ju_{j-1}$ is positive. By induction, all $u_j$ are positive (provided that $u_0\equiv f\geq 0$ or no $f$ is specified, in which case $u_0\equiv 1$ is positive), and consequently their limit $u$ is positive: If $\varphi\in \CT(M,\R)$ satisfies $\varphi\geq 0$, then $u(\varphi)=\lim_{j\to \infty}u_j(\varphi)\geq 0$ because $u_j(\varphi)\geq 0$ for all $j$. 

Similarly, we prove (4):\label{step8} Assume that $\Lambda$ is positively $H^\bullet$-atomizable.  Then again we make small adjustments:  In Step 1, we observe that $v_0=0$ is positive.  In Step 3, the positive atomizability of $\Lambda$ allows us to choose  $A_j$ positive. In Step $4$, instead of $\lambda_j\in \C\setminus\{0\}$, we take $\lambda_j\in (0,\infty)$ to ensure that, if $v_{j-1}$ was positive, also $v_j=v_{j-1}+\lambda_j A_j$ is positive. Then $v=\lim_{j\to \infty}v_j$ is positive by the analogous elementary argument as for (3).  Although not necessary, we may in addition simplify the full support argument in Step 6 by choosing the test functions $\psi_i$  such that $\psi_i\geq 0$ and $\langle A_i,\psi_i\rangle>0$ (rather than just $\langle A_i,\psi_i\rangle\ne0$), which in Step 6 gives $\eklm{v,\psi_i}>0$ without invoking \eqref{eq:estim_lambdaSj} because only positive terms are summed.

Finally, to prove (5), we make the above adjustments for (3) and (4) simultaneously. This finishes the proof of Theorem \ref{thm:main}. 
\end{proof}

\subsection{Proofs of other main theorems}\label{sec:remainingthms}

Let us now derive Theorems \ref{thm:intro1}, \ref{thm:intro2a}, \ref{thm:intro2} and \ref{thm:intro3} stated in the introduction from Theorem \ref{thm:main}, Proposition \ref{prop:existence1}, and Corollary \ref{cor:existence1}. For the second part of Theorem \ref{thm:intro1} on the absence of continuity we require two additional technical lemmas.

\begin{lemma}\label{lem:nearidentity}
Let $s\geq0$, and let
$\Gamma,\Lambda\subset\R^n\setminus\{0\}$ be closed cones with $
  \Gamma\cap(-\Lambda)=\emptyset$.  
There exists a constant $\eps_0=\eps_0(s,\Gamma,\Lambda)>0$ such that the following holds:  Let $U\subset\R^n$ be open, $a\in L^\infty(U)$, $w\in H_\mathrm{comp}^{-s}(U)$. Suppose that
\[
\WF(a)\subset U\times\Gamma,  \qquad \WF(w)\subset U\times\Lambda,\qquad \|a-1\|_{L^\infty(U)}<\eps_0.
\]
Then $aw=0$ implies  $w=0$. 
\end{lemma}

\begin{proof}\footnote{As mentioned in the AI use disclosure on p.~\pageref{sec:Aiuse}, this proof is the author's revised version of arguments suggested by the  LLM  GPT-5.6 Sol.}
Choose closed cones $\Gamma_1,\Lambda_1\subset\R^n\setminus\{0\}$ such that $
\Gamma\subset\operatorname{Int}\Gamma_1$,  $\Lambda\subset\operatorname{Int}\Lambda_1$, and $\Gamma_1\cap(-\Lambda_1)=\emptyset$. 
Since $\Gamma_1\cap S^{n-1}$ and $(-\Lambda_1)\cap S^{n-1}$ are compact and disjoint, there is a constant $C_0\geq1$ such that
\bq\label{eq:coneseparation}
|\alpha|+|\beta|\leq C_0|\alpha+\beta|  \qquad \forall\;\alpha\in\Gamma_1, \beta\in\Lambda_1.
\eq
Fix $C>1$ such that
\bq\label{eq:Cchoice}
  2C-1>C_0,
\eq
and put
\bq\label{eq:epsilonchoice}
  \eps_0:= C^{-s}/2.
\eq
Choose $\chi\in C_c^\infty(U)$ such that $0\leq\chi\leq1$ and $\chi=1$ near $\supp w$, and set $b:=\chi(1-a)$.
After extending $b$ by zero to $\R^n$, we have 
$b\in L^\infty_{\mathrm{comp}}(\R^n)$,
\bq\label{eq:bsmall}
 \|b\|_{L^\infty(\R^n)}  \leq\|1-a\|_{L^\infty(U)}  <\eps_0,
\eq
\bq\label{eq:bwf}
\WF(b)\subset\R^n\times\Gamma.
\eq
Now assume that $aw=0$. Then 
\bq\label{eq:wbw}
bw=\chi w-\chi(aw)=\chi w =w.
\eq
We will show that \eqref{eq:bsmall} and \eqref{eq:wbw} imply $w=0$ by separately considering the Fourier transforms of $bw$ and $w$ and using only later that $bw=w$. 

For the first estimate, we extend $w$ by zero to $w\in H^{-s}_\mathrm{comp}(\R^n)$ and recall that the Fourier transform of $bw\in \E'(\R^n)$ is given by
\bq\label{eq:fullconvolution}
\widehat{bw}(\xi)  =(2\pi)^{-n}\int_{\R^n}       \widehat b(\xi-\eta)\widehat w(\eta)\,d\eta,\qquad \xi\in \R^n,
\eq
where the integral converges by \eqref{eq:bwf}, $\WF(w)\subset \R^n\times\Lambda$, and $\Gamma\cap(-\Lambda)=\emptyset$ (see \cite[Def.~2, Thm.~13]{smoothWF}). To study the convolution \eqref{eq:fullconvolution} it is  convenient  to define for $R\geq1$ an auxiliary function $w_R\in L^2(\R^n)\cap \Cinft(\R^n)$ by
\bq\label{eq:wR}
\widehat{w_R}:=1_{B_{CR}}\widehat w,
\eq
where $C>0$ is the constant from \eqref{eq:Cchoice}, we use the simplified notation $B_r:=B_r(0)$ for the open ball of radius $r$ around $0$ in $\R^n$, and $1_{B_{CR}}$ is the characteristic function of $B_{CR}$.    Since $\widehat{w_R}\in L^1(\R^n)\cap L^2(\R^n)$ and $b\in L^\infty_\mathrm{comp}(\R^n)\subset L^1(\R^n)\cap L^\infty(\R^n)$, Fourier inversion gives
\bq\label{eq:truncconvolution}
\widehat{bw_R}(\xi)  =(2\pi)^{-n}\int_{|\eta|\leq CR}  \widehat b(\xi-\eta)\widehat w(\eta)\,d\eta.
\eq
Let us estimate the high-frequency part of the convolution \eqref{eq:fullconvolution}: We claim that
\bq\label{eq:tail}
 \Big\|\xi\mapsto   \int_{|\eta|>CR}\widehat b(\xi-\eta)\widehat w(\eta)\,d\eta
 \Big\|_{L^2(B_R)}=\mathcal O(R^{-\infty})
\eq
as $R\to \infty$. To see this, let $\xi \in \R^n$ with $|\xi|\leq R$, split the integration domain into
\begin{align*}
E_1 :=\{\eta\in\R^n\,|\,|\eta|>CR,\ \eta\not\in\Lambda_1\},\quad 
E_2 :=\{\eta\in\R^n\,|\,|\eta|>CR,\ \eta\in\Lambda_1\}.
\end{align*}
On $E_1$, $\widehat w$ decreases rapidly, uniformly outside $\Lambda_1$, while
$|\widehat b|\leq\|b\|_{L^1}$.  Hence the integral over $E_1$ is
$\mathcal O(R^{-K})$ for every $K\in \N$, uniformly for $|\xi|\leq R$.

If $\eta\in E_2$, then $\xi-\eta\notin\Gamma_1$: otherwise \eqref{eq:coneseparation} with $\alpha=\xi-\eta$ and $\beta=\eta$ would imply
$|\xi-\eta|+|\eta|\leq C_0|\xi|\leq C_0R$, 
whereas $ |\xi-\eta|+|\eta|  \geq 2|\eta|-|\xi|  >(2C-1)R$, 
contradicting \eqref{eq:Cchoice}.  Moreover, for $\eta\in E_2$, we have $
|\xi-\eta|\geq|\eta|-|\xi|  \geq(1-C^{-1})|\eta|$. 
Thus $\widehat b(\xi-\eta)$ is rapidly decaying in $\eta$ on $E_2$, uniformly for $|\xi|\leq R$. That is, for each $L>0$ there is a $C_L>0$ such that for all $\xi \in \R^n$ with $|\xi|\leq R$ and all $\eta\in E_2$ one has $
|\widehat b(\xi-\eta)| \leq C_L\langle\eta\rangle^{-L}$. 
Hence, by applying Cauchy-Schwarz in $L^2(\R^n)$ and using that $w\in H^{-s}(\R^n)$, we get
\begin{align*}
 \int_{E_2}  |\widehat b(\xi-\eta)\widehat w(\eta)|\,d\eta
&\leq C_L\int_{|\eta|>CR}     \langle\eta\rangle^{-L+s}
     \langle\eta\rangle^{-s}|\widehat w(\eta)|\,d\eta\\
&\leq C_L \left(\int_{|\eta|>CR}\langle\eta\rangle^{-2L+2s}\,d\eta\right)^{1/2}
 \|\langle\eta\rangle^{-s}\widehat w\|_{L^2(\R^n)}=\mathcal O(R^{-L+s+n/2}).
\end{align*}
To pass from these uniform pointwise estimates in $B_R$ to an $L^2(B_R)$-estimate for the integrals over $E_1$ and $E_2$, respectively, it suffices to integrate the squares of the estimated expressions over $B_R$ and to take into account that $ \mathrm{vol} (B_R)^{1/2}=\mathcal O(R^{n/2})$. Since $K,L>0$ can be chosen arbitrarily large, this proves \eqref{eq:tail}.

Next, we  turn to the Fourier transform $\widehat w$. Define for $R\geq 1$ the auxiliary function $
F_w(R):=\|\widehat w\|_{L^2(B_R)}$. 
Note that this function is non-decreasing, and since $w\in H^{-s}(\R^n)$, one has  $F_w(R)=\mathcal O(R^s)$ as $R\to \infty$, i.e., there is a constant  $C_w>0$ such that 
\bq
F_w(R)\leq C_w R^s\quad \forall\; R\geq 1.\label{eq:Fgrowth}
\eq
Now we use that $w=bw$ by \eqref{eq:wbw}, and splitting \eqref{eq:fullconvolution} at
$|\eta|=CR$, \eqref{eq:tail} and
\eqref{eq:truncconvolution} imply 
\[
\widehat w|_{B_R}=\widehat{bw_R}|_{B_R}+r_R\quad \text{in } L^2(B_R),
\]
where $r_R\in L^2(B_R)$ satisfies that for every $N\in \N$ there is a $C_N>0$ with $ \|r_R\|_{L^2(B_R)}\leq C_NR^{-N}$ for all $R\geq 1$. 
We can therefore estimate (using Plancherel at the first equality and recalling \eqref{eq:wR} at the last equality):
\begin{align}
 F_w(R)  \leq \|\widehat{bw_R}\|_{L^2(\R^n)}+C_NR^{-N} &=(2\pi)^{n/2}\|bw_R\|_{L^2(\R^n)}+C_NR^{-N}\nonumber\\
 &\leq \|b\|_{L^\infty(\R^n)}\|\widehat{w_R}\|_{L^2(\R^n)}+C_NR^{-N}\nonumber\\
 &=\|b\|_{L^\infty(\R^n)}F_w(CR)+C_NR^{-N}.  \label{eq:Frecursion}
\end{align}
Set $\eps:=\|b\|_{L^\infty(\R^n)}<\eps_0$.  Iterating \eqref{eq:Frecursion} gives for all $j\in \N$
\bq
 F_w(R)\leq \eps^j F_w(C^jR) +C_NR^{-N}\sum_{k=0}^{j-1}(\eps C^{-N})^k.\label{eq:Fiteration}
\eq
By \eqref{eq:Fgrowth}, and recalling from \eqref{eq:epsilonchoice} that $\eps_0=\frac12 C^{-s}$, we get
\[
  \eps^jF_w(C^jR)  \leq C_w R^s(\eps C^s)^j \leq C_w R^s 2^{-j} \longrightarrow 0\quad \text{ as }j\to \infty.
\]
On the other hand, the geometric sum in \eqref{eq:Fiteration} is bounded by $(1-\eps C^{-N})^{-1}$. Therefore, passing to the limit $j\to \infty$ and putting $C_{N}':= C_N (1-\eps C^{-N})^{-1}>0$, we arrive at
\bq
  F_w(R)\leq C'_{N} R^{-N}  \quad \forall\; R\geq 1.\label{eq:Fdecay}
\eq
Since $F_w$ is non-decreasing and non-negative, choosing $N=1$ in \eqref{eq:Fdecay} forces $F_w(R)=0$ for all $R\geq 1$.  Hence $\widehat w=0$ and consequently $w=0$. This finishes the proof.
\end{proof}

\begin{lemma}\label{lem:supportdetection}Let $u\in C(M)$ and $v\in\D'(M)$ be such that $  \WF(u)\cap(-\WF(v))=\emptyset$. Then one has
\bq\label{eq:support-identity}
  \supp(uv)\cap\{p\in M\,|\,u(p)\neq 0\}
  =\supp v\cap\{p\in M\,|\,u(p)\neq 0\}.
\eq
In particular,
\bq\label{eq:support-zero-product}
 uv=0\implies \supp v\subset u^{-1}(\{0\}).
\eq
\end{lemma}
\begin{proof}\footnote{As mentioned in the AI use disclosure on p.~\pageref{sec:Aiuse}, this proof is the author's revised version of arguments suggested by the  LLM  GPT-5.6 Sol.} The inclusion $\supp(uv)\subset\supp v$ is elementary, so we only need to prove that the left-hand side of \eqref{eq:support-identity} contains the right-hand side. To this end, let $
 p\in M\setminus\supp(uv)$ be such that $u(p)\neq 0$. 
We need to prove that $p\notin\supp v$. Choose a relatively compact chart $U\subset M$ around $p$ on which
$uv=0$, and identify $U$ with an open subset of $\R^n$.  After shrinking $U$, there are closed cones $\Gamma,\Lambda\subset\R^n\setminus\{0\}$ such that
\bqn
 \WF(u|_U)\subset U\times\Gamma, \qquad
  \WF(v|_U)\subset U\times\Lambda,  \qquad \Gamma\cap(-\Lambda)=\emptyset.
\eqn
Choose a smaller open set $U_1\subset M$ around $p$ such that $\overline U_1\subset U$, and choose some $\rho\in C_c^\infty(U)$ with $\rho=1$ on a neighborhood of $\overline U_1$.  Since $\rho v$ is compactly supported, there exists $s\geq 0$ such that, after extension by zero to $\R^n$,
\bq\label{eq:rho-v-sobolev}
  \rho v\in H^{-s}(\R^n).
\eq
Let $\eps_0=\eps_0(s,\Gamma,\Lambda)$ be as in Lemma \ref{lem:nearidentity}.  By continuity of $u$ and $u(p)\neq0$, there
is an open set $U_2\subset U_1$ around $p$ such that
\bq\label{eq:near-identity-local}
  \left\|\frac{u}{u(p)}-1\right\|_{L^\infty(U_2)}<\eps_0.
\eq
Choose $\chi\in C_c^\infty(U_2)$ with $\chi=1$ near $p$, and define $ a:=\frac{u}{u(p)}\big|_{U_2}$, $w:=\chi v|_{U_2}$. 
Then $a\in C(U_2)$, $w\in H^{-s}_\mathrm{comp}(U_2)$, and 
$\WF(a)\subset U_2\times\Gamma$,  $\WF(w)\subset U_2\times\Lambda$. 
Moreover, using that $\chi$ is smooth, we find
\[
 aw =\frac{1}{u(p)}\,u(\chi v) =\frac{\chi}{u(p)}\,uv =0.
\]
Together with \eqref{eq:near-identity-local}, all hypotheses of
Lemma~\ref{lem:nearidentity} are satisfied.  Hence $w=0$.  Since $\chi=1$
near $p$, the distribution $v$ vanishes near $p$, and therefore
$p\notin\supp v$.
\end{proof}

\begin{proof}[Proof of Theorem \ref{thm:intro1}]
If $\dim M=1$, choose an open conical set $\Upsilon\subset \dot T^\ast M$ with $\emptyset\neq\Upsilon_p\neq \dot T_p^\ast M$ for all $p\in M$; this exists by Proposition \ref{prop:existence1}(1) (applied with $\Gamma=\emptyset$): It gives an atomizable closed conical set $\mathscr C\subset \dot T^\ast M$ with $\mathscr C_p\neq \dot T^\ast_p M$ for all $p\in M$, and  atomizability implies $\mathscr C_p\neq \emptyset$ for all $p\in M$ by Remark \ref{rem:atomiz-nonemptyfiber}. Now put $\Upsilon:= \dot T^\ast M\setminus \mathscr C$.

If $\dim M\geq 2$, choose a symmetric open conical set $\Upsilon=-\Upsilon\subset \dot T^\ast M$ with $\emptyset\neq \Upsilon_p\neq \dot T_p^\ast M$ for all $p\in M$. This exists by Proposition \ref{prop:existence1} (2b) (applied with $\Gamma=\emptyset$), which gives a symmetric atomizable closed conical set $\mathscr C=-\mathscr C\subset \dot T^\ast M$ with $\mathscr C_p\neq \dot T^\ast_pM$ for all $p\in M$, and again $\mathscr C_p\neq \emptyset$ for all $p\in M$  by Remark \ref{rem:atomiz-nonemptyfiber}. Thus $\Upsilon:= \dot T^\ast M\setminus \mathscr C$ works.

By Corollary \ref{cor:existence1}, if $\dim M=1$, there is an atomizable closed conical set $\Lambda\subset\Upsilon$, and if $\dim M>1$,  there is a positively atomizable closed conical set $\Lambda=-\Lambda\subset\Upsilon$. 

If $\dim M=1$, the atomizability of $\Lambda$ implies that $\Lambda_p\neq \emptyset$ for all $p\in M$ by Remark \ref{rem:atomiz-nonemptyfiber}, and since  $\Lambda_p\subset\Upsilon_p\neq \dot T_p^\ast M$  for all $p\in M$, it follows in this one-dimensional situation that $\Lambda\cap (-\Lambda)=\emptyset$. We therefore simply define $\Gamma:=\Lambda$.

If $\dim M>1$, consider the symmetric open conical set $\Upsilon':=\dot T^*M\setminus\Lambda$. Then for all $p\in M$ we have $\Upsilon'_p=\Upsilon'_p\cap(-\Upsilon'_p)\supset\dot T^*_pM\setminus\Upsilon_p\neq \emptyset$. By Corollary \ref{cor:existence1}, there exists a positively atomizable symmetric closed conical set $\Gamma=-\Gamma\subset \Upsilon'$. In particular, $\Gamma_p=\Gamma_p\cap (-\Gamma_p)\neq \emptyset$ for all $p\in M$ by Remark \ref{rem:atomiz-nonemptyfiber}.

In either case, Theorem \ref{thm:main}(1) applies to $(\Gamma,\Lambda)$ and provides $u\in\D'_{\Gamma}(M)$, $v\in\D'_{\Lambda}(M)$ with $\supp  u=\supp  v=M$, $uv=0$. In particular, $\WF(u)\cap (-\WF(v))\subset \Gamma\cap (-\Lambda)=\emptyset$. Here Theorem \ref{thm:main}(1) uses only the atomizability of $\Lambda$; the atomizability of $\Gamma$ is only weakly and indirectly used through its implication that $\Gamma_p\neq \emptyset$ for all $p\in M$, as noted above.

If $\dim M>1$, we have chosen $\Gamma$ and $\Lambda$ such that Theorem \ref{thm:main}(5) applies and yields that $u$ and $v$ are positive. Again, Theorem \ref{thm:main}(5) uses only the positive atomizability of $\Lambda$; the atomizability of $\Gamma=-\Gamma$ is only used to know that $\Gamma_p\cap(-\Gamma_p)\neq \emptyset$ for all $p\in M$. The first part of Theorem \ref{thm:intro1} is proved.

The second part follows directly from Lemma \ref{lem:supportdetection}: If $u|_U\in C(U)$ for some non-empty open $U\subset M$ and $v\in\D'(M)$ is such that $\supp v= M$, $\WF(u)\cap(-\WF(v))=\emptyset$ and $uv=0$, then applying \eqref{eq:support-zero-product} to  $(U,u|_U,v|_U)$ yields $u|_U=0$, hence $\supp u\neq M$. We argue analogously with $u,v$ swapped. This finishes the proof of Theorem  \ref{thm:intro1}.
\end{proof}

\begin{proof}[Proof of Theorem \ref{thm:intro2}]
Corollary \ref{cor:existence1} provides an $H^{-1/2-0}$-atomizable closed conical set $\Lambda\subset \Upsilon$. Since $\Upsilon\cap (-\Gamma)=\emptyset$, it follows that $\Lambda\cap (-\Gamma)=\emptyset$. Atomizability implies that $\Lambda_p\neq \emptyset$ for all $p\in M$ by Remark \ref{rem:atomiz-nonemptyfiber}, and by assumption we have $\Gamma_p\neq \emptyset$ for all $p\in M$. Thus Theorem \ref{thm:main}(1) applies and proves the claim (1). 

If $\Gamma_p\cap (-\Gamma_p)\neq \emptyset$ for all $p\in M$, Theorem \ref{thm:main}(3) applies and gives the claim (2). 

If $\Upsilon_p\cap (-\Upsilon_p)\neq \emptyset$ for all $p\in M$, then Corollary \ref{cor:existence1} provides a symmetric, positively $H^{-1/2-0}$-atomizable closed conical set $\Lambda_0=-\Lambda_0\subset \Upsilon$, and Theorem \ref{thm:main}(4) applied to $(\Gamma,\Lambda_0)$ proves the claim (3).

If $\Gamma_p\cap (-\Gamma_p)\neq \emptyset$  and $\Upsilon_p\cap (-\Upsilon_p)\neq \emptyset$ for all $p\in M$, then Theorem \ref{thm:main}(5) applied to $(\Gamma,\Lambda_0)$ proves the claim (4).
\end{proof}

\begin{proof}[Proof of Theorem \ref{thm:intro2a}] Apply Theorem \ref{thm:intro2}(1),(2) with
$\Upsilon:=\dotT M\setminus(-\Gamma)$ and forget $v$.  
\end{proof}

\begin{proof}[Proof of Theorem \ref{thm:intro3}]
Let $f\in \Cinft(M)$ be a function with zero set of empty interior. Every neighborhood of $(f,0)$ in $V:=(\D'_\Gamma(M)\cap H^{\dim M/2}_\mathrm{loc}(M)) \times (\D'_\Lambda(M)\cap H^{-1/2-0}_\mathrm{loc}(M))$ with respect to the specified topology contains a set of the form
\[
\mathscr{N}_{E,\eps}(f,0):=\{(u,v)\in V\,|\, p((u,v)-(f,0))<\eps\;\forall\, p\in E\},
\]
where $\eps>0$ and $E$ is a finite set of seminorms of the form 
\[
E=\{p_{\Gamma,1},\ldots,p_{\Gamma,N},p_{\Lambda,1},\ldots,p_{\Lambda,N'}\},\quad N,N'\geq 0,
\]
where $p_{\Gamma,i}(u,v)=q_{\Gamma,i}(u)$ and $p_{\Lambda,j}(u,v)=q_{\Lambda,j}(v)$ for some $q_{\Gamma,i}\in \mathsf{SEM}(\Gamma,\dim M/2)$, $q_{\Lambda,j}\in\mathsf{SEM}(\Lambda,-1/2-\eps_0)$, $1\leq i\leq N$, $1\leq j\leq N'$. Fix such $E,\eps$. Then
\[
\mathsf{p}_\Gamma:=q_{\Gamma,1}+\cdots+ q_{\Gamma,N},\qquad \mathsf{p}_\Lambda:=q_{\Lambda,1}+\cdots+ q_{\Lambda,N'}
\]
are control functions in the sense of Definition \ref{def:controlfunction}. Thus, for the claim (1), it suffices to apply Theorem \ref{thm:main}(2) with these control functions and the given $\eps$.  Similarly, to get the claims (2)--(4), we apply the statements (3)--(5) of Theorem \ref{thm:main}.
\end{proof}

\section{Examples}\label{sec:examples}

Here we collect some geometric and dynamical settings to which our results apply.

\subsection{Foliations}

The following definition formalizes the concept of a ``continuous foliation with smooth leaves and continuous tangent (or conormal) bundle''. Our definition is a bit more general than others found in the literature because in (4) below we only consider first order derivatives: All we need is that the conormal bundle of the foliation is continuous; for that, we do not need continuity of higher order derivatives.

\begin{definition}\label{def:foliation} A \emph{$C^{\infty,0}$-foliation} $\mathcal F$ of $M$ is a collection of disjoint immersed smooth submanifolds of $M$ of some common dimension $k\geq 0$,  called \emph{leaves}, such that $
M=\bigsqcup_{\mathscr F\in \F}\mathscr F$ 
and for each point $p\in M$ there exists an open set  $U$ containing $p$ and a homeomorphism 
\[
\Phi:U\to B^k_1(0)\times B^{n-k}_1(0),\qquad n:=\dim M,
\]
called a \emph{foliating chart} centered at $p$, with the following properties: \begin{enumerate}
\item $\Phi(p)=(0,0)$. 
\item For each $y\in B^{n-k}_1(0)$, the so-called \emph{plaque} 
\[
\Phi^{-1}(B^k_1(0)\times \{y\})=:\mathscr P_y\subset M
\]
coincides with the connected component of the intersection $U\cap \mathscr F$ with the unique leaf $\mathscr F\in \F$ containing $\Phi^{-1}(0,y)$. In particular, $\mathscr P_y$ is an embedded smooth submanifold of $M$.
\item For each $y\in B^{n-k}_1(0)$, the restriction $\Phi^{-1}(\cdot,y):B^k_1(0)\times \{y\}\to \mathscr P_y$ is a diffeomorphism.
\item The derivatives of the above diffeomorphisms are continuous in the sense that for each continuous $1$-form $\xi\in C^0(U;T^\ast M)$ and each $v\in \R^k$ the map 
\[
B^k_1(0)\times B^{n-k}_1(0)\to \R,\qquad (x,y)\mapsto \xi(D(\Phi^{-1}(\cdot,y))|_x(v))
\] 
is continuous, where we identify  $\R^k\equiv T_{(x,y)}(B^k_1(0)\times \{y\})$.
\end{enumerate}
The \emph{tangent space} of $\F$ at a point $p\in M$ is $
T_p\F:=T_p\mathscr F$, 
where $\mathscr F\in \F$ is the leaf through $p$. The \emph{tangent} and \emph{conormal bundles} of $\F$ are
\[
T\F:=\bigsqcup_{p\in M}T_p\F,\qquad N^\ast \F:=\bigsqcup_{p\in M}\{\xi \in T_p^\ast M\,|\, \xi(T_p \F)=0\}.
\]
\end{definition}
More generally than in Definition \ref{def:foliation}, given a $C^{\infty,0}$-foliation $\F$ of $M$ and any open subset $U\subset M$ which is not necessarily equal to, but only \emph{contained} in the domain $W$ of some foliating chart for $\F$, we shall refer to the connected components of the intersections of $U$ with the leaves of $\F$ as \emph{plaques}. Then every point $p\in U$ lies in a unique plaque of $\F$ in $U$, and every plaque is an embedded submanifold of $M$ (being a connected component of the intersection of $U$ with a plaque of $W$).

Property (4) in Definition \ref{def:foliation} ensures that the vector bundles $T\F$ and $N^\ast \F$ are continuous, so in particular $\dot N^\ast \F\subset \dot T^\ast M$ is a closed conical subset. It is easy to show that $\dot N^\ast \F$  is atomizable (provided that it is non-empty):
\begin{lemma}\label{lem:foliations}Let $\mathcal F$ be a $C^{\infty,0}$-foliation of $M$ with leaves of codimension $k>0$. Then $\Lambda:=\dot N^\ast \mathcal F$ is positively  $H^{-k/2-0}$-atomizable. 
\end{lemma}
\begin{proof}We carry out a simplified version of the construction from the proof of Proposition \ref{prop:existence1}: Let $\{{\mathscr U}_j\}_{j\in J}$ be a basis of the topology of $M$ such that each $\mathscr U_j$ is contained in a foliating chart domain for $\F$. For a given $j$, choose some $p_j\in {\mathscr U}_j$ and let $S_j\subset {\mathscr U}_j$ be the plaque of $\F$ through $p_j$ in ${\mathscr U}_j$. Take a non-zero non-negative cutoff $\chi_j\in C^\infty_c(S_j,\R)$, fix a  smooth measure $ds$ on $S_j$, and define $A_j\in \E'(M)$ by
\[
A_j(\varphi):=\int_{S_j} \varphi(s) \chi_j(s)\,ds,\qquad \varphi\in C^\infty_c(M).
\]
Then $A_j$ is positive, $\supp  A_j\subset S_j\cap {\mathscr U}_j$ has empty interior in $M$ because $S_j$ has positive codimension (being contained in a leaf of $\F$), $\mathrm{WF}(A_j)\subset \dot N^\ast S_j\subset\Lambda$, and $A_j\in H^{-k/2-0}_\mathrm{comp}(M)$ by Example \ref{ex:codimSobolev}.  Any open subset of $M$ contains some ${\mathscr U}_j$, hence also $\supp  A_j$, while $A_j\neq 0$. Thus $\cA:=\{A_j\,|\, j\in J\}$ atomizes $\Lambda$.
\end{proof}

Continuous foliations with smooth leaves occur naturally in the field of  hyperbolic dynamical systems:
\begin{example}\label{ex:Anosov}Let $M$ be a compact smooth manifold.
\begin{enumerate}
\item Let $f:M\to M$ be an Anosov diffeomorphism \cite{BrinStuck,KatokHasselblatt}. Then $M$ possesses two $C^{\infty,0}$-foliations $\mathcal W^s$, $\mathcal W^u$ called \emph{stable} and \emph{unstable} foliations of $f$, which are Hölder-regular and transverse: Their conormal bundles satisfy
\[
T^\ast M=N^\ast \mathcal W^s\oplus N^\ast \mathcal W^u.
\]
\item Let $\varphi_t:M\to M$ be an Anosov flow \cite{Ano67,KatokHasselblatt, dyatlovnotes}. Then $M$ possesses   two $C^{\infty,0}$-foliations $\mathcal W^s$, $\mathcal W^u$ called \emph{stable} and \emph{unstable} foliations of $\varphi_t$, respectively, which are Hölder-regular. Furthermore, $M$ possesses two $C^{\infty,0}$-foliations $\mathcal W^{\mathrm{cs}}$, $\mathcal W^{\mathrm{cu}}$ called \emph{weak} (or \emph{center-}) \emph{stable} and \emph{weak} (or \emph{center-}) \emph{unstable}  foliations of $\varphi_t$, respectively, which are also Hölder-regular. Their conormal bundles satisfy $\dot N^\ast \mathcal W^{\mathrm{cs}}\cap \dot N^\ast \mathcal W^{\mathrm{cu}}=\emptyset$ and
\[
T^\ast M=N^\ast \mathcal W^s\oplus N^\ast \mathcal W^{\mathrm{cu}}=N^\ast \mathcal W^{\mathrm{cs}}\oplus N^\ast \mathcal W^{u}.
\]
\end{enumerate}
\end{example}

\begin{corollary}\label{cor:foliations1}Let $\F$ be a $C^{\infty,0}$-foliation of $M$ with leaves of codimension $k>0$ and let $\Gamma\subset\dot T^\ast M$ be a  closed conical set with $\Gamma\cap \dot N^\ast \F=\emptyset$ and $\Gamma_p\neq \emptyset$ for all $p\in M$. Then there exist distributions
\[
u\in\D'_{\Gamma}(M)\cap H^{\dim M/2}_\mathrm{loc}(M),\qquad v\in\D'_{\dot N^\ast \F}(M)\cap H^{-k/2-0}_\mathrm{loc}(M)
\]
such that
\[
\supp  u=\supp  v=M, \qquad uv=0.\vspace*{0.5em}
\]
Moreover, if $\Gamma_p\cap (-\Gamma_p)\neq \emptyset$ for all $p\in M$, then $u$ and $v$ may be chosen positive.
\end{corollary} 
\begin{proof}Apply Lemma \ref{lem:foliations} and Theorem \ref{thm:main}(1),(5).
\end{proof}
It is worth specializing this result further to the case where $\Lambda$ is a vector bundle transverse to the conormal bundle of the foliation:
\begin{corollary}\label{cor:foliations2}Let $\F$ be a $C^{\infty,0}$-foliation of $M$ with leaves of dimension $>0$ and codimension $k>0$ and let $E\subset T^\ast M$ be a continuous subbundle such that
\[
E\oplus N^\ast \F=T^\ast M.
\]
Then there exist positive distributions
\[
u\in\D'_{\dot E}(M)\cap H^{\dim M/2}_\mathrm{loc}(M),\qquad v\in\D'_{\dot N^\ast \F}(M)\cap H^{-k/2-0}_\mathrm{loc}(M)
\]
such that
\[
\supp  u=\supp  v=M, \qquad uv=0.\vspace*{0.5em}
\]
\end{corollary} 
\begin{proof}As the leaves of $\F$ are of dimension $>0$, the transversality $E\oplus N^\ast \F=T^\ast M$ implies that $E$ has positive rank, in particular $\pi(\dot E)=M$. Since one also has $E=-E$, it suffices to apply Corollary \ref{cor:foliations1} to $\Gamma:=\dot E$.
\end{proof}
Note that in Corollary \ref{cor:foliations2} it is not required that $E$ be integrable, i.e., $E$ need not be the conormal bundle of a foliation. On the other hand, in cases where two transverse foliations $\F$ and $\mathcal G$ with $N^\ast \F\oplus N^\ast \mathcal G=T^\ast M$ are given, such as the Anosov dynamical systems mentioned in Example \ref{ex:Anosov}, Corollary \ref{cor:foliations2} gives positive distributions 
\[
u\in\D'_{\dot N^\ast \mathcal G}(M)\cap H^{\dim M/2}_\mathrm{loc}(M),\qquad v\in\D'_{\dot N^\ast \F}(M)\cap H^{-k/2-0}_\mathrm{loc}(M)
\]
such that
\[
\supp  u=\supp  v=M, \qquad uv=0.\vspace*{0.5em}
\]
Here one may interchange the roles of $\mathcal F$ and $\mathcal G$ and take $k$ as the codimension of the leaves of  $\mathcal G$ instead of $\mathcal F$.

\subsection{More general families of submanifolds}

Inspecting the proof of Lemma \ref{lem:foliations} shows that we can generalize the result to a much larger class of families of embedded submanifolds of $M$ than plaques of foliations. It suffices to abstractly characterize what is actually needed in the proof:

\begin{definition}\label{def:fragmentation}
A collection $\mathcal Z=\{\Sigma_i\}_{i\in I}$ of embedded submanifolds $\Sigma_i\subset M$ of possibly different dimensions is a \emph{fragmentation} of $M$ if for every open subset $U\subset M$ there is an $i\in I$ such that $\Sigma_i\cap U\neq \emptyset$. The manifolds $\Sigma_i\in \mathcal Z$ are called \emph{fragments}.

Let $n_\mathrm{min}:=\min_{i\in I}\dim \Sigma_i$, $n_\mathrm{max}:=\max_{i\in I}\dim \Sigma_i$. Then the \emph{maximal (resp.\ minimal) codimension} of $\mathcal Z$ is $k_\mathrm{max}:=\dim M-n_\mathrm{min}$ (resp.\ $k_\mathrm{min}:=\dim M-n_\mathrm{max}$).

The \emph{conormal cone} of $\mathcal Z$ is $
\Lambda^\ast \mathcal Z:=\overline{\bigcup_{i\in I}\dot N^\ast \Sigma_i}\subset \dot T^\ast M$, 
with closure  taken in $\dot T^\ast M$.
\end{definition}

Note that different fragments may not only have different dimensions but may also intersect, and we do not impose a restriction on the cardinality of the index set $I$.  Of course, for our purposes, a fragmentation of $M$ is only interesting if one has
\bq
\emptyset\neq \Lambda^\ast_p \mathcal Z\neq \dot T^\ast_p M\quad \forall\; p\in M, \label{eq:nontrivialitycond}
\eq
since otherwise we cannot apply Theorem \ref{thm:main} to the conical set $\Lambda^\ast \mathcal Z$.

\begin{example}\begin{enumerate}[leftmargin=*]\item A trivial fragmentation is given by $\mathcal Z:=\{\{q\}\,|\,q\in Q\}$, where $Q\subset M$ is any dense set and the fragments are the singleton sets formed by the points in $Q$, which are $0$-dimensional embedded submanifolds. In this case, $\Lambda^\ast \mathcal Z=\dot T^\ast M$ is positively  $H^{-\dim M/2-0}$-atomizable by Example \ref{ex:maximal}.
\item  Every $C^{\infty,0}$-foliation $\F$ of  $M$  with leaves of codimension $k$ defines a fragmentation $\mathcal Z$ of $M$ of minimal and maximal codimension $k$ with $\Lambda^\ast \mathcal Z=\dot N^\ast \F$: Choose an atlas of $M$ consisting of foliating charts for $\F$, and let $\mathcal Z$ be the collection of all plaques in all foliating charts. If  $0<k<\dim M$, \eqref{eq:nontrivialitycond} is satisfied.

\item Let $M=\R^3$ and consider the horizontal planes $P_{q}:=\{(s,t,q)\,|\,s,t\in \R\}$, $q\in \Q$,  
intersecting the vertical axis in all rational numbers, and let 
$L_1,\ldots,L_N\subset \R^3$ be a finite family of $1$-dimensional embedded closed submanifolds. Then $\mathcal Z:=\{P_q,L_j\,|\,q\in \Q,1\leq j\leq N\}$ 
is a fragmentation of $\R^3$ of maximal codimension $2$ and minimal codimension $1$ whose conormal cone satisfies \eqref{eq:nontrivialitycond}: 
\[
\Lambda^\ast \mathcal Z=\big\{(x,(0,0,s))\in \R^3\times (\R^3\setminus \{0\})\,|\, s\in \R,s\neq 0\big\}\cup \dot N^\ast L_1\cup \cdots \cup\dot N^\ast L_N.
\]
\end{enumerate}
\end{example}

\begin{corollary}\label{cor:fragments}Let $\mathcal Z$ be a fragmentation of $M$ of minimal codimension $k_\mathrm{min}>0$ and maximal codimension $k_\mathrm{max}$. Then $\Lambda^\ast \mathcal Z$ is positively  $H^{-k_\mathrm{max}/2-0}$-atomizable. 
\end{corollary}
\begin{proof}
Let $\{{\mathscr U}_j\}_{j\in \N}$ be a countable basis of the topology of $M$, and for each $j\in \N$ choose a fragment $\Sigma_j\in \mathcal Z$ with $S_j:=\mathscr U_j\cap \Sigma_j\neq \emptyset$. Then we argue exactly as in the proof of Lemma \ref{lem:foliations}, starting from ``For a given $j$,...'',  to get atoms $A_j\in H^{-k_j/2-0}_\mathrm{comp}(M)$ supported in $S_j$, where $0<k_\mathrm{min}\leq k_j\leq k_\mathrm{max}$ is the codimension of $S_j$. It then suffices to observe that $H^{-k_j/2-0}_\mathrm{comp}(M)\subset H^{-k_\mathrm{max}/2-0}_\mathrm{comp}(M)$.
\end{proof}

The following result generalizes its counterpart for foliations Corollary \ref{cor:foliations1}:
\begin{corollary}\label{cor:fragmentations}Let $\mathcal Z$ be a fragmentation of $M$ of minimal codimension $k_\mathrm{min}>0$ and maximal codimension $k_\mathrm{max}$ and let $\Gamma\subset\dot T^\ast M$ be a  closed conical set with $\Gamma\cap (-\Lambda^\ast \mathcal Z)=\emptyset$ and $\Gamma_p\neq \emptyset$ for all $p\in M$. Then there exist distributions
\[
u\in\D'_{\Gamma}(M)\cap H^{\dim M/2}_\mathrm{loc}(M),\qquad v\in\D'_{\Lambda^\ast \mathcal Z}(M)\cap H^{-k_\mathrm{max}/2-0}_\mathrm{loc}(M)
\]
such that
\[
\supp  u=\supp  v=M, \qquad uv=0.\vspace*{0.5em}
\]
Moreover, if $\Gamma_p\cap (-\Gamma_p)\neq \emptyset$ for all $p\in M$, then $u$ and $v$ may be chosen positive.
\end{corollary} 
\begin{proof}Apply Corollary \ref{cor:fragments} and Theorem \ref{thm:main}(1),(5).
\end{proof}

\subsection{Ray bundles}

Since ray bundles are the smallest possible non-empty wavefront sets projecting onto all of $M$, they deserve special attention. Let $
\alpha:M\to T^\ast M$ 
be a continuous, nowhere-vanishing $1$-form. Then the associated ray bundle is
\[
\Lambda_\alpha:=\{(p,t\alpha(p))\,|\,t>0\}\subset \dot T^\ast M.
\]
Thanks to the continuity of $\alpha$, this is a closed conical set.

 If $\beta$ is a smooth and closed  $1$-form (i.e., $d\beta=0$, where $d$ is the exterior derivative), then $\beta$ is locally exact by the Poincaré-Lemma, and if $\beta$ is nowhere-vanishing,  Corollary \ref{cor:integrable}  implies that $\Lambda_\beta$ is $H^{-1/2-0}$-atomizable. More generally, the Corollary implies the same if $\beta$ is nowhere-vanishing and locally conformally closed, i.e., around every point $p\in M$ there is an open set $U\subset M$ on which the $1$-form $\beta$ is of the shape $\beta|_U=f_U\beta_U$, where $f_U\in \Cinft(U)$ is nowhere-vanishing and $\beta_U:U\to T^\ast U$ is a smooth, closed, nowhere-vanishing $1$-form. Such a $1$-form $\beta$ is called \emph{locally integrable}; an equivalent characterization is that $\beta\wedge d\beta=0$. In this case, the Frobenius theorem implies that the line bundle $\R\beta\subset T^\ast M$ is the conormal bundle of a smooth foliation. 

\begin{corollary}\label{cor:raybundles}Let $\alpha:M\to T^\ast M$ be a continuous, nowhere-vanishing $1$-form and $\beta:M\to T^\ast M$ a smooth, nowhere-vanishing, locally conformally closed $1$-form with $
\Lambda_\alpha\cap (-\Lambda_\beta)=\emptyset$. Then there exist distributions
\[
u\in\D'_{\Lambda_\alpha}(M)\cap H^{\dim M/2}_\mathrm{loc}(M),\qquad v\in\D'_{\Lambda_\beta}(M)\cap H^{-1/2-0}_\mathrm{loc}(M)
\]
such that
\bq
\supp  u=\supp  v=M, \qquad uv=0.\vspace*{0.5em} \label{eq:253902390}
\eq
Moreover, if the line bundles $\mathcal L_\alpha:=\R\alpha,\mathcal L_\beta:=\R \beta\subset T^\ast M$ intersect only in the zero section, then there exist positive distributions
\[
u\in\D'_{\dot {\mathcal L}_\alpha}(M)\cap H^{\dim M/2}_\mathrm{loc}(M),\qquad v\in\D'_{\dot {\mathcal L}_\beta}(M)\cap H^{-1/2-0}_\mathrm{loc}(M)
\]
satisfying \eqref{eq:253902390}.
\end{corollary} 
\begin{proof}
Since  the nowhere-vanishing $1$-forms $\alpha,\beta$ are defined on all of $M$, one has $\pi(\Lambda_\alpha)=\pi(\Lambda_\beta)=M$. We already noted above that $\Lambda_\beta$ is $H^{-1/2-0}$-atomizable by Corollary \ref{cor:integrable}, so  Theorem \ref{thm:main}(1)  gives the first claimed statement. For the claim involving the line bundles, we use that $\mathcal L_\beta$ is the conormal bundle of a smooth codimension-$1$ foliation by the Frobenius theorem. Lemma \ref{lem:foliations} therefore says that $\dot{\mathcal L}_\beta$ is positively $H^{-1/2-0}$-atomizable. Since in addition we have $\dot{\mathcal L}_\alpha=-\dot{\mathcal L}_\alpha$, Theorem \ref{thm:main}(5) applies and the proof is finished.
\end{proof}
In particular, the above corollary applies to any contact manifold $(M,\alpha)$ with a global contact $1$-form $\alpha$ admitting a complementary locally conformally closed $1$-form $\beta$. For example, consider on $M=\R^3$ with coordinates $(x,y,z)$ the contact $1$-form \[
\alpha:=dz-x\,dy
\]
and the exact $1$-form\label{ex:contactexact} 
\[
\beta:=dx.
\]
Then Corollary \ref{cor:raybundles} provides distributions
\[
 u\in\D'_{\R_{>0}\alpha}(\R^3)\cap H^{3/2}_\mathrm{loc}(\R^3),\qquad         v\in\D'_{\R_{>0}\beta}(\R^3)\cap H^{-1/2-0}_\mathrm{loc}(\R^3)
\]
and positive distributions
\[
 \tilde u\in\D'_{\R_\times\alpha}(\R^3)\cap H^{3/2}_\mathrm{loc}(\R^3),\qquad         \tilde v\in\D'_{\R_\times\beta}(\R^3)\cap H^{-1/2-0}_\mathrm{loc}(\R^3)
\]
with 
\[
\supp u=\supp v=\supp \tilde u=\supp \tilde v=\R^3,\qquad  uv=\tilde u\tilde v=0.\vspace*{0.5em}
\]
\subsection{Proof of introductory theorem}\label{sec:introthmproof}Here we quickly prove Theorem \ref{thm:intro4}: In case (1), Lemma \ref{lem:foliations} shows that $\Lambda$ is
positively $H^{-k/2-0}$-atomizable. In case (2), the corresponding conclusion follows from Corollary \ref{cor:fragments}. In case (3), Corollary \ref{cor:integrable}, together with the inheritance observed before Remark \ref{rem:atomiz-nonemptyfiber}, shows that $\Lambda$ is $H^{-1/2-0}$-atomizable. In all three cases, the claims now follow from
Theorem \ref{thm:main}(1),(2). \qed

\subsection{An example with two thin factors}\label{sec:alternative}
If one lowers the wavefront set confinement and  Sobolev regularity requirements, there are  slightly simpler constructions of full-support zero divisor pairs than the ``thick-thin'' construction developed in the proof of the main Theorem \ref{thm:main}. In particular, if one accepts a negative Sobolev regularity of both factors and one only requires confinement of the wavefront sets to closed conical subsets  $\Gamma,\Lambda\subset \dot T^\ast M$ with $\Gamma\cap (-\Lambda)=\emptyset$ which are \emph{both} atomizable, then a ``thin-thin'' construction works. Instead of systematically exploring such constructions, which would lead us outside the scope of this paper, we consider an illustrative example:

Take $M=\R^2$ with coordinates $x_1,x_2$ and put $\Gamma:=\R_\times dx_2,\;\Lambda:=\R_\times dx_1\subset \dot T^\ast M$. Fix a dense sequence of points $\{t_j\}_{j\in \N}\subset \R$, for example an enumeration of $\Q$. Choose a smooth function $h\in \Cinft(\R)$ with $h(0)=0$,  $0<h(t)\leq 1$ for $t\neq 0$, and $\Vert h^{(k)}\Vert_{\infty}<\infty$ for all $k\in \N$, for example $h(t):=\frac{t^2}{t^2+1}$. Now, define for each $n\in \N$ a function $a_n\in \Cinft(\R)$ by $a_n(x):=\prod_{j=1}^n h(x-t_j)$. Then, for each $N\in \N$, we define distributions  $u_N\in H^{-1/2-0}_\mathrm{loc}(\R^2)\cap \D'_\Gamma(\R^2)$, $v_N\in H^{-1/2-0}_\mathrm{loc}(\R^2)\cap \D'_\Lambda(\R^2)$ by
\begin{align*}
u_N=u_N(x_1,x_2)&:=\sum_{n=1}^N 2^{-n}a_n(x_1)\delta(x_2-t_n),\\
v_N=v_N(x_1,x_2)&:=\sum_{n=1}^N 2^{-n}a_n(x_2)\delta(x_1-t_n).
\end{align*}
For all $N_1,N_2\in \N$ one finds that the product $u_{N_1}v_{N_2}\in \D'(\R^2)$ vanishes:
\[
u_{N_1}v_{N_2}=\sum_{n=1}^{N_1}\sum_{m=1}^{N_2} 2^{-n-m}a_n(x_1)a_m(x_2)\delta(x_2-t_n) \delta(x_1-t_m)=0,
\]
since $a_n(t_m)=0$ when $n\geq m$ and $a_m(t_n)=0$ when $n\leq m$. On the other hand, since $a_n$ vanishes at $t_1,\ldots,t_n$ and is strictly positive elsewhere, and $\supp \delta(x-t)=\{t\}$ for all $t,x\in \R$ (with the obvious abuse of notation), we have for all $N\in \N$
\bqn
\supp u_N= \R\times \{t_1,\ldots,t_N\},\qquad \supp v_N= \{t_1,\ldots,t_N\}\times \R.
\eqn
By similar arguments as in the proof of Theorem \ref{thm:main} (with  simplifications similar to those mentioned in the proof of (4) in Step 8 on p.~\pageref{step8}), one can show for every $\eps>0$ that $u_N$ and $v_N$ converge in $H^{-1/2-\eps}_\mathrm{loc}(\R^2)\cap \D'_\Gamma(\R^2)$ and $H^{-1/2-\eps}_\mathrm{loc}(\R^2)\cap \D'_\Lambda(\R^2)$ to positive limits $u\in H^{-1/2-0}_\mathrm{loc}(\R^2)\cap \D'_\Gamma(\R^2)$ and $v\in H^{-1/2-0}_\mathrm{loc}(\R^2)\cap \D'_\Lambda(\R^2)$, respectively, such that
\[
\supp u=\supp v=\R^2,\qquad uv=0.
\]
Here one uses that for all $k,n\in \N$ one has $\Vert a_n^{(k)}\Vert_\infty\leq C_k(1+n)^k$ for some $C_k>0$ independent of $n$, and this growth is compensated by the $2^{-n}$ in the sums defining $u_N$ and $v_N$, leading to convergence with respect to all microlocal and Sobolev seminorms.

Such thin-thin examples lack the additional Sobolev regularity and accumulation properties established in Theorems \ref{thm:intro2} and \ref{thm:intro3}. Nevertheless, they illustrate that full-support zero divisor pairs can be constructed in various ways and exist in abundance.

\printbibliography

\bigskip

\end{document}